\documentclass{scrartcl}
\usepackage{preamble}

\begin{document}

\maketitle

\begin{abstract}
     \small{\textsc{Abstract.} We prove that certain categories of Bousfield localizations are presentable and give explicit regular cardinals $\kappa$ for which they are $\kappa$-presentable. As an application, we show that the category of analytic rings in the sense of Clausen and Scholze is $(2^{\aleph_0})^+$-presentable.}
\end{abstract}

\section{Introduction} 
Many constructions in geometry and homotopy theory replace an object by one satisfying a prescribed condition. Sheafification, for example, associates a sheaf to a presheaf, while a completion functor replaces a module by a complete one. A common feature is that the resulting objects form a full subcategory whose inclusion admits a left adjoint. Such left adjoints are called \emph{reflective localizations}.

In this paper, we study reflective localizations of $\infty$-categories, which we henceforth call \emph{categories}. Explicit constructions in this setting must keep track of higher homotopies and their coherences, making abstract existence criteria especially valuable. For presentable categories, such criteria take a particularly simple form: by the reflection theorem, the inclusion of a full subcategory $\mathsf{D}\hookrightarrow\mathsf{C}$ admits a left adjoint if and only if $\mathsf{D}$ is closed under small limits and sufficiently filtered colimits in $\mathsf{C}$ \cite[Thm.~1.1]{reflectionprinciple}.

Informally, a presentable category is a large category controlled by a small collection of objects. More precisely, a category is \emph{$\kappa$-presentable}, for a regular cardinal $\kappa$, if it admits small colimits and is generated under $\kappa$-filtered colimits by a small subcategory of $\kappa$-compact objects. A category is \emph{presentable} if it is $\kappa$-presentable for some $\kappa$. A reflective localization between presentable categories is traditionally called a \emph{Bousfield localization}.

We ask whether Bousfield localizations themselves form a presentable category when their sources and targets are allowed to vary subject to fixed size bounds. We prove that they do, both with and without compatible symmetric monoidal structures, and give explicit cardinal bounds. We also obtain an accessibility bound for the subcategory of localizations whose right adjoints preserve small colimits. Our main motivating application is the presentability of the category of analytic rings in the theory of condensed mathematics developed by Clausen and Scholze.

\subsection*{Presentability of categories of localizations}

Let $\kappa$ be an uncountable regular cardinal. We denote by $\mathsf{Pr}^{\mathsf{L}}_\kappa$ the category of $\kappa$-presentable categories and functors preserving small colimits and $\kappa$-compact objects. The category $\mathsf{Pr}^{\mathsf{L}}_\kappa$ is itself $\kappa$-presentable \cite[Prop.~2.3]{aoki2025higherpresentablecategorieslimits}, providing an ambient category in which to study localizations together with their sources and targets.

For $\mathsf{V}\in\mathsf{Pr}^{\mathsf{L}}_\kappa$, we denote by
\[\mathsf{Loc}_\kappa(\mathsf{V})\subseteq\mathsf{Fun}\mleft(\Delta^1,(\mathsf{Pr}^{\mathsf{L}}_\kappa)_{\mathsf{V}/}\mright)\]
the full subcategory spanned by Bousfield localizations. Concretely, an object is a localization together with compatible functors from $\mathsf{V}$ to its source and target, all lying in $\mathsf{Pr}^{\mathsf{L}}_\kappa$. Its morphisms are commutative squares compatible with the corresponding structure functors from $\mathsf{V}$.

We also consider a symmetric monoidal variant. The category $\mathsf{CAlg}(\mathsf{Pr}^{\mathsf{L}}_\kappa)$ consists of the $\kappa$-presentable symmetric monoidal categories whose tensor product preserves small colimits separately in each variable, whose tensor unit is $\kappa$-compact, and whose $\kappa$-compact objects are closed under tensor products. For $\mathsf{W}\in\mathsf{CAlg}(\mathsf{Pr}^{\mathsf{L}}_\kappa)$, we define
\[\mathsf{Loc}^\otimes_\kappa(\mathsf{W})\subseteq\mathsf{Fun}\mleft(\Delta^1,\mathsf{CAlg}(\mathsf{Pr}^{\mathsf{L}}_\kappa)_{\mathsf{W}/}\mright)\]
in the same way, now requiring the localizations and the structure functors from $\mathsf{W}$ to be symmetric monoidal.

\begin{introthm}[\ref{walking localization}, \ref{result on monoidal localizations}]\label{thm A}
The categories $\mathsf{Loc}_\kappa(\mathsf{V})$ and $\mathsf{Loc}^\otimes_\kappa(\mathsf{W})$ are $\kappa$-presentable. Moreover, their natural inclusions into the respective arrow categories preserve small colimits and $\kappa$-compact objects, and hence are morphisms in $\mathsf{Pr}^{\mathsf{L}}_\kappa$.
\end{introthm}

The proof relies on a criterion for full faithfulness due to Haine--Ramzi--Steinebrunner \cite[Prop.~1.4]{haine2025fullyfaithfulfunctorspushouts}. Applying this criterion to right adjoints, we construct a functor
\[P_\kappa\colon\mathsf{Fun}(\Delta^1,\mathsf{Pr}^{\mathsf{L}}_\kappa)\to\mathsf{Fun}(\Delta^1\times\Delta^1,\mathsf{Pr}^{\mathsf{L}}_\kappa)\]
with the property that a morphism $F$ in $\mathsf{Pr}^{\mathsf{L}}_\kappa$ is a Bousfield localization if and only if the commutative square $P_\kappa(F)$ is a pushout.

Thus $\mathsf{Loc}_\kappa(\ast)$ is the fiber product of $P_\kappa$ and the functor sending each span to its pushout square. We prove that both functors lie in $\mathsf{Pr}^{\mathsf{L}}_\kappa$, so this pullback gives the asserted presentability and preservation properties. Passing to coslice categories gives the general case; compatibility with tensor products gives the symmetric monoidal statement.

\subsection*{Colimit-preserving right adjoints}

For the application to analytic rings, we impose an additional condition on a Bousfield localization: we require its fully faithful right adjoint to preserve small colimits. We denote the corresponding full subcategories by
\[\overline{\mathsf{Loc}}_\kappa(\mathsf{V})\subseteq\mathsf{Loc}_\kappa(\mathsf{V}),\qquad\overline{\mathsf{Loc}}^\otimes_\kappa(\mathsf{W})\subseteq\mathsf{Loc}^\otimes_\kappa(\mathsf{W}).\]
In the symmetric monoidal setting, this condition is weaker than the familiar notion of a \emph{smashing localization} \cite[Thm.~7.7]{reflectionprinciple}. The distinction is essential for our application: the completion functor associated with an analytic ring, introduced below, is smashing if and only if it is an equivalence. 

Recall that, for a regular cardinal $\mu$, a category is \emph{$\mu$-accessible} if it satisfies the definition of $\mu$-presentability above, except that it is not required to admit all small colimits. The following theorem shows that the subcategories introduced above are accessible and provides an explicit bound on their accessibility rank.

\begin{introthm}[\ref{main categorical theorem}, \ref{main theorem on monoidal localizations}]\label{thm B}
Let $\mu>\kappa$ be a regular cardinal satisfying $\operatorname{ecf}(\mu)\geq\kappa$, and suppose that $\mathsf{V}$ and $\mathsf{W}$ are $\mu$-compact. Then $\overline{\mathsf{Loc}}_\kappa(\mathsf{V})$ and $\overline{\mathsf{Loc}}^\otimes_\kappa(\mathsf{W})$ are $\mu$-accessible. Moreover, the natural inclusions
\[\overline{\mathsf{Loc}}_\kappa(\mathsf{V})\hookrightarrow \mathsf{Loc}_\kappa(\mathsf{V}) \quad\text{and}\quad \overline{\mathsf{Loc}}^\otimes_\kappa(\mathsf{W})\hookrightarrow \mathsf{Loc}^\otimes_\kappa(\mathsf{W})\]
preserve $\mu$-compact objects and $\kappa$-filtered colimits.
\end{introthm}

Here $\operatorname{ecf}(\mu)$ denotes the \emph{exponential cofinality} of $\mu$; see Remark~\ref{exponential cofinality}. The smallest regular cardinal $\mu>\kappa$ satisfying this condition is $(2^{<\kappa})^+$. Although the proof is more involved than that of the preceding theorem, it follows the same general strategy. We construct a functor
\[Q_\kappa\colon\mathsf{Fun}(\Delta^1,\mathsf{Pr}^{\mathsf{L}}_\kappa)\to\mathsf{Fun}(\Delta^1,\mathsf{Ani}),\]
where $\mathsf{Ani}$ denotes the category of \emph{animae} (also known as \emph{$\infty$-groupoids} or \emph{spaces}), such that $Q_\kappa(F)$ is an equivalence if and only if the right adjoint of $F$ preserves small colimits.

Restricting $Q_\kappa$ to $\mathsf{Loc}_\kappa(\ast)$, we obtain $\overline{\mathsf{Loc}}_\kappa(\ast)$ as the pullback of $\mathsf{Loc}_{\kappa}(\ast)\to\mathsf{Fun}(\Delta^{1},\mathsf{Ani})$ along the inclusion of the full subcategory of equivalences into $\mathsf{Fun}(\Delta^{1},\mathsf{Ani})$. We prove that $Q_\kappa$ preserves $\kappa$-filtered colimits and $\mu$-compact objects, so this pullback yields the stated accessibility bound and preservation properties. As above, coslice categories and compatibility with tensor products give the general and symmetric monoidal cases.

\subsection*{Applications to condensed mathematics}

Condensed mathematics provides a setting for studying algebraic objects equipped with topologies using the tools of homological algebra. Analytic rings enrich this framework with a notion of completeness for modules, allowing constructions such as tensor products to be followed by completion. Throughout this introduction, rings and abelian groups are understood to be animated, and we leave the adjective \emph{light} implicit. We write $\mathsf{Ring}$ and $\mathsf{Ab}$ for the corresponding categories, and $\mathsf{Cond}(\mathsf{Ring})$ and $\mathsf{Cond}(\mathsf{Ab})$ for their categories of condensed objects.

For a condensed ring $A^\triangleright$, we write $\mathsf{D}(A^\triangleright)$ for its derived category of condensed modules. A \emph{pre-analytic ring} is a pair $A=(A^\triangleright,\mathsf{D}(A))$ consisting of a condensed ring $A^\triangleright$ and a full subcategory $\mathsf{D}(A)$ of $\mathsf{D}(A^\triangleright)$. We think of its objects as the complete $A^\triangleright$-modules, and require the following conditions:
\begin{enumerate}[label=(\roman*)]
    \item The subcategory $\mathsf{D}(A)$ is closed under small limits and colimits. We denote its inclusion by $\iota_A$ and the left adjoint of $\iota_A$ by $-\otimes_{A^\triangleright}A$.
    \item For every $M\in\mathsf{D}(A^\triangleright)$ and $N\in\mathsf{D}(A)$, the object $\underline{\operatorname{Hom}}(M,N)$ lies in $\mathsf{D}(A)$.
    \item The endofunctor $L_A\colonequals\iota_A\circ(-\otimes_{A^\triangleright}A)$ preserves connective objects.
\end{enumerate}

By the reflection theorem of Ragimov--Schlank \cite[Thm.~1.1]{reflectionprinciple}, the left adjoint in (i) exists and $\mathsf{D}(A)$ is presentable. Thus
\[-\otimes_{A^\triangleright}A\colon\mathsf{D}(A^\triangleright)\to\mathsf{D}(A)\]
is a Bousfield localization, which we regard as a completion functor. By (i), its fully faithful right adjoint $\iota_A$ preserves small colimits, so this localization belongs to the class considered in the preceding subsection.

A morphism of pre-analytic rings $A\to B$ is a morphism of condensed rings $\varphi\colon A^\triangleright\to B^\triangleright$ such that the restriction-of-scalars functor
\[\varphi_*\colon\mathsf{D}(B^\triangleright)\to\mathsf{D}(A^\triangleright)\]
maps $\mathsf{D}(B)$ into $\mathsf{D}(A)$. We denote by $\mathsf{PAnRing}$ the category of pre-analytic rings. An \emph{analytic ring} is a pre-analytic ring $A$ for which $A^\triangleright$ lies in $\mathsf{D}(A)$, so that the underlying ring is itself complete. We denote by $\mathsf{AnRing}\subseteq\mathsf{PAnRing}$ the full subcategory consisting of analytic rings.

\begin{introthm}[\ref{PAnRing is presentable}, \ref{AnRing is presentable}]
The categories $\mathsf{PAnRing}$ and $\mathsf{AnRing}$ are $\mu$-presentable, for $\mu=(2^{\aleph_0})^+$.
\end{introthm}

The $\mu$-presentability of analytic rings was announced in \cite[Lecture~15]{ClausenScholzeAnalyticStacks}, but, to our knowledge, a proof has not appeared in the literature; with the question being raised as an open point in \cite[\S1.2, Footnote~2]{camargo2026notessolidgeometry}. As an immediate consequence of presentability, we obtain the following completeness statement, for which we are not aware of an independent proof.

\begin{introcor}[\ref{PAnRing is complete}, \ref{AnRing is complete}]
The categories $\mathsf{PAnRing}$ and $\mathsf{AnRing}$ are complete. Moreover, the inclusion $\mathsf{AnRing}\hookrightarrow\mathsf{PAnRing}$ preserves and reflects small limits. 
\end{introcor}

To prove the theorem, we first use condition (iii) to equip
$\mathsf{D}(A)$ with the inherited $t$-structure from $\mathsf{D}(A^\triangleright)$. We then show that a pre-analytic ring is determined by its connective part, while condition~(ii) ensures that completion is compatible with the symmetric monoidal structure. Setting $\mathsf{W}=\mathsf{Cond}(\mathsf{Ab})$, we obtain a cartesian diagram
\[
\begin{tikzcd}[column sep=large]
    {\mathsf{PAnRing}}
        & {\overline{\mathsf{Loc}}_{\aleph_1}^\otimes(\mathsf{W})} \\
    {\mathsf{Cond}(\mathsf{Ring})}
        & {\mathsf{CAlg}(\mathsf{Pr}^{\mathsf{L}}_{\aleph_1})_{\mathsf{W}/},}
    \arrow[from=1-1, to=1-2]
    \arrow[from=1-1, to=2-1]
    \arrow["s", from=1-2, to=2-2]
    \arrow["{\mathsf{D}(-)_{\geq0}}", from=2-1, to=2-2]
\end{tikzcd}
\]
where $s$ sends a localization to its source.

The $\mu$-accessibility of $\mathsf{PAnRing}$ then follows from the general categorical results established in Theorems~\ref{thm A} and \ref{thm B}. Since $\mathsf{PAnRing}$ admits small colimits, it is in fact $\mu$-presentable. Finally, $\mathsf{AnRing}$ is a reflective localization of $\mathsf{PAnRing}$ whose inclusion preserves filtered colimits; consequently, $\mathsf{AnRing}$ is also $\mu$-presentable.

\subsection*{Notation and conventions}

In addition to the notation introduced above, we use the following conventions.
\begin{itemize}[leftmargin=*, itemsep=0.5\baselineskip, topsep=0.5\baselineskip, parsep=0pt, before={\interlinepenalty=10000}]
    \item \emph{Cardinals.} For a cardinal $\kappa$, we write $2^{<\kappa}=\sup_{\tau<\kappa}2^\tau$, where $\tau$ ranges over cardinals, and denote by $\kappa^+$ the least cardinal strictly greater than $\kappa$. The cardinality of a set $X$ is denoted by $\# X$.

    \item \emph{Categories and animae.} We denote the categories of small categories and small animae by $\mathsf{Cat}$ and $\mathsf{Ani}$, respectively, and their large counterparts by $\widehat{\mathsf{Cat}}$ and $\widehat{\mathsf{Ani}}$.

    \item \emph{Compact objects.} For a regular cardinal $\kappa$, we denote by $\mathsf{C}^\kappa\subseteq\mathsf{C}$ the full subcategory of $\kappa$-compact objects: those objects $X$ for which $\operatorname{Map}_{\mathsf{C}}(X,-)$ preserves $\kappa$-filtered colimits.

    \item \emph{Presentable categories.} We write $\mathsf{Pr}^{\mathsf{L}}$ and $\mathsf{Pr}^{\mathsf{R}}$ for the categories of presentable categories and left and right adjoints, respectively. For regular $\kappa$, the variants $\mathsf{Pr}^{\mathsf{L}}_\kappa$ and $\mathsf{Pr}^{\mathsf{R}}_\kappa$ have $\kappa$-presentable categories as objects. Morphisms in the former preserve small colimits and $\kappa$-compact objects; morphisms in the latter preserve small limits and $\kappa$-filtered colimits. Their full subcategories spanned by stable categories are denoted by $\mathsf{Pr}^{\mathsf{L}}_{\kappa,\mathrm{st}}$ and $\mathsf{Pr}^{\mathsf{R}}_{\kappa,\mathrm{st}}$.

    \item \emph{Accessible categories.} We write $\mathsf{Acc}$ for the category of accessible categories and accessible functors, namely functors preserving $\lambda$-filtered colimits for some regular cardinal $\lambda$. For regular $\kappa$, the category $\mathsf{Acc}_\kappa$ has $\kappa$-accessible categories as objects and functors preserving both $\kappa$-filtered colimits and $\kappa$-compact objects as morphisms.

    \item \emph{Groupoid completion and cores.} For a category $\mathsf{C}$, we denote by $\mleft|\mathsf{C}\mright|$ its groupoid completion, obtained by inverting all morphisms, and by $\mathsf{C}^\simeq$ its core, obtained by retaining only equivalences. These constructions are left and right adjoint, respectively, to the inclusion $\mathsf{Ani}\hookrightarrow\mathsf{Cat}$. We also write $\mleft|S_\ast\mright|$ for the geometric realization of a simplicial object $S_\ast$.
\end{itemize}

\vspace{0.5em}

\textsc{Acknowledgements.} The idea of studying properties of the category $\mathsf{AnRing}$ of analytic rings, including its presentability, through a suitable category of Bousfield localizations goes back to Niklas Kipp's doctoral thesis~\cite{Kipp}. We are grateful to him for useful discussions related to this approach.

\vspace{0.5em}

\textsc{Disclosure on AI use.} The authors used AI tools for literature search and to assist with language editing and presentation. The authors take full responsibility for the accuracy and integrity of all mathematical content.

\section{The category of localizations}\label{sec2}
\begin{definition}
Let $\kappa$ be a regular cardinal, and let $\mathsf{V}$ be a $\kappa$-presentable category. We denote by
\[\mathsf{Loc}_\kappa(\mathsf{V})\subseteq\mathsf{Fun}\mleft(\Delta^1,(\mathsf{Pr}^\mathsf{L}_{\kappa})_{\mathsf{V}/}\mright)\]
the full subcategory spanned by the morphisms that are Bousfield localizations. For the terminal category $\ast$, we denote $\mathsf{Loc}_\kappa(\ast)$ simply as $\mathsf{Loc}_\kappa$.
\end{definition}

\begin{theorem}\label{walking localization}
    Let $\kappa\geq \aleph_1$ be a regular cardinal, and let $\mathsf{V}$ be a $\kappa$-presentable category. Then the natural inclusion
    \[\mathsf{Loc}_\kappa(\mathsf{V})\hookrightarrow \mathsf{Fun}\mleft(\Delta^1,(\mathsf{Pr}^\mathsf{L}_{\kappa})_{\mathsf{V}/}\mright)\]
    is a morphism in $\mathsf{Pr}^\mathsf{L}_{\kappa}$. In particular, the category $\mathsf{Loc}_\kappa(\mathsf{V})$ is $\kappa$-presentable.
\end{theorem}

The proof of this result requires some preparation. Let $i\colon \partial\Delta^1\hookrightarrow \Delta^1$ denote the natural inclusion. If $\mathsf{C}$ is a category admitting binary coproducts, then the restriction functor $i^*\colon \mathsf{Fun}(\Delta^1,\mathsf{C})\to \mathsf{Fun}(\partial\Delta^1,\mathsf{C})\simeq \mathsf{C}\times \mathsf{C}$ admits a left adjoint
\begin{align*}
    i_\sharp\colon \mathsf{C}\times \mathsf{C} &\to \mathsf{Fun}(\Delta^1,\mathsf{C})\\
    (X,Y) &\mapsto (X \to X\sqcup Y).
\end{align*}
Moreover, for every morphism $F\colon \mathsf{C}\to \mathsf{D}$ in $\mathsf{Pr}^\mathsf{L}_\kappa$, the functors $i_\sharp$ fit into a commutative square
    \[\begin{tikzcd}
    	{\mathsf{C}\times \mathsf{C}} && {\mathsf{D}\times \mathsf{D}} \\
    	{\mathsf{Fun}(\Delta^1,\mathsf{C})} && {\mathsf{Fun}(\Delta^1,\mathsf{D}).}
    	\arrow["{F\times F}", from=1-1, to=1-3]
    	\arrow["{i_\sharp}"', from=1-1, to=2-1]
    	\arrow["{i_\sharp}", from=1-3, to=2-3]
    	\arrow["{\mathsf{Fun}(\Delta^1,F)}", from=2-1, to=2-3]
    \end{tikzcd}\]

We now consider the functor
\[P_{\kappa}\colon \mathsf{Fun}(\Delta^1,\mathsf{Pr}^\mathsf{L}_\kappa)\to \mathsf{Fun}(\Delta^1\times \Delta^1,\mathsf{Pr}^\mathsf{L}_\kappa)\]
that sends a morphism $F\colon \mathsf{C}\to \mathsf{D}$ in $\mathsf{Pr}^\mathsf{L}_\kappa$ to the square displayed above.

\begin{lemma}\label{localization_lemma}
Let $\kappa$ be a regular cardinal and let $F$ be a morphism in $\mathsf{Pr}^\mathsf{L}_\kappa$. Then $F$ is a Bousfield localization if and only if the commutative square $P_\kappa(F)$ is a pushout in $\mathsf{Pr}^\mathsf{L}_\kappa$.
\end{lemma}

\begin{proof}
   Let $G\colon \mathsf{D}\to \mathsf{C}$ be the right adjoint of $F$. By the equivalence $(\mathsf{Pr}^\mathsf{L}_\kappa)^\mathrm{op}\simeq \mathsf{Pr}^\mathsf{R}_\kappa$, the square $P_\kappa(F)$ is a pushout in $\mathsf{Pr}^\mathsf{L}_\kappa$ if and only if the square
\[\begin{tikzcd}
	{\mathsf{Fun}(\Delta^1,\mathsf{D})} && {\mathsf{Fun}(\Delta^1,\mathsf{C})} \\
	{\mathsf{D}\times \mathsf{D}} && {\mathsf{C}\times \mathsf{C}}
	\arrow["{\mathsf{Fun}(\Delta^1,G)}", from=1-1, to=1-3]
	\arrow["{i^*}"', from=1-1, to=2-1]
	\arrow["{i^*}", from=1-3, to=2-3]
	\arrow["{G\times G}", from=2-1, to=2-3]
\end{tikzcd}\]
   is a pullback in $\mathsf{Pr}^\mathsf{R}_\kappa$ (or, equivalently, in $\widehat{\mathsf{Cat}}$). This holds precisely if $G$ is fully faithful \cite[Prop.~1.4]{haine2025fullyfaithfulfunctorspushouts}.
\end{proof}

\begin{lemma}\label{Q is nice}
    Let $\kappa$ be a regular cardinal. Then $P_\kappa$ is a morphism in $\mathsf{Pr}^\mathsf{L}_\kappa$.
\end{lemma}

\begin{proof}
According to \cite[Prop.~2.3]{aoki2025higherpresentablecategorieslimits}, $\mathsf{Pr}^\mathsf{L}_\kappa$ is a commutative algebra in $\mathsf{Pr}^\mathsf{L}_\kappa$. Thus, since $\Delta^1$ and $\Delta^1\times \Delta^1$ are finite posets, \cite[Tag~\href{https://kerodon.net/tag/06HY}{06HY}]{kerodon} implies that the source and target of $P_\kappa$ are $\kappa$-presentable, and that $\kappa$-compactness in these categories is checked pointwise. It therefore suffices to show that the endofunctors
\[\mathsf{C}\mapsto \mathsf{Fun}(\partial\Delta^1,\mathsf{C})\simeq \mathsf{C}\times\mathsf{C},\qquad\mathsf{C}\mapsto\mathsf{Fun}(\Delta^1,\mathsf{C})\]
of $\mathsf{Pr}^\mathsf{L}_\kappa$ preserve small colimits and
$\kappa$-compact objects of $\mathsf{Pr}^\mathsf{L}_\kappa$.

Let $K$ be either $\partial\Delta^1$ or $\Delta^1$, and let $\yo\colon K^\mathrm{op}\to \mathsf{P}(K^\mathrm{op})$ denote the Yoneda embedding. Given a small-colimit-preserving functor $F\colon \mathsf{P}(K^\mathrm{op})\to \mathsf{C}$, we claim that $F$ preserves $\kappa$-compact objects if and only if $F\circ \yo$ takes values in $\mathsf{C}^\kappa$. Since colimits in functor categories are computed pointwise, the Yoneda lemma implies that every representable presheaf is compact. Thus, if $F$ preserves $\kappa$-compact objects, then $F\circ \yo$ takes values in $\mathsf{C}^\kappa$, proving one implication.

Conversely, suppose that $F\circ \yo$ takes values in $\mathsf{C}^\kappa$, and let $f$ be a $\kappa$-compact object of $\mathsf{P}(K^\mathrm{op})$. By \cite[Prop.~5.3.4.17]{lurie2009HTT}, $f$ is a retract of a $\kappa$-small colimit of representables $g$. Since $F$ preserves small colimits, $F(g)$ is therefore a $\kappa$-small colimit of objects in $\mathsf{C}^\kappa$. As $\mathsf{C}^\kappa$ is closed under $\kappa$-small colimits in $\mathsf{C}$ \cite[Tag~\href{https://kerodon.net/tag/064W}{064W}]{kerodon}, it follows that $F(g)$ is $\kappa$-compact. Similarly, given that $F(f)$ is a retract of $F(g)$, \cite[Tag~\href{https://kerodon.net/tag/064Z}{064Z}]{kerodon} implies that $F(f)$ is also $\kappa$-compact. This proves the converse and completes the proof of the claim.

The claim above shows that the universal property of the Yoneda embedding \cite[Tag~\href{https://kerodon.net/tag/03W9}{03W9}]{kerodon} restricts to an equivalence
\[\yo^*\colon \operatorname{Map}_{\mathsf{Pr}^\mathsf{L}_\kappa}(\mathsf{P}(K^\mathrm{op}),\mathsf{C})\xrightarrow{\ \sim \ }\operatorname{Map}_{\mathsf{Cat}}(K^\mathrm{op},\mathsf{C}^\kappa).\]
Now $K^\mathrm{op}$ is a compact object of $\mathsf{Cat}$ \cite[Tag~\href{https://kerodon.net/tag/068G}{068G}]{kerodon}, while the functor
\[\mathsf{Pr}^\mathsf{L}_\kappa \to \mathsf{Cat},\qquad \mathsf{C}\mapsto \mathsf{C}^\kappa\]
preserves $\kappa$-filtered colimits \cite[Props.~5.5.7.10-5.5.7.11 and Cor.~4.4.5.21]{lurie2009HTT}. It follows that $\mathsf{P}(K^\mathrm{op})$ is a $\kappa$-compact object of $\mathsf{Pr}^\mathsf{L}_\kappa$.

Since $\mathsf{Pr}^\mathsf{L}_\kappa$ is a commutative algebra
object in $\mathsf{Pr}^\mathsf{L}_\kappa$, tensoring with
$\mathsf{P}(K^\mathrm{op})$ therefore preserves small colimits
and $\kappa$-compact objects.
Finally, the equivalence
\[\mathsf{P}(K^\mathrm{op})\otimes\mathsf{C}\simeq\mathsf{Fun}(K,\mathsf{C}),\]
which is natural in $\mathsf{C}\in\mathsf{Pr}^\mathsf{L}_\kappa$
\cite[Tags~\href{https://kerodon.net/tag/03W9}{03W9}, \href{https://kerodon.net/tag/071C}{071C}, \href{https://kerodon.net/tag/071F}{071F}]{kerodon}, identifies this functor
with $\mathsf{Fun}(K,{-})$. The two choices of $K$ give the
required endofunctors, and hence the result.
\end{proof}

\begin{proof}[Proof of Theorem~\ref{walking localization}]
Let $u\colon \Lambda^2_0\hookrightarrow \Delta^1\times \Delta^1$ be the natural inclusion of the walking span into the walking commutative square. As above, if $\mathsf{C}$ admits pushouts, the restriction functor $u^*\colon \mathsf{Fun}(\Delta^1\times \Delta^1,\mathsf{C})\to \mathsf{Fun}(\Lambda^2_0,\mathsf{C})$ admits a fully faithful left adjoint
\[u_\sharp\colon \mathsf{Fun}(\Lambda^2_0,\mathsf{C}) \to \mathsf{Fun}(\Delta^1\times \Delta^1,\mathsf{C}),\]
that sends a span $X\leftarrow Y\to Z$ to the pushout
\[\begin{tikzcd}[ampersand replacement=\&]
	Y \& Z \\
	X \& {X\sqcup_Y Z.}
	\arrow[from=1-1, to=1-2]
	\arrow[from=1-1, to=2-1]
	\arrow[from=1-2, to=2-2]
	\arrow[from=2-1, to=2-2]
\end{tikzcd}\]

Lemma~\ref{localization_lemma} can then be reformulated as asserting that the commutative square
\[\begin{tikzcd}[ampersand replacement=\&]
	{\mathsf{Loc}_\kappa} \&\& {\mathsf{Fun}(\Lambda^2_0,\mathsf{Pr}^\mathsf{L}_\kappa)} \\
	{\mathsf{Fun}(\Delta^1,\mathsf{Pr}^\mathsf{L}_\kappa)} \&\& {\mathsf{Fun}(\Delta^1\times \Delta^1,\mathsf{Pr}^\mathsf{L}_\kappa)}
	\arrow[from=1-1, to=1-3]
	\arrow[hook, from=1-1, to=2-1]
	\arrow["{u_\sharp}", hook, from=1-3, to=2-3]
	\arrow["P_\kappa", from=2-1, to=2-3]
\end{tikzcd}\]
is cartesian. Given that $\Delta^1$, $\Delta^1\times \Delta^1$, and $\Lambda^2_0$ are finite posets, \cite[Prop.~2.3]{aoki2025higherpresentablecategorieslimits} and \cite[Tag \href{https://kerodon.net/tag/06HY}{06HY}]{kerodon} imply that the categories
\[\mathsf{Fun}(\Delta^1,\mathsf{Pr}^\mathsf{L}_\kappa), \quad \mathsf{Fun}(\Delta^1\times \Delta^1,\mathsf{Pr}^\mathsf{L}_\kappa), \quad \text{and}\quad \mathsf{Fun}(\Lambda^2_0,\mathsf{Pr}^\mathsf{L}_\kappa)\]
are all $\kappa$-presentable. 

Note that, since the functor $u^*$ preserves all small colimits, its left adjoint $u_\sharp$ is a morphism in $\mathsf{Pr}^\mathsf{L}_\kappa$. Moreover, by Lemma~\ref{Q is nice}, so is the functor $P_\kappa$. We conclude from \cite[Prop.~3.2]{aoki2025higherpresentablecategorieslimits} that $\mathsf{Loc}_\kappa$ is $\kappa$-presentable and that the inclusion
\[j\colon \mathsf{Loc}_\kappa\hookrightarrow \mathsf{Fun}(\Delta^1,\mathsf{Pr}^\mathsf{L}_\kappa)\]
is a morphism in $\mathsf{Pr}^\mathsf{L}_\kappa$. 

Let $s\colon \mathsf{Fun}(\Delta^1, \mathsf{Pr}^\mathsf{L}_\kappa)\to \mathsf{Pr}^\mathsf{L}_\kappa$ be the functor that maps a morphism $\mathsf{C}\to \mathsf{D}$ in $\mathsf{Pr}^\mathsf{L}_\kappa$ to $\mathsf{C}$. Since the square
\[
\begin{tikzcd}
	{(\mathsf{Pr}^{\mathsf{L}}_{\kappa})_{\mathsf{V}/}} && {\mathsf{Fun}(\Delta^{1},\mathsf{Pr}^{\mathsf{L}}_{\kappa})} \\
	{*} && {\mathsf{Pr}^{\mathsf{L}}_{\kappa}}
	\arrow[from=1-1, to=1-3]
	\arrow[from=1-1, to=2-1]
	\arrow["s", from=1-3, to=2-3]
	\arrow["{\mathsf{V}}", from=2-1, to=2-3]
\end{tikzcd}\]
is cartesian and the functor $\mathsf{Fun}(\Delta^{1},-)$ preserves limits, we obtain a natural equivalence
\[\mathsf{Fun}\mleft(\Delta^{1},(\mathsf{Pr}^{\mathsf{L}}_{\kappa})_{\mathsf{V}/}\mright)\simeq\{\operatorname{id}_{\mathsf{V}}\}\times_{\mathsf{Fun}(\Delta^{1},\mathsf{Pr}^{\mathsf{L}}_{\kappa})}\mathsf{Fun}(\Delta^{1},\mathsf{Fun}(\Delta^{1},\mathsf{Pr}^{\mathsf{L}}_{\kappa}))\simeq\mathsf{Fun}(\Delta^{1},\mathsf{Pr}^{\mathsf{L}}_{\kappa})_{\operatorname{id}_{\mathsf{V}}/}.\]
Under this equivalence, the full subcategory $\mathsf{Loc}_\kappa(\mathsf{V})$ of $\mathsf{Fun}(\Delta^{1},(\mathsf{Pr}^{\mathsf{L}}_{\kappa})_{\mathsf{V}/})$ corresponds to the full subcategory $(\mathsf{Loc}_{\kappa})_{\operatorname{id}_{\mathsf{V}}/}$ of $\mathsf{Fun}(\Delta^{1},\mathsf{Pr}^{\mathsf{L}}_{\kappa})_{\operatorname{id}_{\mathsf{V}}/}$. The desired result then follows from Proposition~\ref{lemma on coslices}.
\end{proof}

We now study a symmetric monoidal analogue of Theorem~\ref{walking localization}.

\begin{definition}
Let $\kappa$ be a regular cardinal, and let $\mathsf{W}$ be a commutative algebra object in $\mathsf{Pr}^\mathsf{L}_{\kappa}$. We denote by
\[\mathsf{Loc}^\otimes_\kappa(\mathsf{W})\subseteq\mathsf{Fun}\mleft(\Delta^1,\mathsf{CAlg}(\mathsf{Pr}^\mathsf{L}_{\kappa})_{\mathsf{W}/}\mright)\]
the full subcategory spanned by the morphisms that are Bousfield localizations. For $\mathsf{W}=\mathsf{Ani}$, we denote $\mathsf{Loc}^\otimes_\kappa(\mathsf{W})$ simply as $\mathsf{Loc}_\kappa^\otimes$.
\end{definition}

\begin{theorem}\label{result on monoidal localizations}
    Let $\kappa\geq \aleph_1$ be a regular cardinal, and let $\mathsf{W}$ be a commutative algebra object in $\mathsf{Pr}^\mathsf{L}_{\kappa}$. Then the natural inclusion
    \[\mathsf{Loc}^\otimes_\kappa(\mathsf{W})\hookrightarrow\mathsf{Fun}\mleft(\Delta^1,\mathsf{CAlg}(\mathsf{Pr}^\mathsf{L}_{\kappa})_{\mathsf{W}/}\mright)\]
    is a morphism in $\mathsf{Pr}^\mathsf{L}_{\kappa}$. In particular, the category $\mathsf{Loc}^\otimes_\kappa(\mathsf{W})$ is $\kappa$-presentable.
\end{theorem}

The proof uses the pointwise symmetric monoidal structure on the arrow category. We first show that the inclusion of $\mathsf{Loc}_\kappa$ is compatible with this structure.

\begin{lemma}\label{commutative algebras in the functor category}
Let $\kappa\geq \aleph_1$ be a regular cardinal, and endow $\mathsf{Fun}(\Delta^1,\mathsf{Pr}^{\mathsf{L}}_\kappa)$ with the pointwise symmetric monoidal structure. Then the following statements hold.
    \begin{enumerate}
        \item The category $\mathsf{Fun}(\Delta^1,\mathsf{Pr}^{\mathsf{L}}_\kappa)$ is a commutative algebra object in $\mathsf{Pr}^{\mathsf{L}}_\kappa$, and there is a natural equivalence
        \[\mathsf{CAlg}\mleft(\mathsf{Fun}(\Delta^1,\mathsf{Pr}^{\mathsf{L}}_\kappa)\mright)\simeq\mathsf{Fun}\mleft(\Delta^1,\mathsf{CAlg}(\mathsf{Pr}^{\mathsf{L}}_\kappa)\mright)\]
        in $\mathsf{Pr}^{\mathsf{L}}_\kappa$.
        \item The full subcategory $\mathsf{Loc}_\kappa\subseteq \mathsf{Fun}(\Delta^1,\mathsf{Pr}^{\mathsf{L}}_\kappa)$ contains the symmetric monoidal unit and is closed under tensor products. The induced symmetric monoidal structure makes the inclusion
        \[\mathsf{Loc}_\kappa\hookrightarrow\mathsf{Fun}(\Delta^1,\mathsf{Pr}^{\mathsf{L}}_\kappa)\]
        a morphism in $\mathsf{CAlg}(\mathsf{Pr}^{\mathsf{L}}_\kappa)$.
    \end{enumerate}
\end{lemma}

\begin{proof}
According to \cite[Prop.~2.3]{aoki2025higherpresentablecategorieslimits}, $\mathsf{Pr}^{\mathsf{L}}_\kappa$ is a commutative algebra object in $\mathsf{Pr}^{\mathsf{L}}_\kappa$. Since $\Delta^1$ is a finite poset, \cite[Prop.~2.17]{AokiTTG} implies that the same holds for $\mathsf{Fun}(\Delta^1,\mathsf{Pr}^{\mathsf{L}}_\kappa)$ equipped with the pointwise symmetric monoidal structure. Moreover, \cite[Rem.~2.1.3.4]{higheralgebra} gives a natural equivalence
\[\mathsf{CAlg}\mleft(\mathsf{Fun}(\Delta^1,\mathsf{Pr}^{\mathsf{L}}_\kappa)\mright)\simeq\mathsf{Fun}\mleft(\Delta^1,\mathsf{CAlg}(\mathsf{Pr}^{\mathsf{L}}_\kappa)\mright).\]

Recall that, for every commutative algebra object $\mathsf{C}$ in $\mathsf{Pr}^{\mathsf{L}}_\kappa$, the category $\mathsf{CAlg}(\mathsf{C})$ is $\kappa$-presentable. Indeed, it is presentable by \cite[Cor.~3.2.3.5]{higheralgebra}, and the free commutative algebras on $\kappa$-compact generators of $\mathsf{C}$ form a set of $\kappa$-compact generators, since the forgetful functor preserves $\kappa$-filtered colimits \cite[Cor.~3.2.3.2]{higheralgebra}. Together with \cite[Tag~\href{https://kerodon.net/tag/06HY}{06HY}]{kerodon}, this shows that both sides of the displayed equivalence are $\kappa$-presentable.

To prove (2), note that the symmetric monoidal unit of $\mathsf{Fun}(\Delta^1,\mathsf{Pr}^{\mathsf{L}}_\kappa)$ is $\operatorname{id}_{\mathsf{Ani}}$, which is a Bousfield localization. Moreover, tensor products of Bousfield localizations remain Bousfield localizations \cite[Tag~\href{https://kerodon.net/tag/071G}{071G}]{kerodon}. It follows from \cite[Prop.~2.2.1.1 and Rem.~2.2.1.2]{higheralgebra} that the pointwise symmetric monoidal structure restricts to $\mathsf{Loc}_\kappa$, making the inclusion
\[j\colon \mathsf{Loc}_\kappa\hookrightarrow\mathsf{Fun}(\Delta^1,\mathsf{Pr}^{\mathsf{L}}_\kappa)\]
a symmetric monoidal functor.

By Theorem~\ref{walking localization}, applied to $\mathsf{V}=\ast$, the category $\mathsf{Loc}_\kappa$ is $\kappa$-presentable and $j$ preserves small colimits and $\kappa$-compact objects. Since $j$ is fully faithful and preserves small colimits, the tensor product on $\mathsf{Loc}_\kappa$ preserves small colimits separately in each variable. Moreover, full faithfulness and preservation of $\kappa$-filtered colimits imply that $j$ also reflects $\kappa$-compact objects. Thus, the $\kappa$-compactness of the symmetric monoidal unit and the stability of $\kappa$-compact objects under tensor products follow from the corresponding properties in $\mathsf{Fun}(\Delta^1,\mathsf{Pr}^{\mathsf{L}}_\kappa)$. This proves that $j$ is a morphism in $\mathsf{CAlg}(\mathsf{Pr}^{\mathsf{L}}_\kappa)$.
\end{proof}

\begin{proof}[Proof of Theorem~\ref{result on monoidal localizations}.]
By Lemma~\ref{commutative algebras in the functor category}, both $\mathsf{Loc}_\kappa$ and the inclusion $j\colon\mathsf{Loc}_\kappa\hookrightarrow\mathsf{Fun}(\Delta^1,\mathsf{Pr}^{\mathsf{L}}_\kappa)$ lie in $\mathsf{CAlg}(\mathsf{Pr}^\mathsf{L}_\kappa)$. As a result, it induces a fully faithful functor
\[\mathsf{CAlg}(j)\colon \mathsf{CAlg}(\mathsf{Loc}_\kappa)\hookrightarrow\mathsf{CAlg}\mleft(\mathsf{Fun}(\Delta^1,\mathsf{Pr}^{\mathsf{L}}_\kappa)\mright)\simeq\mathsf{Fun}\mleft(\Delta^1,\mathsf{CAlg}(\mathsf{Pr}^{\mathsf{L}}_\kappa)\mright),\]
whose essential image consists of the morphisms whose underlying functors are Bousfield localizations \cite[Rem.~2.2.1.5]{higheralgebra}.

We claim that this functor is a morphism in $\mathsf{Pr}^{\mathsf{L}}_\kappa$. Let $R$ be the right adjoint of $j$. By \cite[Rem.~2.2.1.5]{higheralgebra}, the functor $\mathsf{CAlg}(j)$ admits a right adjoint whose underlying functor is given by $R$. Since $j$ preserves $\kappa$-compact objects, $R$ preserves $\kappa$-filtered colimits. These colimits are computed on underlying objects in categories of commutative algebras \cite[Cor.~3.2.3.2]{higheralgebra}, so the right adjoint of $\mathsf{CAlg}(j)$ also preserves $\kappa$-filtered colimits. Thus, $\mathsf{CAlg}(j)$ preserves small colimits and $\kappa$-compact objects, proving the claim.

As in the proof of Theorem~\ref{walking localization}, there is a natural equivalence
\[\mathsf{Fun}\mleft(\Delta^1,\mathsf{CAlg}(\mathsf{Pr}^{\mathsf{L}}_\kappa)_{\mathsf{W}/}\mright)\simeq\mathsf{Fun}\mleft(\Delta^1,\mathsf{CAlg}(\mathsf{Pr}^{\mathsf{L}}_\kappa)\mright)_{\operatorname{id}_{\mathsf{W}}/}.\]
Since $\operatorname{id}_{\mathsf{W}}$ is a Bousfield localization, it determines an object of $\mathsf{CAlg}(\mathsf{Loc}_\kappa)$. Under the displayed equivalence, the inclusion of $\mathsf{Loc}^\otimes_\kappa(\mathsf{W})$ identifies with the functor
\[\mathsf{CAlg}(\mathsf{Loc}_\kappa)_{\operatorname{id}_{\mathsf{W}}/}\hookrightarrow\mathsf{Fun}\mleft(\Delta^1,\mathsf{CAlg}(\mathsf{Pr}^{\mathsf{L}}_\kappa)\mright)_{\operatorname{id}_{\mathsf{W}}/}\]
induced by $\mathsf{CAlg}(j)$. The desired result then follows from Proposition~\ref{lemma on coslices}.
\end{proof}

\begin{remark}
The same arguments prove variants of Theorems~\ref{walking localization} and \ref{result on monoidal localizations} in which coslice categories are replaced by slice categories. More precisely, under the same hypotheses, the full subcategories of
\[\mathsf{Fun}\mleft(\Delta^1,(\mathsf{Pr}^{\mathsf{L}}_\kappa)_{/\mathsf{V}}\mright)
\qquad\text{and}\qquad
\mathsf{Fun}\mleft(\Delta^1,\mathsf{CAlg}(\mathsf{Pr}^{\mathsf{L}}_\kappa)_{/\mathsf{W}}\mright)\]
spanned by the morphisms that are Bousfield localizations are $\kappa$-presentable, and their natural inclusions are morphisms in $\mathsf{Pr}^{\mathsf{L}}_\kappa$. It suffices to replace the application of Proposition~\ref{lemma on coslices} by \cite[Tags~\href{https://kerodon.net/tag/06HU}{06HU} and \href{https://kerodon.net/tag/06P6}{06P6}]{kerodon}.
\end{remark}

\section{Colimit-preserving right adjoints}\label{sec3}
\begin{definition}
    Let $\kappa$ be a regular cardinal, and let $\mathsf{V}$ be a $\kappa$-presentable category. We denote by
    \[\overline{\mathsf{Loc}}_\kappa(\mathsf{V})\subseteq \mathsf{Loc}_\kappa(\mathsf{V})\]
    the full subcategory spanned by those Bousfield localizations whose fully faithful right adjoints preserve small colimits. For the terminal category $\ast$, we denote $\overline{\mathsf{Loc}}_\kappa(\ast)$ simply as $\overline{\mathsf{Loc}}_\kappa$.
\end{definition}

The main result of this section implies, in particular, that $\overline{\mathsf{Loc}}_\kappa(\mathsf{V})$ is $\mu$-accessible for every $\mu$-compact object $\mathsf{V}$ of $\mathsf{Pr}^\mathsf{L}_\kappa$, provided that $\mu>\kappa\geq\aleph_{1}$ are regular and $\operatorname{ecf}(\mu)\geq\kappa$. Here $\operatorname{ecf}(\mu)$ denotes the \emph{exponential cofinality} of $\mu$. The smallest regular cardinal $\mu>\kappa$ satisfying this condition is $(2^{<\kappa})^+$; see Remark~\ref{exponential cofinality}.

\begin{theorem}\label{main categorical theorem}
Let $\mu>\kappa\geq \aleph_1$ be regular cardinals satisfying $\operatorname{ecf}(\mu)\geq \kappa$. For a $\mu$-compact object $\mathsf{V}$ of $\mathsf{Pr}^\mathsf{L}_\kappa$, the natural inclusion
\[\overline{\mathsf{Loc}}_\kappa(\mathsf{V})\hookrightarrow\mathsf{Loc}_\kappa(\mathsf{V})\]
is a morphism in $\mathsf{Acc}_{\mu}$ and preserves $\kappa$-filtered colimits. In particular, the category $\overline{\mathsf{Loc}}_\kappa(\mathsf{V})$ is $\mu$-accessible.
\end{theorem}

The proof of this result, like that of the main theorem of the previous section, requires some preparation. We begin with a simple observation. Let $F\colon \mathsf{C}\to \mathsf{D}$ be a morphism in $\mathsf{Pr}^\mathsf{L}_\kappa$, with right adjoint $G\colon \mathsf{D}\to \mathsf{C}$. Since $G$ preserves $\kappa$-filtered colimits, it preserves all small colimits if and only if its restriction to $\mathsf{D}^\kappa$ preserves pushouts and $\kappa$-small coproducts \cite[Props.~4.4.2.7 and 5.5.1.9]{lurie2009HTT}. This motivates the following definition.

\begin{definition}\label{I kappa}
    Let $\kappa$ be a regular cardinal. We denote by $\operatorname{I}_\kappa$ a set consisting of the walking span $\Lambda^2_0$ and one discrete category of each cardinality strictly less than $\kappa$.
\end{definition}

We introduce the following notation, which will be convenient for the next lemma. Given a full subcategory $\mathsf{D}\subseteq\mathsf{C}$ and an object $S\in\mathsf{C}$, we set
\[\mathsf{D}_{/S}\colonequals \mathsf{D}\times_{\mathsf{C}}\mathsf{C}_{/S}.\]
Since the projection $\mathsf{C}_{/S}\to\mathsf{C}$ is an isofibration \cite[Tag~\href{https://kerodon.net/tag/0236}{0236}]{kerodon}, this fiber product has the expected concrete description: it is the full subcategory of $\mathsf{C}_{/S}$ spanned by the morphisms $X\to S$ whose source $X$ belongs to $\mathsf{D}$ \cite[Tag~\href{https://kerodon.net/tag/033T}{033T}]{kerodon}.

\begin{lemma}\label{basic lemma on functor categories}
    Let $\kappa$ be a regular cardinal, $\mathsf{C}$ be a $\kappa$-accessible category, and let $\mathsf{K}\in \operatorname{I}_\kappa$. For every diagram $u\colon \mathsf{K}\to \mathsf{C}$, the category
    \[\mathsf{Fun}(\mathsf{K},\mathsf{C}^\kappa)_{/u}\]
    is essentially small and $\kappa$-filtered. Its tautological diagram has colimit $u$ in $\mathsf{Fun}(\mathsf{K},\mathsf{C})$.
\end{lemma}

\begin{proof}
By \cite[Tags~\href{https://kerodon.net/tag/06HY}{06HY} and \href{https://kerodon.net/tag/06J3}{06J3}]{kerodon}, the category $\mathsf{Fun}(\mathsf{K},\mathsf{C})$ is $\kappa$-compactly generated, and its $\kappa$-compact objects are precisely the functors that factor through $\mathsf{C}^\kappa\hookrightarrow\mathsf{C}$. Hence the inclusion
\[\mathsf{Fun}(\mathsf{K},\mathsf{C}^\kappa)\hookrightarrow \mathsf{Fun}(\mathsf{K},\mathsf{C})\]
exhibits $\mathsf{Fun}(\mathsf{K},\mathsf{C})$ as the $\operatorname{Ind}_\kappa$-completion of $\mathsf{Fun}(\mathsf{K},\mathsf{C}^\kappa)$ \cite[Tag~\href{https://kerodon.net/tag/067C}{067C}]{kerodon}. Moreover, since $\mathsf{C}^\kappa$ is essentially small and $\mathsf{K}$ is small, the category $\mathsf{Fun}(\mathsf{K},\mathsf{C}^\kappa)$ is essentially small. The result therefore follows from \cite[Tag~\href{https://kerodon.net/tag/06GA}{06GA}]{kerodon}.
\end{proof}

\begin{lemma}\label{groupoid completion and fibers}
    Let $\mathsf{S}\to B$ be a functor from a small category to a small anima. This functor factors uniquely through the groupoid completion $\mathsf{S}\to |\mathsf{S}|$ and, for every $x\in B$, the natural map
    \[\mleft|\mathsf{S}\times_B\{x\}\mright|\to \mleft|\mathsf{S}\mright|\times_B\{x\}\]
    is an equivalence.
\end{lemma}

\begin{proof}
The existence and uniqueness of the factorization follow from the universal property of groupoid completion $\mleft|-\mright|\colon \mathsf{Cat}\to \mathsf{Ani}$, as the left adjoint to the natural inclusion $\mathsf{Ani}\hookrightarrow\mathsf{Cat}$. The \emph{Rezk nerve}
\[N\colon \mathsf{Cat}\to \mathsf{Fun}(\Delta^\mathsf{op},\mathsf{Ani}),\qquad
\mathsf{S}\mapsto \mleft([n]\mapsto N_n(\mathsf{S})\colonequals\mathsf{Fun}([n],\mathsf{S})^\simeq\mright),\]
is fully faithful, with essential image the complete Segal animae. Under this identification, animae correspond to constant simplicial animae, and groupoid completion is computed by geometric realization:
\[\mleft|\mathsf{S}\mright|\simeq \operatornamewithlimits{colim}_{[n]\in\Delta^\mathsf{op}}N_n(\mathsf{S}).\]
See \cite[\S2]{aaron}.

Since $N_n\simeq\operatorname{Map}_{\mathsf{Cat}}([n],-)$ preserves limits, we have natural equivalences
\[N_n(\mathsf{S}\times_B\{x\})\simeq N_n(\mathsf{S})\times_B\{x\},\]
compatible with the simplicial structure maps. Taking geometric realizations and using the universality of colimits in $\mathsf{Ani}$ \cite[Tag~\href{https://kerodon.net/tag/05V5}{05V5}]{kerodon}, we conclude that the natural map
\[\mleft|\mathsf{S}\times_B\{x\}\mright|\to \mleft|\mathsf{S}\mright|\times_B\{x\}\]
is an equivalence.
\end{proof}

We now construct a functor that detects whether the right adjoint of a morphism in $\mathsf{Pr}^\mathsf{L}_\kappa$ preserves small colimits. In the proof of Theorem~\ref{main categorical theorem}, this functor plays a role analogous to that of $P_\kappa$ in the proof of Theorem~\ref{walking localization}.

\begin{proposition}\label{key proposition}
    Let $\kappa$ be a regular cardinal. There exists a functor
    \[Q_\kappa\colon \mathsf{Fun}(\Delta^1,\mathsf{Pr}^\mathsf{L}_\kappa)\to \mathsf{Fun}(\Delta^1,\mathsf{Ani})\]
    such that $Q_\kappa(F)$ is an equivalence if and only if the right adjoint of $F$ preserves small colimits.
\end{proposition}

\begin{proof}
    Let $F\colon \mathsf{C}\to \mathsf{D}$ be a morphism in $\mathsf{Pr}^\mathsf{L}_\kappa$, with right adjoint $G\colon \mathsf{D}\to \mathsf{C}$. Fix an object $\mathsf{K}$ of $\operatorname{I}_\kappa$ and a diagram $v\colon \mathsf{K}\to \mathsf{D}^\kappa$. Since $\operatorname{Map}_{\mathsf{C}}(c,-)$ are jointly conservative in $\mathsf{C}$ for $c\in\mathsf{C}^\kappa$, the natural map
    \[\varphi\colon \operatornamewithlimits{colim}_{\mathsf{K}} G\circ v\to G\mleft( \operatornamewithlimits{colim}_{\mathsf{K}}v\mright)\]
    is an equivalence if and only if $\operatorname{Map}_{\mathsf{C}}(X,\varphi)$ is an equivalence for all $X\in \mathsf{C}^\kappa$.

    To express these maps in finitary terms and solely in terms of the left adjoint $F$, consider the category $\mathsf{J}_\kappa(F,v)$ defined by the fiber product
\[\begin{tikzcd}[ampersand replacement=\&]
	{\mathsf{J}_\kappa(F,v)} \&\& {\mathsf{Fun}(\mathsf{K},\mathsf{D}^\kappa)_{/v}} \\
	{\mathsf{Fun}(\mathsf{K},\mathsf{C}^\kappa)} \&\& {\mathsf{Fun}(\mathsf{K},\mathsf{D}^\kappa).}
	\arrow[from=1-1, to=1-3]
	\arrow[from=1-1, to=2-1]
	\arrow[from=1-3, to=2-3]
	\arrow["{F\circ -}", from=2-1, to=2-3]
\end{tikzcd}\]
Thus, an object of $\mathsf{J}_\kappa(F,v)$ is a pair $(u,\alpha)$ consisting of a functor $u\colon \mathsf{K}\to\mathsf{C}^\kappa$ and a natural transformation $\alpha\colon F\circ u\to v$. By adjunction, we obtain an equivalence
\[\mathsf{J}_\kappa(F,v)\simeq \mathsf{Fun}(\mathsf{K},\mathsf{C}^\kappa)_{/G\circ v}\]
over $\mathsf{Fun}(\mathsf{K},\mathsf{C}^\kappa)$.

Lemma~\ref{basic lemma on functor categories} consequently shows that $\mathsf{J}_\kappa(F,v)$ is $\kappa$-filtered and that the tautological diagram
\[\Phi_\kappa(F,v)\colon \mathsf{J}_\kappa(F,v)\to\mathsf{Fun}(\mathsf{K},\mathsf{C}^\kappa)\hookrightarrow\mathsf{Fun}(\mathsf{K},\mathsf{C})\]
has colimit $G\circ v$. Since, by \cite[Tag~\href{https://kerodon.net/tag/064X}{064X}]{kerodon}, a $\kappa$-small colimit of $\kappa$-compact objects remains $\kappa$-compact, we conclude that $\operatorname{Map}_{\mathsf{C}}(X,\varphi)$ is a morphism of the form
\begin{equation}\label{the desired equivalence}
\operatornamewithlimits{colim}_{\mathsf{J}_\kappa(F,v)}\operatorname{Map}_{\mathsf{C}^\kappa}\mleft(X,\operatornamewithlimits{colim}_{\mathsf{K}}\Phi_\kappa(F,v)\mright) \to  \operatorname{Map}_{\mathsf{D}^\kappa}\mleft(F(X),\operatornamewithlimits{colim}_{\mathsf{K}}v\mright).
    \tag{$\ast$}
\end{equation}

For each $\mathsf{K}\in \operatorname{I}_\kappa$, denote by $Q^\mathsf{K}_{\kappa}(F)$ the colimit in $\mathsf{Fun}(\Delta^1,\mathsf{Ani})$ of the maps \eqref{the desired equivalence} indexed by the anima
\[B^\mathsf{K}_\kappa(F)\colonequals (\mathsf{C}^\kappa)^\simeq\times \mathsf{Fun}(\mathsf{K},\mathsf{D}^\kappa)^\simeq,\]
and let $Q_{\kappa}(F)$ be the coproduct of the maps $Q^\mathsf{K}_{\kappa}(F)$ over all $\mathsf{K}\in \operatorname{I}_\kappa$. Since small colimits are strongly universal in $\mathsf{Ani}$ \cite[Tag~\href{https://kerodon.net/tag/05VY}{05VY}]{kerodon}, the preceding discussion shows that $Q_{\kappa}(F)$ is an equivalence if and only if the right adjoint $G$ preserves small colimits. In the remainder of the proof, we give an equivalent construction of $Q^\mathsf{K}_\kappa(F)$ that makes its functoriality in $F$ transparent.

\textsc{Step I: Construction of the target of} $Q^\mathsf{K}_\kappa(F)$. Denote by $s$ and $t$ the source and target functors $\mathsf{Fun}(\Delta^1,\mathsf{D}^\kappa)\to \mathsf{D}^\kappa$ and consider the fiber product
\[\begin{tikzcd}[ampersand replacement=\&]
	{T^\mathsf{K}_\kappa(F)} \&\& {\mathsf{Fun}(\Delta^1,\mathsf{D}^\kappa)^\simeq} \\
	{B^\mathsf{K}_\kappa(F)} \&\& {(\mathsf{D}^\kappa)^\simeq\times (\mathsf{D}^\kappa)^\simeq.}
	\arrow[from=1-1, to=1-3]
	\arrow[from=1-1, to=2-1]
	\arrow["{(s,t)}", from=1-3, to=2-3]
	\arrow["{F\times\operatornamewithlimits{colim}_\mathsf{K}}", from=2-1, to=2-3]
\end{tikzcd}\]
Since $(s,t)\colon \mathsf{Fun}(\Delta^1,\mathsf{D}^\kappa)^\simeq\to (\mathsf{D}^\kappa)^\simeq\times (\mathsf{D}^\kappa)^\simeq$ is a Kan fibration \cite[Tag~\href{https://kerodon.net/tag/01F4}{01F4}]{kerodon}, objects of $T^\mathsf{K}_\kappa(F)$ are triples $(X,v,g)$, where
\[X\in\mathsf{C}^\kappa,\qquad v\colon\mathsf{K}\to\mathsf{D}^\kappa,\qquad g\colon F(X)\to\operatornamewithlimits{colim}_{\mathsf{K}}v.\]

The projection $T^\mathsf{K}_\kappa(F)\to B^\mathsf{K}_\kappa(F)$ is itself a Kan fibration, being a base change of the source--target
functor, hence a left fibration. Its fiber over $(X,v)$ is naturally equivalent to
\[\operatorname{Map}_{\mathsf{D}^\kappa}\mleft(F(X),\operatornamewithlimits{colim}_{\mathsf{K}}v\mright).\]
Its straightening is therefore precisely the diagram $B^\mathsf{K}_\kappa(F)\to \mathsf{Ani}$ whose colimit defines the target of $Q^\mathsf{K}_\kappa(F)$. By \cite[Tag~\href{https://kerodon.net/tag/02VE}{02VE}]{kerodon}, this colimit is the groupoid completion $\mleft| T^\mathsf{K}_\kappa(F)\mright|$ of $T^\mathsf{K}_\kappa(F)$. Since $T^\mathsf{K}_\kappa(F)$ is already an anima, we obtain
\[\text{target of }Q^\mathsf{K}_\kappa(F)\simeq T^\mathsf{K}_\kappa(F).\]

\textsc{Step II: Construction of the source of} $Q^\mathsf{K}_\kappa(F)$. Let $i\colon \mathsf{Fun}(\mathsf{K},\mathsf{D}^\kappa)^\simeq\hookrightarrow \mathsf{Fun}(\mathsf{K},\mathsf{D}^\kappa)$ denote the inclusion of the core. As above, denote by $s$ and $t$ the source and target functors $\mathsf{Fun}(\mathsf{K}\times\Delta^1,\mathsf{D}^\kappa)\to \mathsf{Fun}(\mathsf{K},\mathsf{D}^\kappa)$. Define $\mathsf{J}^\mathsf{K}_\kappa(F)$ by the cartesian square
\[\begin{tikzcd}[ampersand replacement=\&]
	{\mathsf{J}^\mathsf{K}_\kappa(F)} \&\& {\mathsf{Fun}(\mathsf{K}\times \Delta^1,\mathsf{D}^\kappa)} \\
	{\mathsf{Fun}(\mathsf{K},\mathsf{C}^\kappa)\times \mathsf{Fun}(\mathsf{K},\mathsf{D}^\kappa)^\simeq} \&\& {\mathsf{Fun}(\mathsf{K},\mathsf{D}^\kappa)\times \mathsf{Fun}(\mathsf{K},\mathsf{D}^\kappa).}
	\arrow[from=1-1, to=1-3]
	\arrow[from=1-1, to=2-1]
	\arrow["{(s,t)}", from=1-3, to=2-3]
	\arrow["{(F\circ -)\times i}", from=2-1, to=2-3]
\end{tikzcd}\]
The source--target functor is an isofibration by \cite[Tag~\href{https://kerodon.net/tag/01F3}{01F3}]{kerodon}, hence the objects of $\mathsf{J}^\mathsf{K}_\kappa(F)$ are triples $(u,v,\alpha)$ consisting of functors
\[u\colon\mathsf{K}\to\mathsf{C}^\kappa,\qquad v\colon\mathsf{K}\to\mathsf{D}^\kappa\]
and a natural transformation $\alpha\colon F\circ u\to v$. Taking the fiber of $\mathsf{J}^\mathsf{K}_\kappa(F)\to \mathsf{Fun}(\mathsf{K},\mathsf{D}^\kappa)^\simeq$ over a functor $v\colon \mathsf{K}\to \mathsf{D}^\kappa$, we recover the category $\mathsf{J}_\kappa(F,v)$.

Again, let $i\colon(\mathsf{C}^\kappa)^\simeq\hookrightarrow\mathsf{C}^\kappa$ denote the inclusion of the core, and write the source and target functors $\mathsf{Fun}(\Delta^1,\mathsf{C}^\kappa)\to\mathsf{C}^\kappa$ as $s$ and $t$, respectively. Define $\mathsf{S}^\mathsf{K}_\kappa(F)$ as the fiber product
\[\begin{tikzcd}[ampersand replacement=\&]
	{\mathsf{S}^\mathsf{K}_\kappa(F)} \&\&\& {\mathsf{Fun}(\Delta^1,\mathsf{C}^\kappa)} \\
	{(\mathsf{C}^\kappa)^\simeq\times \mathsf{J}^\mathsf{K}_\kappa(F)} \& {(\mathsf{C}^\kappa)^\simeq\times \mathsf{Fun}(\mathsf{K},\mathsf{C}^\kappa)} \&\& {\mathsf{C}^\kappa\times\mathsf{C}^\kappa.}
	\arrow[from=1-1, to=1-4]
	\arrow[from=1-1, to=2-1]
	\arrow["{(s,t)}", from=1-4, to=2-4]
	\arrow[from=2-1, to=2-2]
	\arrow["{i\times\operatornamewithlimits{colim}_\mathsf{K}}", from=2-2, to=2-4]
\end{tikzcd}\]
The objects of $\mathsf{S}^\mathsf{K}_\kappa(F)$ are quintuples $(X,u,v,\alpha,f)$, where
\[X\in\mathsf{C}^\kappa,\qquad u\colon\mathsf{K}\to\mathsf{C}^\kappa,\qquad v\colon\mathsf{K}\to\mathsf{D}^\kappa,\qquad \alpha\colon F\circ u\to v,\qquad f\colon X\to\operatornamewithlimits{colim}_{\mathsf{K}}u.\]

By composing the natural projections, we obtain a functor
\[U\colon \mathsf{S}^\mathsf{K}_\kappa(F)\to (\mathsf{C}^\kappa)^\simeq\times \mathsf{J}^\mathsf{K}_\kappa(F)\to (\mathsf{C}^\kappa)^\simeq\times\mathsf{Fun}(\mathsf{K},\mathsf{C}^\kappa)\times \mathsf{Fun}(\mathsf{K},\mathsf{D}^\kappa)^\simeq\to B^\mathsf{K}_\kappa(F),\]
that maps a quintuple $(X,u,v,\alpha,f)$ to the pair $(X,v)$. Given that the fiber of the projection $\mathsf{J}^\mathsf{K}_\kappa(F)\to \mathsf{Fun}(\mathsf{K},\mathsf{D}^\kappa)^\simeq$ at $v$ is naturally equivalent to $\mathsf{J}_\kappa(F,v)$, we obtain that the fiber $\mathsf{S}^\mathsf{K}_\kappa(F)_{(X,v)}$ of $U$ at $(X,v)\in B^\mathsf{K}_\kappa(F)$ fits into the cartesian diagram
\[\begin{tikzcd}[ampersand replacement=\&]
	{\mathsf{S}^\mathsf{K}_\kappa(F)_{(X,v)}} \&\&\& {(\mathsf{C}^\kappa)_{X/}} \\
	{\mathsf{J}_\kappa(F,v)} \& {\mathsf{Fun}(\mathsf{K},\mathsf{C}^\kappa)} \&\& {\mathsf{C}^\kappa.}
	\arrow[from=1-1, to=1-4]
	\arrow[from=1-1, to=2-1]
	\arrow[from=1-4, to=2-4]
	\arrow[from=2-1, to=2-2]
	\arrow["{\operatornamewithlimits{colim}_\mathsf{K}}", from=2-2, to=2-4]
\end{tikzcd}\]

According to \cite[Tag~\href{https://kerodon.net/tag/018K}{018K}]{kerodon}, the natural functor $(\mathsf{C}^\kappa)_{X/}\to \mathsf{C}^\kappa$ is a left fibration; hence so is its pullback $\mathsf{S}^\mathsf{K}_\kappa(F)_{(X,v)}\to \mathsf{J}_\kappa(F,v)$. Its straightening is precisely the functor $\mathsf{J}_\kappa(F,v)\to \mathsf{Ani}$ that maps $(u,\alpha)$ to
\[\operatorname{Map}_{\mathsf{C}^\kappa}\mleft(X,\operatornamewithlimits{colim}_{\mathsf{K}}u\mright).\]
As a result, Lemma~\ref{groupoid completion and fibers} and \cite[Tag~\href{https://kerodon.net/tag/02VE}{02VE}]{kerodon} imply that
\[\mleft|\mathsf{S}^\mathsf{K}_\kappa(F)\mright|_{(X,v)}\simeq \mleft|\mathsf{S}^\mathsf{K}_\kappa(F)_{(X,v)}\mright|\simeq \operatornamewithlimits{colim}_{\mathsf{J}_\kappa(F,v)}\operatorname{Map}_{\mathsf{C}^\kappa}\mleft(X,\operatornamewithlimits{colim}_{\mathsf{K}}\Phi_\kappa(F,v)\mright).\]

Since $B^\mathsf{K}_\kappa(F)$ is an anima, the functor $U$ factors uniquely through the groupoid completion; yielding a functor $\widetilde{U}\colon \mleft|\mathsf{S}^\mathsf{K}_\kappa(F)\mright|\to B^\mathsf{K}_\kappa(F)$. One more application of \cite[Tag~\href{https://kerodon.net/tag/02VE}{02VE}]{kerodon} yields that
\[\text{source of }Q^\mathsf{K}_\kappa(F)\simeq \mleft|\mathsf{S}^\mathsf{K}_\kappa(F)\mright|.\]

\textsc{Step III: Construction of the map} $Q^\mathsf{K}_\kappa(F)\colon \mleft|\mathsf{S}^\mathsf{K}_\kappa(F)\mright|\to T^\mathsf{K}_\kappa(F)$. We begin by constructing a functor $R\colon\mathsf{S}^\mathsf{K}_\kappa(F)\to \mathsf{Fun}(\Delta^1,\mathsf{D}^\kappa)$ sending $(X,f,u,v,\alpha)$ to the composite
\[F(X) \xrightarrow{\ F(f)\ }F\mleft(\operatornamewithlimits{colim}_\mathsf{K} u\mright)\simeq \operatornamewithlimits{colim}_\mathsf{K} F\circ u \xrightarrow{\ \operatornamewithlimits{colim}_\mathsf{K}\alpha\ }\operatornamewithlimits{colim}_\mathsf{K} v.\]
We construct separately the three functors corresponding to the three morphisms in this composite.

In what follows, any unlabeled functor is understood to be a projection from a fiber product. The first functor $R_1\colon \mathsf{S}^\mathsf{K}_\kappa(F)\to \mathsf{Fun}(\Delta^1,\mathsf{D}^\kappa)$ is the composite
\[\mathsf{S}^\mathsf{K}_\kappa(F)\to \mathsf{Fun}(\Delta^1,\mathsf{C}^\kappa)\xrightarrow{\ F\circ -\ }\mathsf{Fun}(\Delta^1,\mathsf{D}^\kappa).\]
By the cartesian square defining $\mathsf{S}^\mathsf{K}_\kappa(F)$, it sends $(X,f,u,v,\alpha)$ to $F(f)\colon F(X)\to F\mleft(\operatornamewithlimits{colim}_\mathsf{K}u\mright)$. Similarly, the third functor $R_3\colon \mathsf{S}^\mathsf{K}_\kappa(F)\to \mathsf{Fun}(\Delta^1,\mathsf{D}^\kappa)$ is the composite
\[\mathsf{S}^\mathsf{K}_\kappa(F)\to \mathsf{J}^\mathsf{K}_\kappa(F)\to \mathsf{Fun}(\mathsf{K}\times \Delta^1,\mathsf{D}^\kappa)\simeq \mathsf{Fun}(\Delta^1,\mathsf{Fun}(\mathsf{K},\mathsf{D}^\kappa))\xrightarrow{\ \operatornamewithlimits{colim}_\mathsf{K}\ }\mathsf{Fun}(\Delta^1,\mathsf{D}^\kappa).\]

It remains to construct the middle morphism. Let $\gamma_{\mathsf{K}}(F)\colon \mathsf{Fun}(\mathsf{K},\mathsf{C}^\kappa)\to \mathsf{Fun}(\Delta^1,\mathsf{D}^\kappa)$ be the functor induced by the natural transformation
\[\operatornamewithlimits{colim}_\mathsf{K} (F\circ -)\to F\mleft(\operatornamewithlimits{colim}_\mathsf{K}-\mright).\]
Since $F$ preserves $\kappa$-small colimits, $\gamma_\mathsf{K}(F)$ takes every diagram $u\colon\mathsf{K}\to\mathsf{C}^\kappa$ to an equivalence. Let $\overline{\gamma}_{\mathsf{K}}(F)\colon \mathsf{Fun}(\mathsf{K},\mathsf{C}^\kappa)\to \mathsf{Fun}(\Delta^1,\mathsf{D}^\kappa)$ be the functor obtained by taking the inverse of this equivalence. We then define $R_2$ as the composite
\[\mathsf{S}^\mathsf{K}_\kappa(F)\to \mathsf{J}^\mathsf{K}_\kappa(F)\to \mathsf{Fun}(\mathsf{K},\mathsf{C}^\kappa)\xrightarrow{\ \overline{\gamma}_{\mathsf{K}}(F) \ } \mathsf{Fun}(\Delta^1,\mathsf{D}^\kappa).\]

Finally, we define $R\colon \mathsf{S}^\mathsf{K}_\kappa(F)\to \mathsf{Fun}(\Delta^1,\mathsf{D}^\kappa)$ as the composite
\[\mathsf{S}^\mathsf{K}_\kappa(F) \xrightarrow{\ (R_1,R_2,R_3)\ } \mathsf{Fun}(\Delta^1,\mathsf{D}^\kappa)\times_{\mathsf{D}^\kappa} \mathsf{Fun}(\Delta^1,\mathsf{D}^\kappa)\times_{\mathsf{D}^\kappa} \mathsf{Fun}(\Delta^1,\mathsf{D}^\kappa)\to \mathsf{Fun}(\Delta^1,\mathsf{D}^\kappa),\]
where the morphism on the right is composition. By definition, a morphism
\[(X,f,u,v,\alpha)\to (X^\prime,f^\prime,u^\prime,v^\prime,\alpha^\prime)\]
in $\mathsf{S}^\mathsf{K}_\kappa(F)$ has components
$X\to X^\prime$ and $v\to v^\prime$ that are equivalences. Hence $R$ carries every morphism to an equivalence in $\mathsf{Fun}(\Delta^1,\mathsf{D}^\kappa)$, so that it factors through $\mathsf{Fun}(\Delta^1,\mathsf{D}^\kappa)^\simeq$.

Using the universal property of pullbacks, the functor $R\colon \mathsf{S}^\mathsf{K}_\kappa(F)\to \mathsf{Fun}(\Delta^1,\mathsf{D}^\kappa)^\simeq$ together with the natural projection $\mathsf{S}^\mathsf{K}_\kappa(F)\to B^\mathsf{K}_\kappa(F)$ give rise to a functor $\mathsf{S}^\mathsf{K}_\kappa(F)\to T^\mathsf{K}_\kappa(F)$. Since $T^\mathsf{K}_\kappa(F)$ is an anima, it induces the desired map
\[Q^\mathsf{K}_\kappa(F)\colon\mleft|\mathsf{S}^\mathsf{K}_\kappa(F)\mright|\to T^\mathsf{K}_\kappa(F).\]
Opening up the constructions above, we see that the map induced on the fibers over $(X,v)\in B^\mathsf{K}_\kappa(F)$ is the map \eqref{the desired equivalence}.
\end{proof}

We next establish the properties of $Q_\kappa$ needed for the accessibility argument. We first prove a compactness bound for the functor $\mathsf{C}\mapsto\mathsf{C}^\kappa$. 

\begin{lemma}\label{compactness bound}
    Let $\mu>\kappa\geq \aleph_1$ be regular cardinals satisfying $\operatorname{ecf}(\mu)\geq \kappa$. Then the functor
    \[\mathsf{Pr}^\mathsf{L}_\kappa\to \mathsf{Cat},\quad \mathsf{C}\mapsto \mathsf{C}^\kappa\]
     preserves $\mu$-compact objects.
\end{lemma}

\begin{proof}
Let $\mathsf{Rex}_\kappa$ denote the subcategory of $\mathsf{Cat}$ whose objects are small categories admitting $\kappa$-small colimits and whose morphisms are functors preserving $\kappa$-small colimits. By \cite[Prop.~5.5.7.8 and Rem.~5.5.7.9]{lurie2009HTT}, the functor $\mathsf{Pr}^\mathsf{L}_\kappa\to \mathsf{Cat}$ appearing in the statement of the lemma factors through an equivalence $\mathsf{Pr}^\mathsf{L}_\kappa\simeq \mathsf{Rex}_\kappa$. It therefore suffices to show that the natural inclusion
\[i\colon \mathsf{Rex}_\kappa\hookrightarrow \mathsf{Cat}\]
preserves $\mu$-compact objects.

By the universal property of presheaf categories \cite[Thm.~5.1.5.6]{lurie2009HTT}, the inclusion $i$ admits a left adjoint $F$, given on objects by $F(\mathsf{C})=\mathsf{P}(\mathsf{C})^\kappa$. Moreover, this adjunction is monadic by \cite[Prop.~2.7]{aoki2025higherpresentablecategorieslimits}, yielding an equivalence between $\mathsf{Rex}_\kappa$ and the category of algebras over the monad $T\colonequals i\circ F$ on $\mathsf{Cat}$. Since $F$ preserves all colimits and $i$ preserves $\kappa$-filtered colimits \cite[Prop.~5.5.7.11]{lurie2009HTT}, the monad $T$ also preserves $\kappa$-filtered colimits.

To match the notation of \cite{monads}, set $\mathsf{A}\colonequals\mathsf{Cat}^\kappa$ and let $\mathsf{Th}_\mathsf{A}(T)\subseteq\mathsf{Rex}_\kappa$ be the essential image of the composite
\[\mathsf{A}=\mathsf{Cat}^\kappa\hookrightarrow\mathsf{Cat}\xrightarrow{\ F\ }\mathsf{Rex}_\kappa.\]
We continue to write $F\colon\mathsf{A}\to\mathsf{Th}_\mathsf{A}(T)$ for the induced functor. By \cite[Thm.~6.7]{monads} and \cite[Lem.~2.4]{aoki2025higherpresentablecategorieslimits}, there is a pullback square
\[\begin{tikzcd}[ampersand replacement=\&]
    {\mathsf{Rex}_\kappa} \&\& {\mathsf{P}(\mathsf{Th}_\mathsf{A}(T))} \\
    {\mathsf{Cat}} \&\& {\mathsf{P}(\mathsf{A}),}
    \arrow[from=1-1, to=1-3]
    \arrow["i"', hook, from=1-1, to=2-1]
    \arrow["{-\circ F^\mathrm{op}}", from=1-3, to=2-3]
    \arrow["\yo", from=2-1, to=2-3]
\end{tikzcd}\]
where $\yo$ denotes the restricted Yoneda embedding.

Next, we claim that both $\yo$ and $-\circ F^\mathrm{op}$ belong to $\mathsf{Acc}_\mu$ and preserve $\kappa$-filtered colimits. Their sources and targets are all $\aleph_0$-presentable: presheaf categories are compactly generated by representables, while $\mathsf{Cat}$ is $\aleph_0$-presentable by \cite[Lem.~2.4]{aoki2025higherpresentablecategorieslimits}. Since colimits in presheaf categories are computed pointwise, precomposition with $F^\mathrm{op}$ preserves all small colimits. On the other hand, since every object of $\mathsf{A}=\mathsf{Cat}^\kappa$ is $\kappa$-compact in $\mathsf{Cat}$, the restricted Yoneda embedding $\yo$ preserves $\kappa$-filtered colimits.

To verify that $\yo$ preserves $\mu$-compact objects, first note that $\mathsf{A}$ is $\mu$-small \cite[Tag~\href{https://kerodon.net/tag/03UW}{03UW}]{kerodon}. Hence, by \cite[Tag~\href{https://kerodon.net/tag/06JY}{06JY}]{kerodon}, a presheaf on $\mathsf{A}$ is $\mu$-compact if and only if it takes values in $\mathsf{Ani}^\mu$. Now let $\mathsf{C}$ be a $\mu$-compact object of $\mathsf{Cat}$. For every $\mathsf{D}\in\mathsf{A}=\mathsf{Cat}^\kappa$, we have
\[\yo(\mathsf{C})(\mathsf{D})=\operatorname{Map}_{\mathsf{Cat}}(\mathsf{D},\mathsf{C})\simeq\mathsf{Fun}(\mathsf{D},\mathsf{C})^\simeq.\]
Since $\mathsf{D}$ is $\kappa$-small and $\mathsf{C}$ is $\mu$-small \cite[Tag~\href{https://kerodon.net/tag/068N}{068N}]{kerodon}, we obtain that this anima is $\mu$-compact by \cite[Tags~\href{https://kerodon.net/tag/03ST}{03ST}, \href{https://kerodon.net/tag/03SV}{03SV}, and \href{https://kerodon.net/tag/0680}{0680}]{kerodon}. Thus $\yo(\mathsf{C})$ takes values in $\mathsf{Ani}^\mu$, proving the claim.

It remains to show that $-\circ F^\mathrm{op}$ preserves $\mu$-compact objects. By \cite[Prop.~5.3.4.17]{lurie2009HTT}, every $\mu$-compact presheaf is a retract of a $\mu$-small colimit of representables. Since $-\circ F^\mathrm{op}$ preserves small colimits, it therefore suffices to verify that it maps representables to $\mu$-compact objects. For $\mathsf{E}\in\mathsf{A}=\mathsf{Cat}^\kappa$, adjunction identifies the image of the presheaf represented by $F(\mathsf{E})$ with
\[\operatorname{Map}_{\mathsf{Rex}_\kappa}(F(-),F(\mathsf{E}))\simeq \operatorname{Map}_{\mathsf{Cat}}(-,T(\mathsf{E}))\simeq \mathsf{Fun}(-,T(\mathsf{E}))^\simeq.\]
On the other hand, another application of \cite[Tag~\href{https://kerodon.net/tag/06JY}{06JY}]{kerodon} gives that
\[T(\mathsf{E})\simeq\mathsf{Fun}(\mathsf{E}^\mathrm{op},\mathsf{Ani}^\kappa).\]
Since $\mathsf{E}$ is $\kappa$-small and $\mathsf{Ani}^\kappa$ is $\mu$-small, the same size estimates used above imply that $T(\mathsf{E})$ is a $\mu$-compact category. Hence the preceding paragraph shows that $\mathsf{Fun}(-,T(\mathsf{E}))^\simeq$ is a $\mu$-compact presheaf, as desired.

Having established that $\yo$ and $-\circ F^\mathrm{op}$ belong to $\mathsf{Acc}_\mu$ and preserve $\kappa$-filtered colimits, we may apply \cite[Tag~\href{https://kerodon.net/tag/06JE}{06JE}]{kerodon} to conclude that the natural inclusion $i\colon \mathsf{Rex}_\kappa\hookrightarrow\mathsf{Cat}$ preserves $\mu$-compact objects.
\end{proof}

\begin{remark}[On exponential cofinality]\label{exponential cofinality}
Let $\kappa,\mu$ be regular cardinals. Then $\mathrm{ecf}(\mu)\geq\kappa$ in the sense of \cite[Tag~\href{https://kerodon.net/tag/03RT}{03RT}]{kerodon} holds if and only if $\mu\gg\kappa$ in the sense of \cite[Def.~A.2.6.3]{lurie2009HTT}. Either condition implies $\mu\unrhd\kappa$ in the sense of \cite[Tag~\href{https://kerodon.net/tag/061K}{061K}]{kerodon}.

Indeed, $\mathrm{ecf}(\mu)\geq\kappa$ means that every product of $\mu$-small sets indexed by a $\kappa$-small set is $\mu$-small \cite[Tag~\href{https://kerodon.net/tag/03RV}{03RV}]{kerodon}. Applying this to constant families gives $\mu_0^{\kappa_0}<\mu$ for all cardinals $\kappa_0<\kappa$ and $\mu_0<\mu$, which is precisely the condition $\mu\gg\kappa$.

Conversely, suppose $\mu\gg\kappa$, and let $\{T_i\}_{i\in I}$ be a family of $\mu$-small sets indexed by a $\kappa$-small set. Writing $\kappa_0\colonequals\#I$, we have $\kappa_0<2^{\kappa_0}<\mu$ by Cantor's theorem \cite[Tag~\href{https://kerodon.net/tag/03R4}{03R4}]{kerodon}. Regularity of $\mu$ therefore gives a cardinal $\mu_0<\mu$, the supremum of the $\# T_i$, such that $\#T_i\leq\mu_0$ for every $i\in I$ \cite[Tag~\href{https://kerodon.net/tag/03RJ}{03RJ}]{kerodon}. Hence
\[\#\mleft(\prod_{i\in I}T_i\mright)\leq\mu_0^{\kappa_0}<\mu,\]
proving $\mathrm{ecf}(\mu)\geq\kappa$. Since either condition forces $\mu\geq \kappa$, the implication $\mu\unrhd\kappa$ follows from \cite[Tag~\href{https://kerodon.net/tag/061P}{061P}]{kerodon}.

The smallest regular cardinal $\mu>\kappa$ satisfying these equivalent conditions is $\mu=(2^{<\kappa})^+$. This cardinal satisfies $\mu\gg\kappa$ by \cite[Prop.~5.4.7]{borceux1994HCA2}. To verify minimality, let $\mu>\kappa$ be any regular cardinal satisfying $\mu\gg\kappa$. Then $2^{\kappa_0}<\mu$ for every $\kappa_0<\kappa$, and regularity of $\mu$ gives
\[2^{<\kappa}=\sup_{\kappa_0<\kappa}2^{\kappa_0}<\mu.\]
Here, we used that $\#\mleft\{\text{cardinals }\kappa_0<\kappa\mright\}\leq \kappa<\mu$, which can be proved by seeing $\kappa$ as the set of \emph{ordinals} $\alpha<\kappa$. Thus $(2^{<\kappa})^+\leq\mu$, as claimed.
\end{remark}

\begin{remark}[Lurie's simplicial argument]
    The proof of \cite[Lem.~4.8.4.2]{higheralgebra} yields an alternative proof of Lemma~\ref{compactness bound}. If $\mathcal{K}$ is a $\kappa$-small collection of $\kappa$-small simplicial sets and $\mu\gg\kappa$, Lurie shows that every object of $\mathsf{Cat}(\mathcal{K})$ is a $\mu$-filtered colimit of $\mu$-small categories in $\mathsf{Cat}(\mathcal{K})$. Consequently, every $\mu$-compact object of $\mathsf{Cat}(\mathcal{K})$ is a retract of a $\mu$-small category, and hence $\mu$-compact in $\mathsf{Cat}$.

    For $\mathcal{K}=I_{\kappa}$ as in Definition~\ref{I kappa}, we have $\mathsf{Cat}(\mathcal{K})=\mathsf{Rex}_{\kappa}$. In this case, the hypothesis that $\mathcal{K}$ is $\kappa$-small holds when $\kappa$ is a successor cardinal, but fails when $\kappa$ is a limit cardinal. However, this hypothesis is used in Lurie's proof only to establish claim $(\ast)$, where it ensures that every $\mu$-small simplicial subset is contained in a $\mu$-small subcategory whose inclusion is a morphism in $\mathsf{Cat}(\mathcal{K})$.
    
    Under our assumptions $\mu>\kappa\geq\aleph_1$ and $\mathrm{ecf}(\mu)\geq\kappa$, the collection of $\mathsf{K}$-shaped diagrams in any $\mu$-small simplicial set is $\mu$-small for every $\kappa$-small simplicial set $\mathsf{K}$ \cite[Tag~\href{https://kerodon.net/tag/03SG}{03SG}]{kerodon}. Moreover, $I_\kappa$ is itself $\mu$-small, since $\#I_\kappa\leq\kappa<\mu$. Consequently, the construction used to establish $(\ast)$ in \emph{loc.~cit.} still involves adjoining only a $\mu$-small simplicial subset at each stage, and the rest of the argument carries over verbatim.
\end{remark}

\begin{proposition}\label{key proposition 2}
    Let $\mu>\kappa\geq \aleph_1$ be regular cardinals satisfying $\operatorname{ecf}(\mu)\geq \kappa$. The functor
    \[Q_\kappa\colon \mathsf{Fun}(\Delta^1,\mathsf{Pr}^\mathsf{L}_\kappa)\to \mathsf{Fun}(\Delta^1,\mathsf{Ani})\]
    constructed in Proposition~\ref{key proposition} preserves $\kappa$-filtered colimits and $\mu$-compact objects. In particular, it is a morphism in $\mathsf{Acc}_\mu$.
\end{proposition}

\begin{proof}
The functor $Q_\kappa$ is the coproduct of the functors $Q_\kappa^\mathsf{K}$, indexed by $\mathsf{K}\in \operatorname{I}_\kappa$. Since colimits commute with colimits, $Q_\kappa$ preserves $\kappa$-filtered colimits if each $Q_\kappa^\mathsf{K}$ does. Likewise, \cite[Tag~\href{https://kerodon.net/tag/064X}{064X}]{kerodon} shows that $Q_\kappa$ preserves $\mu$-compact objects provided that each $Q_\kappa^\mathsf{K}$ does. Colimits in functor categories being computed pointwise, \cite[Tag~\href{https://kerodon.net/tag/06HY}{06HY}]{kerodon} then implies that it is enough to verify that, for every $\mathsf{K}\in \operatorname{I}_\kappa$, the functors
\[\mleft|\mathsf{S}^\mathsf{K}_\kappa\mright|, T^\mathsf{K}_\kappa\colon \mathsf{Fun}(\Delta^1,\mathsf{Pr}^\mathsf{L}_\kappa)\to \mathsf{Ani},\]
constructed in the proof of Proposition~\ref{key proposition}, preserve $\kappa$-filtered colimits and $\mu$-compact objects.

\textsc{Step I: Preservation of} $\kappa$\textsc{-filtered colimits}. Let $F\colon \mathsf{C}\to \mathsf{D}$ be the colimit of a $\kappa$-filtered diagram
\[\mathsf{I}\to \mathsf{Fun}(\Delta^1,\mathsf{Pr}^\mathsf{L}_\kappa),\qquad i\mapsto \mleft[F_i\colon \mathsf{C}_i\to \mathsf{D}_i\mright].\]
Since colimits in functor categories are computed pointwise, we have
\[\mathsf{C}\simeq \operatornamewithlimits{colim}_{i\in\mathsf{I}}\mathsf{C}_i \qquad\text{and}\qquad\mathsf{D}\simeq \operatornamewithlimits{colim}_{i\in\mathsf{I}}\mathsf{D}_i\]
in $\mathsf{Pr}^\mathsf{L}_\kappa$. By \cite[Cor.~4.4.5.21 and Props.~5.5.7.8--5.5.7.11]{lurie2009HTT}, the functor $\mathsf{Pr}^\mathsf{L}_\kappa\to \mathsf{Cat}$, that maps a $\kappa$-presentable category $\mathsf{E}$ to the full subcategory $\mathsf{E}^\kappa$, preserves $\kappa$-filtered colimits. It follows that
\[\mathsf{C}^\kappa\simeq\operatornamewithlimits{colim}_{i\in\mathsf{I}}\mathsf{C}_i^\kappa\qquad\text{and}\qquad\mathsf{D}^\kappa\simeq\operatornamewithlimits{colim}_{i\in\mathsf{I}}\mathsf{D}_i^\kappa,\]
where the colimits on the right-hand side are computed in $\mathsf{Cat}$.

Now let $\mathsf{K}$ be a $\kappa$-small category. In particular, $\mathsf{K}$ may be any element of $\operatorname{I}_\kappa$, the category $\Delta^1$, or any finite product of such categories. By \cite[Tags~\href{https://kerodon.net/tag/068N}{068N} and \href{https://kerodon.net/tag/068G}{068G}]{kerodon}, $\mathsf{K}$ is a $\kappa$-compact object of $\mathsf{Cat}$. Hence the functor
\[\mathsf{Fun}(\mathsf{K},-)^\simeq \simeq \operatorname{Map}_{\mathsf{Cat}}(\mathsf{K},-)\colon \mathsf{Cat}\to \mathsf{Ani}\]
preserves $\kappa$-filtered colimits. Taking $\mathsf{K}=\ast$, we conclude in particular that the functor $(-)^\simeq$ preserves $\kappa$-filtered colimits.

Since filtered colimits are left exact in $\mathsf{Ani}$ \cite[Tag~\href{https://kerodon.net/tag/05Y0}{05Y0}]{kerodon}, it follows that the functor $T^\mathsf{K}_\kappa$ preserves $\kappa$-filtered colimits. The proof that
\[\mathsf{S}^\mathsf{K}_\kappa\colon \mathsf{Fun}(\Delta^1,\mathsf{Pr}^\mathsf{L}_\kappa)\to \mathsf{Cat}\]
preserves $\kappa$-filtered colimits is analogous. Indeed, it suffices to use that filtered colimits are left exact in $\mathsf{Cat}$ and that the functor
\[\mathsf{Fun}(\mathsf{K},-)\colon \mathsf{Cat}\to \mathsf{Cat}\]
preserves $\kappa$-filtered colimits. The first of these facts is proved in \cite[Lem.~4.5.3.1]{lurie2018spectral}.

For the second assertion, let $\mathsf{E}=\operatornamewithlimits{colim}_{i\in \mathsf{I}}\mathsf{E}_i$ be a $\kappa$-filtered colimit in $\mathsf{Cat}$, and consider the natural functor
\[\operatornamewithlimits{colim}_{i\in \mathsf{I}}\mathsf{Fun}(\mathsf{K},\mathsf{E}_i)\to \mathsf{Fun}(\mathsf{K},\mathsf{E}).\]
Since $\mathsf{Cat}$ is compactly generated by $\Delta^1$ \cite[Lem.~2.4]{aoki2025higherpresentablecategorieslimits}, the functor $\operatorname{Map}_\mathsf{Cat}(\Delta^1,-)$ is conservative. It therefore suffices to show that the induced map
\[\begin{tikzcd}[ampersand replacement=\&, row sep=scriptsize]
	{\operatornamewithlimits{colim}_{i\in \mathsf{I}}\operatorname{Map}_\mathsf{Cat}\mleft(\Delta^1,\mathsf{Fun}(\mathsf{K},\mathsf{E}_i)\mright)} \& {\operatorname{Map}_\mathsf{Cat}\mleft(\Delta^1,\operatornamewithlimits{colim}_{i\in \mathsf{I}}\mathsf{Fun}(\mathsf{K},\mathsf{E}_i)\mright)} \\
	\& {\operatorname{Map}_\mathsf{Cat}\mleft(\Delta^1,\mathsf{Fun}(\mathsf{K},\mathsf{E})\mright)}
	\arrow["\sim", from=1-1, to=1-2]
    \arrow[
    from=1-2,
    to=2-2,
    rounded corners,
    to path={
        -- ([xshift=2ex]\tikztostart.east)
        |- ([xshift=-2ex,yshift=1.5em]\tikztotarget.west)
        -- ([xshift=-2ex]\tikztotarget.west)
        -- (\tikztotarget.west)
        \tikztonodes
    }
    ]
\end{tikzcd}\]
is an equivalence.

Since $\mathsf{Cat}$ is cartesian closed, this reduces to the fact that $\Delta^1\times\mathsf{K}$ is $\kappa$-compact in $\mathsf{Cat}$. Finally, we note that the groupoid completion functor 
\[\mleft|-\mright|\colon \mathsf{Cat}\to \mathsf{Ani}\]
also preserves $\kappa$-filtered colimits, as it is left adjoint to the inclusion $\mathsf{Ani}\hookrightarrow\mathsf{Cat}$.

\textsc{Step II: Preservation of} $\mu$\textsc{-compact objects}. Let $F\colon \mathsf{C}\to \mathsf{D}$ be a $\mu$-compact object of $\mathsf{Fun}(\Delta^1,\mathsf{Pr}^\mathsf{L}_\kappa)$. By \cite[Tag~\href{https://kerodon.net/tag/06HY}{06HY}]{kerodon}, both $\mathsf{C}$ and $\mathsf{D}$ are $\mu$-compact objects of $\mathsf{Pr}^\mathsf{L}_\kappa$. Lemma~\ref{compactness bound} therefore implies that
\[\mathsf{C}^\kappa,\mathsf{D}^\kappa\in \mathsf{Cat}^\mu.\]
Equivalently, $\mathsf{C}^\kappa$ and $\mathsf{D}^\kappa$ are $\mu$-small categories \cite[Tag~\href{https://kerodon.net/tag/068N}{068N}]{kerodon}.

As in the first step, let $\mathsf{K}$ be a $\kappa$-small category. Since $\operatorname{ecf}(\mu)\geq \kappa$, the functors
\[\mathsf{Fun}(\mathsf{K},-)\colon \mathsf{Cat}\to \mathsf{Cat}\quad\text{and}\quad\mathsf{Fun}(\mathsf{K},-)^\simeq\colon \mathsf{Cat}\to \mathsf{Ani}\]
preserve $\mu$-compact objects \cite[Tags~\href{https://kerodon.net/tag/03ST}{03ST}, \href{https://kerodon.net/tag/03SV}{03SV}, and \href{https://kerodon.net/tag/0680}{0680}]{kerodon}. In particular, taking $\mathsf{K}=\ast$, we conclude that the functor
\[(-)^\simeq\colon \mathsf{Cat}\to \mathsf{Ani}\]
also preserves $\mu$-compact objects.

By \cite[Tag~\href{https://kerodon.net/tag/03SZ}{03SZ}]{kerodon}, the full subcategory $\mathsf{Cat}^\mu\subseteq \mathsf{Cat}$ is stable under fiber products. Moreover, the groupoid completion functor
\[\mleft|-\mright|\colon \mathsf{Cat}\to \mathsf{Ani}\]
preserves $\mu$-compact objects by \cite[Tags~\href{https://kerodon.net/tag/02ST}{02ST} and \href{https://kerodon.net/tag/05P8}{05P8}]{kerodon}. Since $\mleft|\mathsf{S}^\mathsf{K}_\kappa(F)\mright|$ and $T^\mathsf{K}_\kappa(F)$ are obtained from $F$ by iterating the operations considered above, it follows that both are $\mu$-compact objects of $\mathsf{Ani}$.
\end{proof}

We can now prove Theorem~\ref{main categorical theorem} by combining the criterion of Proposition~\ref{key proposition} with the preservation properties established above.

\begin{proof}[Proof of Theorem~\ref{main categorical theorem}]
From Proposition~\ref{key proposition}, we obtain a cartesian diagram
\[\begin{tikzcd}
	{\overline{\mathsf{Loc}}_\kappa} && {\mathsf{Eq}(\mathsf{Ani})} \\
	{\mathsf{Loc}_\kappa} & {\mathsf{Fun}(\Delta^1,\mathsf{Pr}^\mathsf{L}_\kappa)} & {\mathsf{Fun}(\Delta^1,\mathsf{Ani}),}
	\arrow[from=1-1, to=1-3]
	\arrow[hook, from=1-1, to=2-1]
	\arrow[hook, from=1-3, to=2-3]
	\arrow[hook, from=2-1, to=2-2]
	\arrow["{Q_\kappa}", from=2-2, to=2-3]
\end{tikzcd}\]
where $\mathsf{Eq}(\mathsf{Ani})\subseteq \mathsf{Fun}(\Delta^1,\mathsf{Ani})$ denotes the full subcategory spanned by the equivalences. By Theorem~\ref{walking localization}, the category $\mathsf{Loc}_\kappa$ is $\kappa$-presentable and the inclusion
\[\mathsf{Loc}_\kappa\hookrightarrow \mathsf{Fun}(\Delta^1,\mathsf{Pr}^\mathsf{L}_\kappa)\]
is a morphism in $\mathsf{Pr}^\mathsf{L}_\kappa$. Likewise, Proposition~\ref{key proposition 2} shows that $Q_\kappa$ is a morphism in $\mathsf{Acc}_\mu$ which preserves $\kappa$-filtered colimits.

Consider now the natural inclusion $\mathsf{Eq}(\mathsf{Ani})\hookrightarrow \mathsf{Fun}(\Delta^1,\mathsf{Ani})$. Evaluation at either endpoint induces an equivalence $\mathsf{Eq}(\mathsf{Ani})\simeq \mathsf{Ani}$, under which this inclusion is naturally equivalent to the functor
\[\delta\colon \mathsf{Ani}\hookrightarrow\mathsf{Fun}(\Delta^1,\mathsf{Ani}),\quad X\mapsto (\operatorname{id}\colon X\to X).\]
The category $\mathsf{Ani}$ is $\aleph_0$-presentable \cite[Tag~\href{https://kerodon.net/tag/06KJ}{06KJ}]{kerodon}. Since colimits in functor categories are computed pointwise, the functor $\delta$ preserves small colimits. Moreover, it preserves compact objects by \cite[Tag~\href{https://kerodon.net/tag/06HY}{06HY}]{kerodon}. Thus $\delta$, and therefore the inclusion above, is a morphism in $\mathsf{Pr}^\mathsf{L}_{\aleph_0}$.

Since every morphism in the cartesian diagram above lies in $\mathsf{Acc}_\mu$ and preserves $\kappa$-filtered colimits, \cite[Tags~\href{https://kerodon.net/tag/06LM}{06LM} and \href{https://kerodon.net/tag/06AH}{06AH}]{kerodon} implies that the natural inclusion
\[\overline{\mathsf{Loc}}_\kappa\hookrightarrow\mathsf{Loc}_\kappa\]
is a morphism in $\mathsf{Acc}_\mu$ that preserves $\kappa$-filtered colimits. The desired result now follows by the same argument as in the proof of Theorem~\ref{walking localization}, with Proposition~\ref{lemma on coslices} being replaced by Proposition~\ref{lemma on accessible coslices}.

For the reader's convenience, let us verify the compactness condition needed to apply Proposition~\ref{lemma on accessible coslices}. Since $\mathsf{V}$ is $\mu$-compact in $\mathsf{Pr}^\mathsf{L}_\kappa$, \cite[Tag~\href{https://kerodon.net/tag/06HY}{06HY}]{kerodon} implies that $\operatorname{id}_\mathsf{V}$ is $\mu$-compact in $\mathsf{Fun}(\Delta^1,\mathsf{Pr}^\mathsf{L}_\kappa)$. Using that the inclusions
\[
\overline{\mathsf{Loc}}_\kappa\hookrightarrow \mathsf{Loc}_\kappa\hookrightarrow \mathsf{Fun}(\Delta^1,\mathsf{Pr}^\mathsf{L}_\kappa)
\]
are fully faithful and preserve $\mu$-filtered colimits, we deduce that $\operatorname{id}_\mathsf{V}$ is $\mu$-compact in $\overline{\mathsf{Loc}}_\kappa$, as required.
\end{proof}

We now study a symmetric monoidal analogue of Theorem~\ref{main categorical theorem}.

\begin{definition}\label{barLoc tensor}
Let $\kappa$ be a regular cardinal, and let $\mathsf{W}$ be a commutative algebra object in $\mathsf{Pr}^\mathsf{L}_{\kappa}$. We denote by
\[\overline{\mathsf{Loc}}^\otimes_\kappa(\mathsf{W})\subseteq\mathsf{Loc}^\otimes_\kappa(\mathsf{W})\]
the full subcategory spanned by those Bousfield localizations whose fully faithful right adjoints preserve small colimits. For $\mathsf{W}=\mathsf{Ani}$, we denote $\overline{\mathsf{Loc}}^\otimes_\kappa(\mathsf{W})$ simply as $\overline{\mathsf{Loc}}_\kappa^\otimes$.
\end{definition}

\begin{theorem}\label{main theorem on monoidal localizations}
    Let $\mu>\kappa\geq \aleph_1$ be regular cardinals satisfying $\operatorname{ecf}(\mu)\geq \kappa$. For a $\mu$-compact object $\mathsf{W}$ of $\mathsf{CAlg}(\mathsf{Pr}^\mathsf{L}_\kappa)$, the natural inclusion
    \[\overline{\mathsf{Loc}}^\otimes_\kappa(\mathsf{W})\hookrightarrow\mathsf{Loc}^\otimes_\kappa(\mathsf{W})\]
    is a morphism in $\mathsf{Acc}_{\mu}$ and preserves $\kappa$-filtered colimits. In particular, the category $\overline{\mathsf{Loc}}^\otimes_\kappa(\mathsf{W})$ is $\mu$-accessible.
\end{theorem}

The proof uses the symmetric monoidal structure on $\mathsf{Loc}_\kappa$ constructed in the previous section. We first show that this structure restricts to $\overline{\mathsf{Loc}}_\kappa$.

\begin{lemma}\label{loc bar symmetric monoidal}
Let $\kappa\geq \aleph_1$ be a regular cardinal. Endow $\mathsf{Loc}_\kappa$ with the symmetric monoidal structure of Lemma~\ref{commutative algebras in the functor category}. Then, the full subcategory $\overline{\mathsf{Loc}}_\kappa\subseteq\mathsf{Loc}_\kappa$ contains the symmetric monoidal unit and is closed under tensor products.
\end{lemma}

\begin{proof}
Let $\widehat{\mathdutchcal{Cat}}$ denote the very large 2-category of large categories, and let $\mathdutchcal{Pr}^\mathsf{L}\subseteq\widehat{\mathdutchcal{Cat}}$ be the sub-2-category whose objects are the presentable categories, whose 1-morphisms are the left adjoints, and whose 2-morphisms are arbitrary natural transformations. As in \cite[\S1.2]{basechange}, the Lurie tensor product equips $\mathdutchcal{Pr}^\mathsf{L}$ with the structure of a symmetric monoidal 2-category. In particular, it induces a 2-functor
\[- \otimes -\colon \mathdutchcal{Pr}^\mathsf{L}\times \mathdutchcal{Pr}^\mathsf{L}\to \mathdutchcal{Pr}^\mathsf{L}.\]

Let $F$ and $F^\prime$ be objects of $\overline{\mathsf{Loc}}_\kappa$, with right adjoints $G$ and $G^\prime$, respectively. Since $G$ and $G^\prime$ preserve small colimits, they are themselves 1-morphisms in $\mathdutchcal{Pr}^\mathsf{L}$. Thus the adjunction
\[(F,F^\prime)\dashv(G,G^\prime)\]
is internal to the 2-category $\mathdutchcal{Pr}^\mathsf{L}\times\mathdutchcal{Pr}^\mathsf{L}$. Since 2-functors preserve adjunctions, applying the tensor product gives an adjunction
\[F\otimes F^\prime\dashv G\otimes G^\prime.\]
In particular, $G\otimes G^\prime$ preserves small colimits and so $F\otimes F^\prime$ lies in $\overline{\mathsf{Loc}}_\kappa$. Finally, the symmetric monoidal unit $\operatorname{id}_{\mathsf{Ani}}$ clearly belongs to $\overline{\mathsf{Loc}}_\kappa$.
\end{proof}

\begin{proof}[Proof of Theorem~\ref{main theorem on monoidal localizations}]
By Lemma~\ref{loc bar symmetric monoidal}, the category $\overline{\mathsf{Loc}}_\kappa$ is symmetric monoidal, and the inclusion $\overline{\mathsf{Loc}}_\kappa\hookrightarrow \mathsf{Loc}_\kappa$ is a symmetric monoidal functor. We therefore obtain a cartesian diagram
\[
\begin{tikzcd}[ampersand replacement=\&]
	{\mathsf{CAlg}(\overline{\mathsf{Loc}}_\kappa)} \&\& {\mathsf{CAlg}(\mathsf{Loc}_\kappa)} \\
	{\overline{\mathsf{Loc}}_\kappa} \&\& {\mathsf{Loc}_\kappa,}
	\arrow[hook, from=1-1, to=1-3]
	\arrow[from=1-1, to=2-1]
	\arrow[from=1-3, to=2-3]
	\arrow[hook, from=2-1, to=2-3]
\end{tikzcd}
\]
where the vertical arrows are the forgetful functors. The functor $\mathsf{CAlg}(\mathsf{Loc}_\kappa)\to \mathsf{Loc}_\kappa$ preserves $\kappa$-filtered colimits by \cite[Cor.~3.2.3.2]{higheralgebra} and $\mu$-compact objects by Proposition~\ref{forgetful functor preserves compacts}. Moreover, Theorem~\ref{main categorical theorem} shows that the inclusion $\overline{\mathsf{Loc}}_\kappa\hookrightarrow \mathsf{Loc}_\kappa$ has the same two properties.

Applying \cite[Tags~\href{https://kerodon.net/tag/06LM}{06LM} and \href{https://kerodon.net/tag/06AH}{06AH}]{kerodon} to the cartesian square above, we deduce that
\[\mathsf{CAlg}(\overline{\mathsf{Loc}}_\kappa)\hookrightarrow\mathsf{CAlg}(\mathsf{Loc}_\kappa)\]
lies in $\mathsf{Acc}_\mu$ and preserves $\kappa$-filtered colimits. As observed in the final paragraph of the proof of Theorem~\ref{main categorical theorem}, the object $\operatorname{id}_\mathsf{W}$ is $\mu$-compact in $\mathsf{CAlg}(\overline{\mathsf{Loc}}_\kappa)$. Proposition~\ref{lemma on accessible coslices} therefore implies that the induced functor
\[\mathsf{CAlg}(\overline{\mathsf{Loc}}_\kappa)_{\operatorname{id}_\mathsf{W}/}\hookrightarrow\mathsf{CAlg}(\mathsf{Loc}_\kappa)_{\operatorname{id}_\mathsf{W}/}\]
also lies in $\mathsf{Acc}_\mu$ and preserves $\kappa$-filtered colimits.

Finally, under the canonical equivalence
\[\mathsf{Fun}\mleft(\Delta^1,\mathsf{CAlg}(\mathsf{Pr}^{\mathsf{L}}_\kappa)_{\mathsf{W}/}\mright)\simeq\mathsf{Fun}\mleft(\Delta^1,\mathsf{CAlg}(\mathsf{Pr}^{\mathsf{L}}_\kappa)\mright)_{\operatorname{id}_{\mathsf{W}}/},\]
the latter functor identifies with the desired inclusion $\overline{\mathsf{Loc}}^\otimes_\kappa(\mathsf{W})\hookrightarrow\mathsf{Loc}^\otimes_\kappa(\mathsf{W})$, finishing the proof.
\end{proof}

Combining the main theorems of Sections~\ref{sec2} and~\ref{sec3}, we obtain the following corollary, which is the form of these results that we will use in the rest of the paper.

\begin{corollary}\label{main theorem on monoidal localizations2}
    Let $\mu>\kappa\geq \aleph_1$ be regular cardinals satisfying $\operatorname{ecf}(\mu)\geq \kappa$. For a $\mu$-compact object $\mathsf{W}$ of $\mathsf{CAlg}(\mathsf{Pr}^\mathsf{L}_\kappa)$, the source and target functors
    \[s,t\colon \overline{\mathsf{Loc}}^\otimes_\kappa(\mathsf{W})\to \mathsf{CAlg}(\mathsf{Pr}^\mathsf{L}_{\kappa})_{\mathsf{W}/},\]
    which send a localization $\mathsf{C}\to\mathsf{D}$ to $\mathsf{C}$ and $\mathsf{D}$, respectively, are morphisms in $\mathsf{Acc}_{\mu}$ that preserve $\kappa$-filtered colimits.
\end{corollary}

\begin{proof}
By \cite[Prop.~2.3]{aoki2025higherpresentablecategorieslimits}, Proposition~\ref{lemma on coslices}, and \cite[Tag~\href{https://kerodon.net/tag/06HY}{06HY}]{kerodon}, the source and target functors
\[\mathsf{Fun}\mleft(\Delta^1,\mathsf{CAlg}(\mathsf{Pr}^\mathsf{L}_{\kappa})_{\mathsf{W}/}\mright)\to \mathsf{CAlg}(\mathsf{Pr}^\mathsf{L}_{\kappa})_{\mathsf{W}/}\]
are morphisms in $\mathsf{Pr}^\mathsf{L}_{\kappa}$. The result therefore follows by restricting these functors along the natural inclusions
\[\overline{\mathsf{Loc}}^\otimes_\kappa(\mathsf{W})\hookrightarrow \mathsf{Loc}^\otimes_\kappa(\mathsf{W})\hookrightarrow\mathsf{Fun}\mleft(\Delta^1,\mathsf{CAlg}(\mathsf{Pr}^\mathsf{L}_{\kappa})_{\mathsf{W}/}\mright),\]
which belong to $\mathsf{Acc}_\mu$ (and preserve $\kappa$-filtered colimits) and $\mathsf{Pr}^\mathsf{L}_\kappa$, respectively, by Theorem~\ref{main theorem on monoidal localizations} and Theorem~\ref{result on monoidal localizations}.
\end{proof}

\begin{remark}
The same arguments prove variants of Theorems~\ref{main categorical theorem} and \ref{main theorem on monoidal localizations}, as well as of Corollary~\ref{main theorem on monoidal localizations2}, in which coslice categories are replaced by slice categories. In these variants, the assumptions that $\mathsf{V}$ and $\mathsf{W}$ be $\mu$-compact are not needed. The proofs are the same, except that Proposition~\ref{lemma on accessible coslices} is replaced by \cite[Tags~\href{https://kerodon.net/tag/06HU}{06HU} and \href{https://kerodon.net/tag/06LJ}{06LJ}]{kerodon}.
\end{remark}

\section{Condensed mathematics}
\begin{definition}
    Let $S$ be a profinite set. We define:
    \begin{itemize}
        \item The \emph{size} $\kappa(S)$ as the cardinality of $S$.
        \item The \emph{weight} $\lambda(S)$ as the cardinality of the Boolean ring $\operatorname{Cont}(S,\mathbb{F}_2)$.
    \end{itemize}
    For an uncountable cardinal $\lambda$, we say that $S$ is $\lambda$\emph{-light} if $\lambda(S)< \lambda$. The category of $\lambda$-light profinite sets will be denoted as $\mathsf{Prof}_\lambda$. In particular, for $\lambda=\aleph_1$ (the first uncountable cardinal), we say that $S$ is \emph{light} if $\lambda(S)$ is countable.
\end{definition}

We note that, if $\kappa(S)$ is infinite, $\lambda(S)$ coincides with the cardinality of the smallest index set $I$ for which there exists a presentation
\[S=\lim_{i\in I}S_i,\]
with each $S_i$ finite and discrete. Moreover, if $\lambda$ is an uncountable strong limit cardinal, a profinite set $S$ is $\lambda$-light if and only if $\kappa(S)<\lambda$.

\begin{definition}
    We define a Grothendieck topology on $\mathsf{Prof}_\lambda$ by saying that $\{S_i\to S\}_{i\in I}$ is a covering if $I$ is finite and $\bigsqcup_{i\in I}S_i\to S$ is surjective. For a presentable category $\mathsf{V}$, we denote by $\mathsf{Cond}_\lambda(\mathsf{V})$ the category of $\mathsf{V}$-valued hypersheaves on $\mathsf{Prof}_\lambda$.
\end{definition}

According to \cite[Prop.~A.5.7.2]{lurie2018spectral}, an object of $\mathsf{Cond}_\lambda(\mathsf{V})$ is the same as a functor $X\colon \mathsf{Prof}_\lambda^{\mathrm{op}}\to \mathsf{V}$ satisfying the following conditions:
\begin{enumerate}[label=(\roman*)]
    \item For a collection $\{S_i\}_{i\in I}$ of $\lambda$-light profinite sets indexed by a finite set $I$, the natural map
    \[X\mleft(\coprod_{i\in I}S_i\mright)\to \prod_{i\in I}X(S_i)\]
    is an equivalence;
    \item For a hypercover $S_\ast\to S$ in $\mathsf{Prof}_\lambda$, the natural map
    \[X(S)\to \lim_{n\in \Delta}X(S_n)\]
    is an equivalence.
\end{enumerate}
If $\mathsf{V}$ is an $n$-category for some finite $n$, it suffices to check the second condition above on \v{C}ech nerves \cite[Lem.~6.5.2.9]{lurie2009HTT}. We also note that the category $\mathsf{Cond}_\lambda(\mathsf{V})$ is presentable and equivalent to the tensor product $\mathsf{Cond}_\lambda(\mathsf{Ani})\otimes \mathsf{V}$ \cite[Lem.~A.4.22]{heyer20246functorformalismssmoothrepresentations}.

\begin{lemma}\label{compact condensed anima}
    Let $S$ be a $\lambda$-light profinite set. The representable presheaf $\underline{S}\colon \mathsf{Prof}_\lambda^{\mathrm{op}}\to \mathsf{Set}\subset \mathsf{Ani}$, defined as $\underline{S}(T)=\operatorname{Cont}(T,S)$ is an $\aleph_1$-compact object of $\mathsf{Cond}_\lambda(\mathsf{Ani})$.
\end{lemma}

\begin{proof}
Since the functor $\operatorname{Cont}(-,S)$ takes finite coproducts to products, it is a hypersheaf as soon as the map
\[\operatorname{Cont}(T,S) \to \operatorname{eq}\mleft(\operatorname{Cont}(T^\prime,S)\rightrightarrows \operatorname{Cont}(T^\prime\times_T T^\prime,S)\mright)\]
is an equivalence for every surjection $T^\prime\to T$ of $\lambda$-light profinite sets, which is clear.

Now let $X=\operatorname*{colim}_{i\in I} X_i$ be an $\aleph_1$-filtered colimit of condensed animae. To say that $\underline{S}$ is $\aleph_1$-compact means that the natural map
\[\operatorname*{colim}_{i\in I} X_i(S)\simeq \operatorname*{colim}_{i\in I}\operatorname{Map}\mleft(\underline{S},X_i\mright)\to \operatorname{Map}\mleft(\underline{S},X\mright)\simeq X(S)\]
is an equivalence. Equivalently, one must show that the functor
\[S \mapsto \operatorname*{colim}_{i\in I} X_i(S)\]
satisfies hyperdescent. This follows from the fact that each $X_i$ is a hypersheaf and that $\aleph_1$-filtered colimits commute with countable limits in $\mathsf{Ani}$ \cite[Tag \href{https://kerodon.net/tag/05Y0}{05Y0}]{kerodon}.
\end{proof}

Let $G\colon \mathsf{W}\to \mathsf{V}$ be a morphism in $\mathsf{Pr}^\mathsf{R}$ and let $Y\in \mathsf{Cond}_\lambda(\mathsf{W})$. Since $G$ preserves small limits, the composition $G\circ Y$ still satisfies hyperdescent. Thus, we obtain an induced functor
\begin{align*}
    \mathsf{Cond}_\lambda(\mathsf{W}) &\to \mathsf{Cond}_\lambda(\mathsf{V}) \\
    \mleft(\mathsf{Prof}_\lambda^\mathrm{op}\xrightarrow{\ Y\ } \mathsf{W}\mright)&\mapsto \mleft(\mathsf{Prof}_\lambda^\mathrm{op}\xrightarrow{\ Y\ } \mathsf{W}\xrightarrow{\ G\ }\mathsf{V}\mright).
\end{align*}
Under the natural equivalence $\mathsf{Cond}_\lambda(-)\simeq \mathsf{Cond}_\lambda(\mathsf{Ani})\otimes -$, the induced functor is nothing but $\operatorname{id}\otimes \,G$. Similarly, given a morphism $F\colon \mathsf{V}\to \mathsf{W}$ in $\mathsf{Pr}^\mathsf{L}$, we also obtain an induced functor
\[\mathsf{Cond}_\lambda(\mathsf{V})\simeq \mathsf{Cond}_\lambda(\mathsf{Ani})\otimes \mathsf{V}\xrightarrow{\ \operatorname{id}\otimes F\ }\mathsf{Cond}_\lambda(\mathsf{Ani})\otimes \mathsf{W}\simeq \mathsf{Cond}_\lambda(\mathsf{W})\]
which maps a $\mathsf{V}$-valued hypersheaf $X$ to the hypersheafification of $F\circ X$.

\begin{proposition}\label{Cond preserves presentability}
    Let $\kappa\geq \aleph_1$ be a regular cardinal. Then the assignment $\mathsf{V}\mapsto \mathsf{Cond}_\lambda(\mathsf{V})$ defines an endofunctor of each of the categories 
    \[\mathsf{Pr}^\mathsf{R}_{\kappa}, \quad \mathsf{Pr}^\mathsf{R}_{\kappa,\mathrm{st}}, \quad \mathsf{Pr}^\mathsf{L}_{\kappa}, \quad\mathsf{Pr}^\mathsf{L}_{\kappa,\mathrm{st}}.\]
    Moreover, it induces an endofunctor on the categories of commutative algebra objects in $\mathsf{Pr}^\mathsf{L}_{\kappa}$ and $\mathsf{Pr}^\mathsf{L}_{\kappa,\mathrm{st}}$.
\end{proposition}

\begin{proof}
We first show that $\mathsf{Cond}_\lambda(-)$ preserves $\kappa$-presentability. Since the tensor product of presentable categories does so \cite[Lem.~5.3.2.11]{higheralgebra}, it suffices to prove that $\mathsf{Cond}_\lambda(\mathsf{Ani})$ is $\kappa$-presentable. To this end, by \cite[Cor.~5.5.7.3]{lurie2009HTT}, it is enough to show that $\mathsf{Cond}_\lambda(\mathsf{Ani}) = \mathsf{HSh}(\mathsf{Prof}_\lambda,\mathsf{Ani})$ is closed under $\kappa$-filtered colimits in $\mathsf{PSh}(\mathsf{Prof}_\lambda,\mathsf{Ani})$, which is precisely the content of Lemma~\ref{compact condensed anima}. 

Next, since limits and colimits in functor categories are computed pointwise, the category $\mathsf{PSh}(\mathsf{Prof}_\lambda,\mathsf{V})$ is stable whenever $\mathsf{V}$ is. Finally, the full subcategory $\mathsf{Cond}_\lambda(\mathsf{V}) = \mathsf{HSh}(\mathsf{Prof}_\lambda,\mathsf{V})$ is closed under fibers and shifts, and is therefore stable as well.

Since $\mathsf{Cond}_{\lambda}(\mathsf{Ani})$ is a topos, small colimits in $\mathsf{Cond}_{\lambda}(\mathsf{Ani})$ are universal. Therefore, to show that its cartesian symmetric monoidal structure defines an object of $\mathsf{CAlg}(\mathsf{Pr}^{\mathsf{L}}_{\kappa})$, it suffices to verify that the tensor unit $\ast$ is $\kappa$-compact and that the full subcategory of $\kappa$-compact objects is closed under finite products.

The tensor unit is representable, and hence $\kappa$-compact by Lemma~\ref{compact condensed anima}. Moreover, by \cite[Prop.~5.3.4.17 and Cor.~5.5.7.3]{lurie2009HTT}, every $\kappa$-compact object of $\mathsf{Cond}_{\lambda}(\mathsf{Ani})$ is a retract of a $\kappa$-small colimit of representables. It follows again from Lemma~\ref{compact condensed anima} that $\kappa$-compact objects are stable under finite products.

We now show that $\mathsf{Cond}_\lambda(\mathsf{V})$ inherits a $\kappa$-presentably symmetric monoidal structure from $\mathsf{V}$. By \cite[Cor.~3.2.4.8]{higheralgebra}, the category $\mathsf{CAlg}(\mathsf{Pr}^\mathsf{L}_\kappa)$ admits finite coproducts, and these are compatible with the tensor product in $\mathsf{Pr}^\mathsf{L}_\kappa$ in the sense that the diagram
\[\begin{tikzcd}[ampersand replacement=\&]
	{\mathsf{CAlg}(\mathsf{Pr}^\mathsf{L}_\kappa)} \&\&\& {\mathsf{CAlg}(\mathsf{Pr}^\mathsf{L}_\kappa)} \\
	{\mathsf{Pr}^\mathsf{L}_\kappa} \&\&\& {\mathsf{Pr}^\mathsf{L}_\kappa,}
	\arrow["{\mathsf{Cond}_\lambda(\mathsf{Ani})\sqcup-}", from=1-1, to=1-4]
	\arrow[from=1-1, to=2-1]
	\arrow[from=1-4, to=2-4]
	\arrow["{\mathsf{Cond}_\lambda(\mathsf{Ani})\otimes-}", from=2-1, to=2-4]
\end{tikzcd}\]
commutes, where the vertical maps are the forgetful functors. It follows that tensoring with $\mathsf{Cond}_\lambda(\mathsf{Ani})$ carries commutative algebra objects to commutative algebra objects, giving the desired $\kappa$-presentably symmetric monoidal structure on $\mathsf{Cond}_\lambda(\mathsf{V})$.

For the stable case, recall that $-\otimes\mathsf{Sp}\colon\mathsf{Pr}^\mathsf{L}_\kappa\to\mathsf{Pr}^\mathsf{L}_{\kappa,\mathrm{st}}$ is symmetric monoidal. Hence it endows
\[\mathsf{Cond}_{\lambda}(\mathsf{Sp})\simeq \mathsf{Cond}_{\lambda}(\mathsf{Ani})\otimes \mathsf{Sp}.\]
with a $\kappa$-presentably symmetric monoidal structure. Applying the same argument as above, with $\mathsf{Cond}_{\lambda}(\mathsf{Ani})$ replaced by $\mathsf{Cond}_{\lambda}(\mathsf{Sp})$, gives the corresponding result in the stable setting.
\end{proof}

Let $\mathsf{Ab}$ and $\mathsf{Ring}$ denote the categories of (animated) abelian groups and commutative rings, respectively. Equivalently, these are the categories of finite-product-preserving functors
\[\mathsf{Lat}^\mathrm{op}\to \mathsf{Ani} \qquad\text{and}\qquad \mathsf{Poly}^\mathrm{op}\to \mathsf{Ani},\]
where $\mathsf{Lat}$ is the category of finitely generated free abelian groups $\mathbb{Z}^n$, and $\mathsf{Poly}$ is the category of finitely generated polynomial rings $\mathbb{Z}[x_1,\dots,x_n]$, for $n\geq 0$.

By the Dold--Kan correspondence, the category $\mathsf{Ab}$ is equivalent to the connective part $\mathsf{D}(\mathbb{Z})_{\geq 0}$ of the derived category $\mathsf{D}(\mathbb{Z})$. Furthermore, it may be identified with
\[\mathsf{Mod}_{\mathsf{H}\mathbb{Z}}(\mathsf{Sp}^{\mathrm{cn}}),\]
the category of $\mathsf{H}\mathbb{Z}$-modules in connective spectra, where $\mathsf{H}\mathbb{Z}$ denotes the Eilenberg--MacLane spectrum of $\mathbb{Z}$.

Passing to condensed objects, \cite[Thm.~1.2.(2)]{standardtstructures} endows $\mathsf{Cond}_{\lambda}(\mathsf{D}(\mathbb{Z}))$ with a $t$\nobreakdash-structure whose connective and coconnective parts are given by
\[\mathsf{Cond}_{\lambda}(\mathsf{D}(\mathbb{Z}))_{\geq 0}\simeq \mathsf{Cond}_{\lambda}(\mathsf{D}(\mathbb{Z})_{\geq 0}) \quad \text{and}\quad \mathsf{Cond}_{\lambda}(\mathsf{D}(\mathbb{Z}))_{\leq 0}\simeq \mathsf{Cond}_{\lambda}(\mathsf{D}(\mathbb{Z})_{\leq 0}),\]
respectively. Consequently, $\mathsf{Cond}_{\lambda}(\mathsf{Ab})$ identifies with the full subcategory of connective objects in $\mathsf{Cond}_{\lambda}(\mathsf{D}(\mathbb{Z}))$. This $t$-structure is right complete and compatible with both filtered colimits and the symmetric monoidal structure \cite[Prop.~3.8]{standardtstructures}.

\begin{proposition}\label{CondAb is mu-compact}
    Let $\mu=((2^{<\lambda})^{\aleph_0})^+$. Then $\mathsf{Cond}_\lambda(\mathsf{Ab})$ is a $\mu$-compact object of $\mathsf{CAlg}(\mathsf{Pr}^\mathsf{L}_{\aleph_1})$.
\end{proposition}

\begin{proof}
We begin by recalling that, by Stone duality, the category $\mathsf{Prof}_\lambda$ is anti-equivalent to the category $\mathsf{Bool}_\lambda$ of $\lambda$-small Boolean rings. It is a classical theorem that, for each infinite cardinal $\rho$, there are exactly $2^\rho$ isomorphism classes of Boolean rings of cardinality $\rho$ \cite{Monk1989}. Hence
\[\#\pi_0(\mathsf{Bool}_\lambda^\simeq)=\sup_{ \rho<\lambda}2^\rho\equalscolon 2^{<\lambda}.\]
Choose a skeleton $\overline{\mathsf{Bool}}_\lambda$ of $\mathsf{Bool}_\lambda$ and let $A$, $B$ be objects of it. We have that
\[\#\operatorname{Map}_{\mathsf{B}_\lambda}(A,B)\leq \#\operatorname{Map}_{\mathsf{Set}}(A,B)\leq 2^{\max(\aleph_0,\# A,\#B)}\leq 2^{<\lambda}.\]
It follows that the full subcategory $\overline{\mathsf{Prof}}_\lambda\subseteq \mathsf{Prof}_\lambda$ spanned by the profinite sets $S$ such that $\operatorname{Cont}(S,\mathbb{F}_2)\in \overline{\mathsf{Bool}}_\lambda$ is $(2^{<\lambda})^+$-small.

Let $\mathsf{C}$ be the $(2^{<\lambda})^+$-small category $\overline{\mathsf{Prof}}_\lambda\times \mathsf{Lat}$ and, for each object $(S,\mathbb{Z}^n)$ of $\mathsf{C}$, denote by $\yo(S,\mathbb{Z}^n)$ the corresponding representable presheaf in $\mathsf{P}(\mathsf{C})$. Let $W$ be the collection of morphisms in $\mathsf{P}(\mathsf{C})$ consisting of the maps
\begin{align*}
   (W_1)&&  \coprod_{j=1}^n\yo(S,\mathbb{Z})&\to\yo(S,\mathbb{Z}^n), && S\in\overline{\mathsf{Prof}}_\lambda,\ n\geq0;\\
   (W_2)&&  \coprod_{j=1}^m\yo(S_j,\mathbb{Z})&\to\yo\mleft(\coprod_{j=1}^m S_j,\mathbb{Z}\mright), && S_j\in\overline{\mathsf{Prof}}_\lambda,\ m\geq0;\\
    (W_3)&& \operatornamewithlimits{colim}_{[n]\in\Delta^\mathrm{op}}\yo(S_n,\mathbb{Z})&\to\yo(S,\mathbb{Z}), && S_\ast\to S\text{ a hypercover in }\overline{\mathsf{Prof}}_\lambda,
\end{align*}
and let $\mathsf{P}_0\subseteq\mathsf{P}(\mathsf{C})$ be the full subcategory spanned by the $W$-local objects.

Now let $X_0\colon \overline{\mathsf{Prof}}_\lambda^\mathrm{op}\times \mathsf{Lat}^\mathsf{op}\to \mathsf{Ani}$ be a presheaf, and let
\[X\colon \overline{\mathsf{Prof}}_\lambda^\mathrm{op}\to \mathsf{Fun}(\mathsf{Lat}^\mathsf{op},\mathsf{Ani})\]
be the corresponding curried functor. The presheaf $X_0$ is local with respect to the maps in $W_1$ precisely if $X$ takes values in $\mathsf{Ab}\subseteq \mathsf{Fun}(\mathsf{Lat}^\mathsf{op},\mathsf{Ani})$. Similarly, locality with respect to $W_2$ and $W_3$ is equivalent to the condition that $X$ is a hypersheaf. We therefore obtain an equivalence $\mathsf{P}_0\simeq \mathsf{Cond}_\lambda(\mathsf{Ab})$.

By the discussion above, the collections $W_1$ and $W_2$ are visibly $(2^{<\lambda})^+$-small. For $W_3$, note that $\Delta_+$ is countable. Hence the collection of augmented simplicial objects in $\overline{\mathsf{Prof}}_\lambda$ has cardinality at most
\[(2^{<\lambda})^{\aleph_0}<\mu.\]
In particular, $W_3$ is $\mu$-small, and therefore so is $W$. Since the source and target of every morphism in $W$ are $\aleph_1$-compact, the result follows from Propositions~\ref{forgetful functor preserves compacts} and \ref{compactness of local objects}.
\end{proof}

Contrary to what one might expect, the category $\mathsf{Ring}$ of (animated) commutative rings is not equivalent to
\[\mathsf{CAlg}(\mathsf{Ab})\simeq \mathsf{CAlg}\mleft(\mathsf{Mod}_{\mathsf{H}\mathbb{Z}}(\mathsf{Sp}^{\mathrm{cn}})\mright)\simeq\mathsf{CAlg}(\mathsf{Sp}^{\mathrm{cn}})_{\mathsf{H}\mathbb{Z}/}.\]
The universal property of animation \cite[Prop.~5.5.8.15]{lurie2009HTT} provides a sifted-colimit-preserving extension
\[\Theta\colon \mathsf{Ring}\to\mathsf{CAlg}(\mathsf{Sp}^{\mathrm{cn}})_{\mathsf{H}\mathbb{Z}/}\]
of the Eilenberg--MacLane functor $\mathsf{H}\colon \mathsf{Poly}\to\mathsf{CAlg}(\mathsf{Sp}^{\mathrm{cn}})_{\mathsf{H}\mathbb{Z}/}$. While $\Theta$ is not an equivalence, the corresponding functor obtained after replacing $\mathbb{Z}$ by $\mathbb{Q}$ is an equivalence.

That being said, $\Theta$ will be useful in our study of $\mathsf{Ring}$. Indeed, by \cite[Prop.~25.1.2.2]{lurie2018spectral}, the functor $\Theta$ is conservative and preserves all small limits and colimits. The following lemma establishes another important property of $\Theta$.

\begin{lemma}\label{Theta preserves kappa-compact objects}
    The functor $\Theta\colon \mathsf{Ring}\to \mathsf{CAlg}(\mathsf{Sp}^\mathrm{cn})_{\mathsf{H}\mathbb{Z}/}$ preserves $\aleph_1$-compact objects. Consequently, $\Theta$ is a map in $\mathsf{Pr}^\mathsf{L}_{\aleph_1}$.
\end{lemma}

\begin{proof}
    According to the Hilbert Basis Theorem \cite[Prop.~7.2.4.31]{higheralgebra}, the polynomial ring $\mathsf{H}\mathbb{Z}[x_1,\dots,x_n]$ is almost of finite presentation as an $\mathbb{E}_\infty$-algebra over $\mathsf{H}\mathbb{Z}$. By definition, this means that $\mathsf{H}\mathbb{Z}[x_1,\dots,x_n]$ is a compact object of $\tau_{\leq m}\mathsf{CAlg}(\mathsf{Sp}^\mathrm{cn})_{\mathsf{H}\mathbb{Z}/}$ for all $m\geq 0$. We claim that this implies that $\mathsf{H}\mathbb{Z}[x_1,\dots,x_n]$ is $\aleph_1$-compact for all $n\geq 0$.

    Let $(B_i)_{i\in I}$ be an $\aleph_1$-filtered diagram in $\mathsf{CAlg}(\mathsf{Sp}^\mathrm{cn})_{\mathsf{H}\mathbb{Z}/}$. Since Postnikov towers converge in this category \cite[Prop.~7.1.3.19]{higheralgebra}, we have that
    \begin{align*}
        \operatorname{Map}\mleft(\mathsf{H}\mathbb{Z}[x_1,\dots,x_n],\operatornamewithlimits{colim}_{i\in I}B_i\mright) &\simeq \lim_{m\geq 0}\operatorname{Map}\mleft(\mathsf{H}\mathbb{Z}[x_1,\dots,x_n],\tau_{\leq m}\operatornamewithlimits{colim}_{i\in I}B_i\mright) \\
        &\simeq \lim_{m\geq 0}\operatorname{Map}\mleft(\mathsf{H}\mathbb{Z}[x_1,\dots,x_n],\operatornamewithlimits{colim}_{i\in I}\tau_{\leq m}B_i\mright) \\
        &\simeq \lim_{m\geq 0}\operatornamewithlimits{colim}_{i\in I}\operatorname{Map}(\mathsf{H}\mathbb{Z}[x_1,\dots,x_n],\tau_{\leq m}B_i) \\
        &\simeq \operatornamewithlimits{colim}_{i\in I}\lim_{m\geq 0}\operatorname{Map}(\mathsf{H}\mathbb{Z}[x_1,\dots,x_n],\tau_{\leq m}B_i) \\
        &\simeq \operatornamewithlimits{colim}_{i\in I}\operatorname{Map}(\mathsf{H}\mathbb{Z}[x_1,\dots,x_n],B_i).
    \end{align*}
    Here, we used on the fourth line that $\aleph_1$-filtered colimits commute with countable limits in $\mathsf{Ani}$. In other words, $\mathsf{H}\mathbb{Z}[x_1,\dots,x_n]$ is $\aleph_1$-compact for all $n\geq 0$.

    To conclude that $\Theta$ preserves $\aleph_1$-compact objects, we will prove that its right adjoint $\Theta^R$ preserves $\aleph_1$-filtered colimits. Given an $\mathbb{E}_\infty$-algebra $B$ over $\mathsf{H}\mathbb{Z}$, the adjunction implies that $\Theta^R(B)$ is the functor
    \begin{align*}
        \mathsf{Poly}^\mathrm{op} &\to \mathsf{Ani}\\
        \mathbb{Z}[x_1,\dots,x_n] &\mapsto \operatorname{Map}(\mathsf{H}\mathbb{Z}[x_1,\dots,x_n],B).
    \end{align*}
    Since sifted colimits (in particular, $\aleph_1$-filtered colimits) in $\mathsf{Ring}$ are computed pointwise, the $\aleph_1$-compactness of $\mathsf{H}\mathbb{Z}[x_1,\dots,x_n]$ implies that $\Theta^R$ preserves $\aleph_1$-filtered colimits.
\end{proof}

For an object $\mathsf{W}$ of $\mathsf{CAlg}(\mathsf{Pr}^\mathsf{L}_\kappa)$, the natural equivalence
\[\mathsf{Fun}(\mathsf{Prof}_\lambda^\mathrm{op},\mathsf{CAlg}(\mathsf{W}))\simeq \mathsf{CAlg}(\mathsf{Fun}(\mathsf{Prof}_\lambda^\mathrm{op},\mathsf{W}))\]
is compatible with hypersheafification and therefore induces an equivalence
\[\mathsf{Cond}_\lambda(\mathsf{CAlg}(\mathsf{W}))\simeq\mathsf{CAlg}(\mathsf{Cond}_\lambda(\mathsf{W}));\]
see \cite[Prop.~2.1.6]{montagnani2025axiomaticapproachanalytic1affineness} for further details.
In particular, applying this to $\mathsf{W}=\mathsf{Ab}$ allows us, for a $\lambda$-light condensed ring $A^\triangleright$, to regard $\mathsf{Cond}_\lambda(\Theta)(A^\triangleright)$ as an object of
\[\mathsf{CAlg}(\mathsf{Cond}_\lambda(\mathsf{Ab}))\subseteq \mathsf{CAlg}(\mathsf{Cond}_\lambda(\mathsf{D}(\mathbb{Z}))).\] 
We denote this object by $A^\triangleright_\circ$.

\begin{definition}\label{derived category of a condensed ring}
Let $A^\triangleright$ be a $\lambda$-light condensed ring. Its \emph{derived category} is
\[\mathsf{D}(A^\triangleright)\colonequals\mathsf{Mod}_{A^\triangleright_\circ}\mleft(\mathsf{Cond}_\lambda(\mathsf{D}(\mathbb{Z}))\mright);\]
the category of $A^\triangleright_\circ$-modules in $\mathsf{Cond}_\lambda(\mathsf{D}(\mathbb{Z}))$. The \emph{connective derived category} of $A^\triangleright$ is the full subcategory $\mathsf{D}(A^{\triangleright})_{\geq0}$ spanned by those $A^{\triangleright}_{\circ}$-modules whose underlying object of $\mathsf{Cond}_{\lambda}(\mathsf{D}(\mathbb{Z}))$ lies in $\mathsf{Cond}_{\lambda}(\mathsf{Ab})$. 
\end{definition}

The following lemma shows that our nomenclature is well-chosen, and verifies some additional properties of these categories.

\begin{lemma}\label{t-structures on derived categories}
Let $A^\triangleright$ be a $\lambda$-light condensed ring. Then the following statements hold.
\begin{enumerate}
    \item The category $\mathsf{D}(A^\triangleright)$ admits a right complete $t$\nobreakdash-structure whose connective part is $\mathsf{D}(A^\triangleright)_{\geq0}$. This $t$\nobreakdash-structure is compatible with filtered colimits and the symmetric monoidal structure.
    \item There is a natural equivalence
    \[\mathsf{D}(A^\triangleright)_{\geq0}\simeq\mathsf{Mod}_{A^\triangleright_\circ}\mleft(\mathsf{Cond}_\lambda(\mathsf{Ab})\mright),\]
    and $\mathsf{D}(A^\triangleright)$ is the stabilization of $\mathsf{D}(A^\triangleright)_{\geq0}$. Moreover, $\mathsf{D}(A^\triangleright)_{\geq0}$ and $\mathsf{D}(A^\triangleright)$ are commutative algebra objects in the categories of $\mathsf{Cond}_\lambda(\mathsf{Ab})$-linear and $\mathsf{Cond}_\lambda(\mathsf{D}(\mathbb{Z}))$-linear $\aleph_1$-presentable categories, respectively.
    \item Every morphism $\varphi\colon A^\triangleright\to B^\triangleright$ of $\lambda$-light condensed rings induces an adjunction
    \[\varphi^*\colon\mathsf{D}(A^\triangleright)\rightleftarrows\mathsf{D}(B^\triangleright):\hspace{-0.2em}\varphi_*,\]
    where $\varphi^*$ is extension-of-scalars and $\varphi_*$ is restriction-of-scalars. These functors restrict to an adjunction between the connective derived categories.
\end{enumerate}
\end{lemma}

\begin{proof}
    As noted above, the $t$-structure on $\mathsf{Cond}_\lambda(\mathsf{D}(\mathbb{Z}))$ is right complete and compatible with both filtered colimits and the symmetric monoidal structure. By \cite[Prop.~A.15]{CartierCyclotomic}, the category $\mathsf{D}(A^\triangleright)$ inherits a $t$-structure compatible with the symmetric monoidal structure from $\mathsf{Cond}_\lambda(\mathsf{D}(\mathbb{Z}))$, with connective part
\[\mathsf{D}(A^\triangleright)_{\geq0}\simeq \mathsf{Mod}_{A^\triangleright_\circ}\mleft(\mathsf{Cond}_\lambda(\mathsf{Ab})\mright).\]
Since the forgetful functor $\mathsf{D}(A^\triangleright)\to\mathsf{Cond}_\lambda(\mathsf{D}(\mathbb{Z}))$ is conservative, $t$-exact, and preserves small colimits, the right completeness of this $t$-structure and its compatibility with filtered colimits follow from the corresponding properties of $\mathsf{Cond}_\lambda(\mathsf{D}(\mathbb{Z}))$.
    
The fact that $\mathsf{D}(A^\triangleright)$ is the stabilization of $\mathsf{D}(A^\triangleright)_{\geq0}$ is explained in \cite[Rem.~C.3.1.5]{lurie2018spectral}. Moreover, by Proposition~\ref{Cond preserves presentability} and \cite[Lem.~2.2]{aoki2025higherpresentablecategorieslimits}, the categories $\mathsf{D}(A^\triangleright)_{\geq0}$ and $\mathsf{D}(A^\triangleright)$ are naturally objects of
\[\mathsf{CAlg}(\mathsf{Pr}^\mathsf{L}_{\aleph_1})_{\mathsf{W}/},\]
for $\mathsf{W}\colonequals \mathsf{Cond}_\lambda(\mathsf{Ab})$ and $\mathsf{W}\colonequals \mathsf{Cond}_\lambda(\mathsf{D}(\mathbb{Z}))$, respectively.

The third item follows directly from \cite[Prop.~4.6.2.17]{higheralgebra}, applied first with $\mathsf{C}=\mathsf{M}=\mathsf{Cond}_\lambda(\mathsf{D}(\mathbb{Z}))$ and then with $\mathsf{C}=\mathsf{M}=\mathsf{Cond}_\lambda(\mathsf{Ab})$.
\end{proof}

\begin{proposition}\label{D is a map in Pr_k}
   The assignments $A^\triangleright\mapsto \mathsf{D}(A^\triangleright)_{\geq0}$ and $A^\triangleright\mapsto\mathsf{D}(A^\triangleright)$ define functors
    \[\mathsf{Cond}_\lambda(\mathsf{Ring})\to \mathsf{CAlg}(\mathsf{Pr}^\mathsf{L}_{\aleph_1})_{\mathsf{W}/},\]
    for $\mathsf{W}\colonequals \mathsf{Cond}_\lambda(\mathsf{Ab})$ and $\mathsf{W}\colonequals \mathsf{Cond}_\lambda(\mathsf{D}(\mathbb{Z}))$, respectively, that lie in $\mathsf{Pr}^\mathsf{L}_{\aleph_1}$.
\end{proposition}

\begin{proof}
We begin by considering $\mathsf{W}=\mathsf{Cond}_\lambda(\mathsf{Ab})$ and the connective derived category. The functor given on objects by $A^\triangleright\mapsto\mathsf{D}(A^{\triangleright})_{\geq0}$ is the composite
    \[\mathsf{Cond}_\lambda(\mathsf{Ring}) \xrightarrow{\ \mathsf{Cond}_\lambda(\Theta)\ }\mathsf{Cond}_\lambda(\mathsf{CAlg}(\mathsf{Ab}))\simeq \mathsf{CAlg}(\mathsf{W})\xrightarrow{\ \mathsf{Mod}_{(-)}(\mathsf{W}) \ }\mathsf{CAlg}(\mathsf{Pr}^{\mathsf{L}}_{\aleph_1})_{\mathsf{W}/},\]
    where the middle equivalence was explained just above Definition~\ref{derived category of a condensed ring}. Proposition~\ref{Cond preserves presentability} and Lemma~\ref{Theta preserves kappa-compact objects} imply that $\mathsf{Cond}_{\lambda}(\Theta)$ is a morphism in $\mathsf{Pr}^{\mathsf{L}}_{\aleph_1}$, and the same holds for $\mathsf{Mod}_{(-)}(\mathsf{W})$ by \cite[Lem.~2.11]{aoki2025higherpresentablecategorieslimits}.

    According to Lemma~\ref{t-structures on derived categories}, the functor given on objects by $A^\triangleright\mapsto\mathsf{D}(A^{\triangleright})$ is the composition of the functor above with
    \[\mathsf{Sp}\sqcup - \colon \mathsf{CAlg}(\mathsf{Pr}^{\mathsf{L}}_{\aleph_1})_{\mathsf{W}/}\to \mathsf{CAlg}(\mathsf{Pr}^{\mathsf{L}}_{\aleph_1})_{\mathsf{Sp}\sqcup\mathsf{W}/}.\]
    The latter has a right adjoint, given by restriction of scalars along the natural inclusion
    \[\mathsf{W}\to \mathsf{Sp}\sqcup \mathsf{W},\]
    where the coproduct is computed in $\mathsf{CAlg}(\mathsf{Pr}^{\mathsf{L}}_{\aleph_1})$. (See also the proof of Proposition~\ref{Cond preserves presentability}.)

    Since this right adjoint is given by precomposition, it leaves the underlying object of $\mathsf{CAlg}(\mathsf{Pr}^{\mathsf{L}}_{\aleph_1})$ unchanged. Moreover, filtered colimits in coslice categories are computed on underlying objects \cite[Tags~\href{https://kerodon.net/tag/0236}{0236}, \href{https://kerodon.net/tag/02KT}{02KT}, and \href{https://kerodon.net/tag/02PH}{02PH}]{kerodon}. It therefore preserves filtered colimits. Consequently, $\mathsf{Sp}\sqcup -$ is a morphism in $\mathsf{Pr}^{\mathsf{L}}_{\aleph_1}$, which completes the proof.
\end{proof}

\section{Analytic rings}
Before stating the definition of pre-analytic rings, we first make the following construction. Let $\mathsf{FF}$ denote the full subcategory of $\mathsf{Fun}(\Delta^1,\widehat{\mathsf{Cat}})$ consisting of fully faithful functors, and consider the pullback diagram
\[\begin{tikzcd}[ampersand replacement=\&]
	{\mathsf{P}_{\lambda}} \&\& {\mathsf{FF}^\mathrm{op}} \\
	{\mathsf{Cond}_{\lambda}(\mathsf{Ring})} \&\& {\widehat{\mathsf{Cat}}^\mathrm{op},}
	\arrow[from=1-1, to=1-3]
	\arrow[from=1-1, to=2-1]
	\arrow["t", from=1-3, to=2-3]
	\arrow["{\mathsf{D}(-)}", from=2-1, to=2-3]
\end{tikzcd}\]
where $t \colon \mathsf{FF} \to \widehat{\mathsf{Cat}}$ sends a functor to its target.

Since $t$ is an isofibration, \cite[Tag \href{https://kerodon.net/tag/033T}{033T}]{kerodon} provides an explicit description of $\mathsf{P}_\lambda$. Its objects are pairs $A=(A^\triangleright, \mathsf{D}(A))$, where $A^\triangleright$ is a $\lambda$-light condensed ring and $\mathsf{D}(A)$ is a full subcategory of $\mathsf{D}(A^\triangleright)$. A morphism $A \to B$ in $\mathsf{P}_\lambda$ consists of a morphism of $\lambda$-light condensed rings $\varphi\colon A^\triangleright \to B^\triangleright$ such that the associated restriction-of-scalars functor
\[\varphi_*\colon \mathsf{D}(B^\triangleright) \to \mathsf{D}(A^\triangleright)\]
carries $\mathsf{D}(B)$ into $\mathsf{D}(A)$. 

\begin{definition}\label{def panring}
A \emph{$\lambda$-light pre-analytic ring} $A$ is a pair $(A^\triangleright,\mathsf{D}(A))$ consisting of a $\lambda$-light condensed ring $A^\triangleright$ and a full subcategory $\mathsf{D}(A)$ of $\mathsf{D}(A^\triangleright)$ such that:
\begin{enumerate}[label=(\roman*)]
    \item The full subcategory $\mathsf{D}(A)\subset \mathsf{D}(A^\triangleright)$ is closed under small limits and colimits (hence the inclusion $\iota_A\colon \mathsf{D}(A)\hookrightarrow\mathsf{D}(A^\triangleright)$ admits a left adjoint, denoted $-\otimes_{A^\triangleright}A$);
    \item For every $M \in \mathsf{D}(A^\triangleright)$ and $N \in \mathsf{D}(A)$, the object $\underline{\operatorname{Hom}}(M,N)$ lies in $\mathsf{D}(A)$;
    \item The composition $L_A\colonequals\iota_A \circ \bigl(-\otimes_{A^\triangleright}A\bigr)$ preserves connective objects.
\end{enumerate}
We define $\mathsf{PAnRing}_{\lambda}$ to be the full subcategory of $\mathsf{P}_{\lambda}$ spanned by the $\lambda$-light pre-analytic rings.
\end{definition}

\begin{theorem}\label{PAnRing is presentable}
    The category $\mathsf{PAnRing}_\lambda$ is $\mu$-presentable, for $\mu=((2^{<\lambda})^{\aleph_0})^+$.
\end{theorem}

The proof of this theorem requires us to reinterpret each of the three conditions in Definition~\ref{def panring}. We begin with the third, which leads us to consider a variant $\mathsf{P}_\lambda^\mathrm{cn}$ of $\mathsf{P}_\lambda$, obtained by replacing the functor $\mathsf{D}(-)$ in its definition with
\[\mathsf{D}(-)_{\geq 0}\colon \mathsf{Cond}_\lambda(\mathsf{Ring})\to \widehat{\mathsf{Cat}}^\mathrm{op}.\]
In the same spirit, we introduce the following connective analogue of Definition~\ref{def panring}.

\begin{definition}\label{def connective panring}
We denote by $\mathsf{PAnRing}^\mathrm{cn}_\lambda$ the full category of $\mathsf{P}_\lambda^\mathrm{cn}$ spanned by pairs $A=(A^{\triangleright},\mathsf{D}(A)_{\geq0})$, consisting of a $\lambda$-light condensed ring $A^{\triangleright}$ and a full subcategory $\mathsf{D}(A)_{\geq0}$ of $\mathsf{D}(A^{\triangleright})_{\geq0}$ such that: 
    \begin{enumerate}[label=(\roman*)]
        \item The full subcategory $\mathsf{D}(A)_{\geq0}\subset\mathsf{D}(A^{\triangleright})_{\geq0}$ is closed under small limits and colimits (hence the inclusion admits a left adjoint, which we still denote by $-\otimes_{A^{\triangleright}}A$);
        \item For every $M \in \mathsf{D}(A^\triangleright)_{\geq0}$ and $N \in \mathsf{D}(A)_{\geq0}$, the object $\tau_{\geq 0}\underline{\operatorname{Hom}}(M,N)$ lies in $\mathsf{D}(A)_{\geq0}$.
    \end{enumerate}
\end{definition}

There is a natural comparison between $\mathsf{PAnRing}_\lambda$ and $\mathsf{PAnRing}_\lambda^{\mathrm{cn}}$. To construct it, let $\mathsf{FF}^{\mathsf{R}}$ denote the full subcategory of $\mathsf{Fun}(\Delta^1,\mathsf{Pr}^{\mathsf{R}})$ spanned by the fully faithful functors. The reflection theorem implies that the composite $\mathsf{PAnRing}^\mathrm{cn}_\lambda\hookrightarrow \mathsf{P}^\mathrm{cn}_\lambda\to \mathsf{FF}^\mathrm{op}$ factors as
\[\begin{tikzcd}[ampersand replacement=\&]
	{\mathsf{PAnRing}^\mathrm{cn}_\lambda} \&\& {\mathsf{FF}^{\mathsf{R},\mathrm{op}}} \\
	{\mathsf{P}^\mathrm{cn}_\lambda} \&\& {\mathsf{FF}^\mathrm{op}.}
	\arrow[dashed, from=1-1, to=1-3]
	\arrow[hook, from=1-1, to=2-1]
	\arrow[hook, from=1-3, to=2-3]
	\arrow[from=2-1, to=2-3]
\end{tikzcd}\]

Now let $A^\triangleright$ be a $\lambda$-light condensed ring. By Lemma~\ref{t-structures on derived categories}, there is an equivalence of categories $\mathsf{D}(A^\triangleright)\simeq \mathsf{D}(A^\triangleright)_{\geq 0}\otimes \mathsf{Sp}$. It follows from the universal property of the fiber product that these equivalences assemble into the dashed functor $\Psi_\lambda$ in the diagram below:
\[\begin{tikzcd}[ampersand replacement=\&]
	{\mathsf{PAnRing}_\lambda^\mathrm{cn}} \& {\mathsf{FF}^{\mathsf{R},\mathrm{op}}} \&\& {\mathsf{FF}^{\mathsf{R},\mathrm{op}}} \\
	{\mathsf{P}_\lambda^\mathrm{cn}} \& {\mathsf{P}_\lambda} \&\& {\mathsf{FF}^\mathrm{op},} \\
	{\mathsf{Cond}_\lambda(\mathsf{Ring})} \& {\mathsf{Cond}_\lambda(\mathsf{Ring})} \&\& {\widehat{\mathsf{Cat}}^\mathrm{op}.}
	\arrow[from=1-1, to=1-2]
	\arrow[hook, from=1-1, to=2-1]
	\arrow["{\Psi_\lambda}"', dashed, from=1-1, to=2-2]
	\arrow["{-\otimes\mathsf{Sp}}", from=1-2, to=1-4]
	\arrow[hook, from=1-4, to=2-4]
	\arrow[from=2-1, to=3-1]
	\arrow[from=2-2, to=2-4]
	\arrow[from=2-2, to=3-2]
	\arrow["t", from=2-4, to=3-4]
	\arrow[equals, from=3-1, to=3-2]
	\arrow["{\mathsf{D}(-)}", from=3-2, to=3-4]
\end{tikzcd}\]

\begin{lemma}\label{preanalytic rings are determined by connective halves}
The functor $\Psi_\lambda\colon \mathsf{PAnRing}_\lambda^{\mathrm{cn}}\to \mathsf{P}_\lambda$ is fully faithful, with essential image $\mathsf{PAnRing}_\lambda$. In particular, it induces an equivalence $\mathsf{PAnRing}_\lambda^{\mathrm{cn}}\simeq \mathsf{PAnRing}_\lambda$.
\end{lemma}

\begin{proof}
We first outline the argument, before giving the details of each step. After observing that the essential image of $\Psi_\lambda$ is contained in $\mathsf{PAnRing}_\lambda$, we obtain a commutative diagram
\[\begin{tikzcd}[ampersand replacement=\&,column sep = tiny]
	{\mathsf{PAnRing}_\lambda^\mathrm{cn}} \&\& {\mathsf{PAnRing}_\lambda} \\
	\& {\mathsf{Cond}_\lambda(\mathsf{Ring}),}
	\arrow["{\Psi_\lambda}", from=1-1, to=1-3]
	\arrow["{U^\mathrm{cn}}"', from=1-1, to=2-2]
	\arrow["U", from=1-3, to=2-2]
\end{tikzcd}\]
where $U^\mathrm{cn}$ and $U$ forget the categories of complete modules. We then show that both $U^\mathrm{cn}$ and $U$ are cocartesian fibrations and that $\Psi_\lambda$ preserves cocartesian maps. According to \cite[Tag~\href{https://kerodon.net/tag/028B}{028B}]{kerodon}, it therefore suffices to show that $\Psi_\lambda$ induces an equivalence on each fiber. On fibers, this reduces to the classical fact that pre-analytic ring structures can be encoded in terms of abelian categories.

\textsc{Step I: The essential image of $\Psi_\lambda$ is contained in $\mathsf{PAnRing}_\lambda$.} Let $(A^\triangleright,\mathsf{D}(A)_{\geq 0})$ be an object of $\mathsf{PAnRing}_\lambda^\mathrm{cn}$. Since stabilization is a 2-functor \cite[Rem.~6.1.7]{CLL}, we obtain an adjunction
\[\begin{tikzcd}[ampersand replacement=\&]
	{\mathsf{D}(A^\triangleright)_{\geq 0}\otimes \mathsf{Sp}} \&\& {\mathsf{D}(A)_{\geq 0}\otimes \mathsf{Sp}.}
	\arrow[""{name=0, anchor=center, inner sep=0}, "{(- \otimes_{A^\triangleright}A)\otimes \mathsf{Sp}}", bend left, from=1-1, to=1-3]
	\arrow[""{name=1, anchor=center, inner sep=0}, bend left, hook', from=1-3, to=1-1]
	\arrow["\dashv"{anchor=center, rotate=-90}, draw=none, from=0, to=1]
\end{tikzcd}\]
In particular, the full subcategory $\mathsf{D}(A)_{\geq 0}\otimes \mathsf{Sp}\subseteq \mathsf{D}(A^\triangleright)_{\geq 0}\otimes \mathsf{Sp}\simeq \mathsf{D}(A^\triangleright)$ is stable under small limits. On the other hand, the inclusion $\mathsf{D}(A)_{\geq 0}\hookrightarrow \mathsf{D}(A^\triangleright)_{\geq 0}$ preserves small colimits and therefore admits a right adjoint. Stabilizing this adjunction shows that $\mathsf{D}(A)_{\geq 0}\otimes \mathsf{Sp}\subseteq \mathsf{D}(A^\triangleright)$ is also closed under small colimits.

By Proposition~\ref{enhanced reflection theorem}, the completion functor
\[-\otimes_{A^\triangleright}A\colon \mathsf{D}(A^\triangleright)_{\geq 0}\to \mathsf{D}(A)_{\geq 0}\]
endows $\mathsf{D}(A)_{\geq 0}$ with a symmetric monoidal structure for which this functor is symmetric monoidal. Moreover, Proposition~\ref{enhanced reflection theorem} shows that it defines a morphism in $\mathsf{CAlg}(\mathsf{Pr}^\mathsf{L}_{\aleph_1})$. Applying the argument of Proposition~\ref{D is a map in Pr_k}, we obtain the corresponding statement for
\[(-\otimes_{A^\triangleright}A)\otimes\mathsf{Sp}.\]
In particular, this functor is \emph{compatible with the symmetric monoidal structure} in the sense of \cite[Def.~2.2.1.6]{higheralgebra}. Condition~(ii) of Definition~\ref{def panring} now follows from \cite[Lem.~7.5]{reflectionprinciple}.

Finally, condition~(iii) of Definition~\ref{def panring} follows by applying \cite[Prop.~C.3.1.1]{lurie2018spectral} to the inclusion $\mathsf{D}(A)_{\geq 0}\hookrightarrow\mathsf{D}(A^\triangleright)_{\geq 0}$ and to the completion functor $-\otimes_{A^\triangleright}A$. 

\textsc{Step II: $U^\mathrm{cn}$ and $U$ are cocartesian fibrations.} We treat $U$, as the same argument applies to $U^\mathrm{cn}$. Let $A=(A^\triangleright,\mathsf{D}(A))$ be a $\lambda$-light pre-analytic ring, and let $\varphi\colon A^\triangleright\to B^\triangleright$ be a morphism of $\lambda$-light condensed rings. Write $B^\triangleright_{/A}$ for the induced pre-analytic ring structure on $B^\triangleright$ \cite[Lem.~4.1.9]{camargo2026notessolidgeometry}. Thus, we have a pushout square of $\lambda$-light pre-analytic rings
\[\begin{tikzcd}[ampersand replacement=\&]
	{A^\triangleright} \& A \\
	{B^\triangleright} \& {B^\triangleright_{/A}.}
	\arrow[from=1-1, to=1-2]
	\arrow["{\varphi}"', from=1-1, to=2-1]
	\arrow["\overline{\varphi}", from=1-2, to=2-2]
	\arrow[from=2-1, to=2-2]
\end{tikzcd}\]
The functor $U$ maps $\overline{\varphi}\colon A\to B^\triangleright_{/A}$ to $\varphi$, so it suffices to prove that $\overline{\varphi}$ is a $U$-cocartesian morphism.

For any $\lambda$-light pre-analytic ring $C$, the universal property of the pushout gives a cartesian square
\[\begin{tikzcd}[ampersand replacement=\&]
	{\operatorname{Map}_{\mathsf{PAnRing}_\lambda}(B^\triangleright_{/A},C)} \& {\operatorname{Map}_{\mathsf{PAnRing}_\lambda}(A,C)} \\
	{\operatorname{Map}_{\mathsf{PAnRing}_\lambda}(B^\triangleright,C)} \& {\operatorname{Map}_{\mathsf{PAnRing}_\lambda}(A^\triangleright,C).}
	\arrow[from=1-1, to=1-2]
	\arrow[from=1-1, to=2-1]
	\arrow[from=1-2, to=2-2]
	\arrow[from=2-1, to=2-2]
\end{tikzcd}\]
By \cite[Lem.~4.1.5]{camargo2026notessolidgeometry}, we have
\[\operatorname{Map}_{\mathsf{PAnRing}_\lambda}(A^\triangleright,C)\simeq \operatorname{Map}_{\mathsf{Cond}_\lambda(\mathsf{Ring})}(A^\triangleright,C^\triangleright),\]
and similarly for $B^\triangleright$. Hence, the square above identifies with the cartesian square
\[\begin{tikzcd}[ampersand replacement=\&]
	{\operatorname{Map}_{\mathsf{PAnRing}_\lambda}(B^\triangleright_{/A},C)} \& {\operatorname{Map}_{\mathsf{PAnRing}_\lambda}(A,C)} \\
	{\operatorname{Map}_{\mathsf{Cond}_\lambda(\mathsf{Ring})}(B^\triangleright,C)} \& {\operatorname{Map}_{\mathsf{Cond}_\lambda(\mathsf{Ring})}(A^\triangleright,C),}
	\arrow[from=1-1, to=1-2]
	\arrow[from=1-1, to=2-1]
	\arrow[from=1-2, to=2-2]
	\arrow[from=2-1, to=2-2]
\end{tikzcd}\]
which proves that $\overline{\varphi}$ is $U$-cocartesian and completes the proof.

\textsc{Step III: The functor $\Psi_\lambda$ preserves cocartesian morphisms.} Running the preceding argument in reverse, we see that a morphism $A\to B$ in $\mathsf{PAnRing}_\lambda$ (resp.~in $\mathsf{PAnRing}^\mathrm{cn}_\lambda$) is $U$-cocartesian (resp.~$U^\mathrm{cn}$-cocartesian) if and only if the diagram
\[\begin{tikzcd}[ampersand replacement=\&]
	{A^\triangleright} \& {B^\triangleright} \\
	A \& B
	\arrow[from=1-1, to=1-2]
	\arrow[from=1-1, to=2-1]
	\arrow[from=1-2, to=2-2]
	\arrow[from=2-1, to=2-2]
\end{tikzcd}\]
is a pushout, or equivalently, if and only if $B$ carries the pre-analytic ring structure on $B^\triangleright$ induced from $A$.

It therefore remains to show that $\Psi_\lambda$ preserves induced pre-analytic ring structures. Let $A=(A^\triangleright,\mathsf{D}(A)_{\geq 0})$ be an object of $\mathsf{PAnRing}^\mathrm{cn}_\lambda$ and let $\varphi\colon A^\triangleright\to B^\triangleright$ be a morphism of $\lambda$-light condensed rings. By \cite[Lem.~4.1.9]{camargo2026notessolidgeometry}, the pushout $B^\triangleright_{A/}$ defining the induced pre-analytic ring structure exists, and its connective derived category admits the explicit description
\[\mathsf{D}(B^\triangleright_{A/})_{\geq 0}\simeq \{ M\in \mathsf{D}(B^\triangleright)_{\geq 0} \mid \varphi_*M\in \mathsf{D}(A)_{\geq 0}\}.\]
Thus, it suffices to verify that the stabilization $\mathsf{D}(B^\triangleright_{A/})_{\geq 0}\otimes\mathsf{Sp}$ admits the analogous description in terms of $\mathsf{D}(B^\triangleright)\simeq \mathsf{D}(B^\triangleright)_{\geq0}\otimes \mathsf{Sp}$ and $\mathsf{D}(A)_{\geq0}\otimes \mathsf{Sp}$.

By \cite[Prop.~4.1.8]{camargo2026notessolidgeometry}, an object $M\in\mathsf{D}(B^\triangleright)\simeq\mathsf{D}(B^\triangleright)_{\geq 0}\otimes\mathsf{Sp}$ belongs to $\mathsf{D}(B^\triangleright_{A/})_{\geq 0}\otimes\mathsf{Sp}$ if and only if $\pi_i(M)$ does for every $i\in\mathbb{Z}$. Since $\pi_i(M)$ is static, this is equivalent to requiring that $\pi_i(M)$ belong to $\mathsf{D}(B^\triangleright_{A/})_{\geq 0}$ for every $i\in\mathbb{Z}$. By the description above, the latter condition amounts to
\[\varphi_*\pi_i(M)\in\mathsf{D}(A)_{\geq 0}\qquad\text{for all }i\in\mathbb{Z}.\]
Since restriction of scalars is $t$-exact, we have $\varphi_*\pi_i(M)\simeq\pi_i(\varphi_*M)$. A second application of \cite[Prop.~4.1.8]{camargo2026notessolidgeometry} therefore shows that the condition above is equivalent to
\[\varphi_*M\in\mathsf{D}(A)_{\geq 0}\otimes\mathsf{Sp}.\]
This gives the desired description of the induced pre-analytic ring structure.

\textsc{Step IV: $\Psi_\lambda$ is an equivalence on fibers.} We have now established the hypotheses of \cite[Tag~\href{https://kerodon.net/tag/028B}{028B}]{kerodon}, so it remains to show that $\Psi_\lambda$ induces an equivalence on each fiber. For a $\lambda$-light condensed ring $A^\triangleright$, the fiber of $U$ over $A^\triangleright$ is the partially ordered set of pre-analytic ring structures on $A^\triangleright$, ordered by reverse inclusion, and similarly for $U^\mathrm{cn}$. The required equivalence is then a direct consequence of \cite[Thm.~4.4.1]{camargo2026notessolidgeometry}.
\end{proof}

Since presentability is not preserved under passage to opposite categories, it will be convenient to reformulate the definition of $\mathsf{PAnRing}_\lambda^\mathrm{cn}$ in terms of the covariant functor
\[\mathsf{D}(-)_{\geq 0}\colon \mathsf{Cond}_\lambda(\mathsf{Ring})\to \widehat{\mathsf{Cat}},\]
rather than its contravariant counterpart.

\begin{lemma}\label{contractability of fibers}
Let $\varphi\colon A^\triangleright\to B^\triangleright$ be a morphism of $\lambda$-light condensed rings, and consider full subcategories $\mathsf{D}(A)_{\geq0}\subseteq \mathsf{D}(A^\triangleright)_{\geq0}$ and $\mathsf{D}(B)_{\geq0}\subseteq \mathsf{D}(B^\triangleright)_{\geq0}$ closed under small limits and colimits. We retain the notation of Definition~\ref{def connective panring}, but do not assume condition (ii) therein. The following conditions are equivalent:
    \begin{enumerate}[label=(\alph*)]
        \item The restriction-of-scalars $\varphi_*\colon\mathsf{D}(B^\triangleright)_{\geq0} \to\mathsf{D}(A^\triangleright)_{\geq0}$ sends $\mathsf{D}(B)_{\geq0}$ into $\mathsf{D}(A)_{\geq0}$;
        \item There exists a functor $\overline{\varphi}^*\colon \mathsf{D}(A)_{\geq0}\to \mathsf{D}(B)_{\geq0}$ and a natural isomorphism
        \[\alpha_\varphi\colon \bigl(-\otimes_{B^\triangleright}B\bigr)\circ\varphi^*\to \overline{\varphi}^*\circ \bigl(-\otimes_{A^\triangleright}A\bigr)\]
        witnessing the commutativity of the diagram
\[\begin{tikzcd}[ampersand replacement=\&]
	{\mathsf{D}(A^\triangleright)_{\geq0}} \&\& {\mathsf{D}(A)_{\geq0}} \\
	{\mathsf{D}(B^\triangleright)_{\geq0}} \&\& {\mathsf{D}(B)_{\geq0}.}
	\arrow["{-\otimes_{A^\triangleright}A}", from=1-1, to=1-3]
	\arrow["{\varphi^*}"', from=1-1, to=2-1]
	\arrow["{\overline{\varphi}^*}", from=1-3, to=2-3]
	\arrow["{-\otimes_{B^\triangleright}B}", from=2-1, to=2-3]
\end{tikzcd}\]
    \end{enumerate}
The anima of pairs $(\overline{\varphi}^*,\alpha_\varphi)$ as in (b) is contractible when these conditions hold, and empty otherwise.
\end{lemma}

\begin{proof}
Assume (a), so that $\varphi_*$ restricts to a functor $\overline{\varphi}_*\colon \mathsf{D}(B)_{\geq0}\to\mathsf{D}(A)_{\geq0}$, and define $\overline{\varphi}^{*}\colonequals
\bigl(-\otimes_{B^\triangleright}B\bigr)\circ\varphi^{*}\circ\iota_A$. For $M\in\mathsf{D}(A)_{\geq0}$ and $N\in\mathsf{D}(B)_{\geq0}$, we have that
\begin{align*}
        \operatorname{Map}_{\mathsf{D}(B)_{\geq0}}(\overline{\varphi}^{*}M,N) &\simeq \operatorname{Map}_{\mathsf{D}(B^{\triangleright})_{\geq0}}(\varphi^{*}\iota_{A}M,\iota_BN) && B\otimes_{B^{\triangleright}}-\dashv \iota_{B} \\
        &\simeq\operatorname{Map}_{\mathsf{D}(A^{\triangleright})_{\geq0}}(\iota_{A}M,\varphi_{*}\iota_B N) && \varphi^{*}\dashv \varphi_{*} \\
        &\simeq\operatorname{Map}_{\mathsf{D}(A^{\triangleright})_{\geq0}}(\iota_{A}M,\iota_A\overline{\varphi}_{*} N) && \varphi_{*}\circ \iota_B\simeq \iota_A\circ \overline{\varphi}_{*} \\
        &\simeq \operatorname{Map}_{\mathsf{D}(A)_{\geq0}}(M,\overline{\varphi}_{*}N) && \iota_A\text{ is fully faithful}.
    \end{align*}
As a result, the functors $\bigl(-\otimes_{B^\triangleright}B\bigr)\circ\varphi^*$ and $\overline{\varphi}^*\circ \bigl(-\otimes_{A^\triangleright}A\bigr)$ are both left adjoint to
\[\varphi_{*}\circ \iota_B\simeq \iota_A\circ \overline{\varphi}_{*},\]
and so they are naturally isomorphic.

Conversely, assume (b). Precomposing $\alpha_\varphi$ with $\iota_A$ and using that $\iota_A$ is fully faithful, we obtain a natural equivalence
\[\overline{\varphi}^*\simeq\overline{\varphi}^*\circ \bigl(-\otimes_{A^\triangleright}A\bigr)\circ\iota_A\simeq\bigl(-\otimes_{B^\triangleright}B\bigr)\circ\varphi^*\circ\iota_A.\]
Since $\iota_A$ preserves small colimits and the other two functors on the right are left adjoints, $\overline{\varphi}^*$ preserves small colimits; hence admits a right adjoint $\overline{\varphi}_*$. Taking right adjoints in $\alpha_\varphi$ yields a natural equivalence
\[\varphi_*\circ\iota_B\simeq\iota_A\circ\overline{\varphi}_*,\]
showing that $\varphi_*$ sends $\mathsf{D}(B)_{\geq0}$ into $\mathsf{D}(A)_{\geq0}$, as required.

Finally, since the completion functor $-\otimes_{A^\triangleright}A\colon \mathsf{D}(A^\triangleright)_{\geq 0}\to \mathsf{D}(A)_{\geq 0}$ is a localization, it induces a fully faithful functor
\[\mathsf{Fun}\mleft(\mathsf{D}(A)_{\geq 0},\mathsf{D}(B)_{\geq 0}\mright)\xhookrightarrow{\hspace{0.8em} -\circ \bigl(-\otimes_{A^\triangleright}A\bigr)\hspace{0.8em} }\mathsf{Fun}\mleft(\mathsf{D}(A^\triangleright)_{\geq 0},\mathsf{D}(B)_{\geq 0}\mright).\]
The induced functor on cores is still fully faithful \cite[Tag~\href{https://kerodon.net/tag/01JP}{01JP}]{kerodon}, and the anima of pairs $(\overline{\varphi}^*,\alpha_\varphi)$ is its fiber over $\bigl(-\otimes_{B^\triangleright}B\bigr)\circ\varphi^*$. It follows that this anima is either empty or contractible \cite[Tags~\href{https://kerodon.net/tag/06SF}{06SF}
and \href{https://kerodon.net/tag/056B}{056B}]{kerodon}; the latter happening precisely when the equivalent conditions are satisfied.
\end{proof}

The following lemma makes precise, with the size bounds needed here, the connective pullback description outlined in \cite[\S2.1.1.9]{Kipp}.

\begin{lemma}\label{PAnRing as a pullback}
Set $\mathsf{W}\colonequals\mathsf{Cond}_\lambda(\mathsf{Ab})$. Then $\mathsf{PAnRing}^\mathrm{cn}_\lambda$ is naturally equivalent to the fiber product of the diagram
\[\mathsf{Cond}_\lambda(\mathsf{Ring})\xrightarrow{\ \mathsf{D}(-)_{\geq 0}\ }\mathsf{CAlg}(\mathsf{Pr}^\mathsf{L}_{\aleph_1})_{\mathsf{W}/}\xleftarrow{\ s\ }\overline{\mathsf{Loc}}_{\aleph_1}^\otimes(\mathsf{W}),\]
where $s$ is the functor sending each symmetric monoidal localization to its source.
\end{lemma}

\begin{proof}
Let $\mathsf{Q}^\mathrm{cn}_\lambda$ denote the fiber product in the statement. Since $\overline{\mathsf{Loc}}^{\otimes}_{\aleph_1}(\mathsf{W})$ is a replete full subcategory of $\mathsf{Fun}(\Delta^{1},\mathsf{CAlg}(\mathsf{Pr}^\mathsf{L}_{\aleph_1})_{\mathsf{W}/})$, the functor
\[s\colon \overline{\mathsf{Loc}}_{\aleph_1}^\otimes(\mathsf{W})\to \mathsf{CAlg}(\mathsf{Pr}^\mathsf{L}_{\aleph_1})_{\mathsf{W}/}\]
is an isofibration by \cite[Tags~\href{https://kerodon.net/tag/01EX}{01EX} and \href{https://kerodon.net/tag/01F3}{01F3}]{kerodon}. Consequently, as at the outset of this section, \cite[Tag~\href{https://kerodon.net/tag/033T}{033T}]{kerodon} identifies the objects of $\mathsf{Q}_{\lambda}^{\mathrm{cn}}$ with pairs $(A^{\triangleright},-\otimes_{A^{\triangleright}}A)$, consisting of a $\lambda$-light condensed ring $A^{\triangleright}$ and a symmetric monoidal localization
\[-\otimes_{A^{\triangleright}}A\colon \mathsf{D}(A^{\triangleright})_{\geq 0}\to\mathsf{D}(A)_{\geq 0}\]
belonging to $\overline{\mathsf{Loc}}^{\otimes}_{\aleph_{1}}(\mathsf{W})$. In particular, it folows from Definition~\ref{def connective panring} -- rephrasing condition (ii) therein via \cite[Lem.~7.5]{reflectionprinciple} -- that $(A^\triangleright,\mathsf{D}(A)_{\geq 0})$ is an object of $\mathsf{PAnRing}^\mathrm{cn}_\lambda$.

The universal property of fiber products yields a functor $\mathsf{Q}^\mathrm{cn}_\lambda\to\mathsf{P}^\mathrm{cn}_\lambda$, induced by the canonical projection $\mathsf{Q}^\mathrm{cn}_\lambda\to\mathsf{Cond}_\lambda(\mathsf{Ring})$ and the composite
\[\overline{\mathsf{Loc}}^{\otimes}_{\kappa}(\mathsf{W})\to\mathsf{Fun}(\Delta^{1},\mathsf{Pr}^{\mathsf{L}}_{\kappa})\simeq\mathsf{Fun}(\Delta^{1},\mathsf{Pr}^{\mathsf{R}}_{\kappa})^{\mathrm{op}}\to\mathsf{Fun}(\Delta^{1},\widehat{\mathsf{Cat}})^{\mathrm{op}},\]
whose essential image is contained in $\mathsf{FF}^{\mathrm{op}}$. As above, the resulting functor takes values in $\mathsf{PAnRing}^\mathrm{cn}_\lambda$ and thus induces a functor
\[\Phi_\lambda\colon\mathsf{Q}^\mathrm{cn}_\lambda\to\mathsf{PAnRing}^\mathrm{cn}_\lambda,\]
which we claim is an equivalence.

To prove that $\Phi_\lambda$ is essentially surjective, let $(A^\triangleright,\mathsf{D}(A)_{\geq 0})$ be an object of $\mathsf{PAnRing}^\mathrm{cn}_\lambda$. As explained in the first step of the proof of Lemma~\ref{preanalytic rings are determined by connective halves}, the completion functor $-\otimes_{A^\triangleright}A$ equips $\mathsf{D}(A)_{\geq 0}$ with a symmetric monoidal structure making
\[-\otimes_{A^\triangleright}A\colon\mathsf{D}(A^\triangleright)_{\geq 0}\to\mathsf{D}(A)_{\geq 0}\]
a morphism in $\mathsf{CAlg}(\mathsf{Pr}^\mathsf{L}_{\aleph_1})$. Equipping $\mathsf{D}(A)_{\geq 0}$ with the induced symmetric monoidal functor from $\mathsf{W}$, we obtain an object $(A^\triangleright,-\otimes_{A^\triangleright}A)$ of $\mathsf{Q}^\mathrm{cn}_\lambda$ satisfying
\[\Phi_\lambda(A^\triangleright,-\otimes_{A^\triangleright}A)\simeq (A^\triangleright,\mathsf{D}(A)_{\geq 0}).\]

To prove that $\Phi_\lambda$ is fully faithful, observe that it is a functor over $\mathsf{Cond}_\lambda(\mathsf{Ring})$. It therefore suffices to show that, for every morphism $\varphi\colon A^\triangleright\to B^\triangleright$ of $\lambda$-light condensed rings, the induced map on fibers over $\varphi$ of
\[\operatorname{Map}_{\mathsf{Q}_{\lambda}^{\mathrm{cn}}}\mleft(\mleft(A^{\triangleright},-\otimes_{A^{\triangleright}}A\mright),\mleft(B^{\triangleright},-\otimes_{B^{\triangleright}}B\mright)\mright)\to\operatorname{Map}_{\mathsf{PAnRing}_\lambda^\mathrm{cn}}\mleft(\mleft(A^\triangleright,\mathsf{D}(A)_{\geq 0}\mright),\mleft(B^\triangleright,\mathsf{D}(B)_{\geq 0}\mright)\mright)\]
is an equivalence. By definition of $\mathsf{PAnRing}_\lambda^\mathrm{cn}$, the fiber on the right is contractible if $\varphi_*$ maps $\mathsf{D}(B)_{\geq 0}$ into $\mathsf{D}(A)_{\geq 0}$, and empty otherwise. Thus, it remains to establish the same description for the fiber on the left.

To establish the same description for the fiber on the left, consider the following cartesian square of mapping animae, induced by the defining fiber product of $\mathsf{Q}^\mathrm{cn}_\lambda$:
\[\begin{tikzcd}[ampersand replacement=\&]
	{\operatorname{Map}_{\mathsf{Q}_{\lambda}^{\mathrm{cn}}}\mleft(\mleft(A^{\triangleright},-\otimes_{A^{\triangleright}}A\mright),\mleft(B^{\triangleright},-\otimes_{B^{\triangleright}}B\mright)\mright)} \& {\operatorname{Map}_{\overline{\mathsf{Loc}}_{\aleph_{1}}^{\otimes}(\mathsf{W})}\mleft(-\otimes_{A^{\triangleright}}A,-\otimes_{B^{\triangleright}}B\mright)} \\
	{\operatorname{Map}_{\mathsf{Cond}_{\lambda}(\mathsf{Ring)}}(A^{\triangleright},B^{\triangleright})} \& {\operatorname{Map}_{\mathsf{CAlg}(\mathsf{Pr}^{\mathsf{L}}_{\aleph_{1}})_{\mathsf{W}/}}\mleft(\mathsf{D}(A^{\triangleright})_{\geq 0},\mathsf{D}(B^{\triangleright})_{\geq 0}\mright).}
	\arrow[from=1-1, to=1-2]
	\arrow["U"', from=1-1, to=2-1]
	\arrow["s", from=1-2, to=2-2]
	\arrow[from=2-1, to=2-2]
\end{tikzcd}\]
The pasting law gives an equivalence between the fiber of $U$ over $\varphi$ and the fiber of $s$ over $\varphi^*$. By Lemma~\ref{contractability of fibers}, the latter is contractible if $\varphi_*$ maps $\mathsf{D}(B)_{\geq 0}$ into $\mathsf{D}(A)_{\geq 0}$, and empty otherwise. This establishes the desired description and completes the proof.
\end{proof}

\begin{proof}[Proof of Theorem~\ref{PAnRing is presentable}]
Lemmas~\ref{preanalytic rings are determined by connective halves} and \ref{PAnRing as a pullback} identify $\mathsf{PAnRing}_\lambda$ with the pullback of the diagram
\[\mathsf{Cond}_\lambda(\mathsf{Ring})\xrightarrow{\ \mathsf{D}(-)_{\geq 0}\ }\mathsf{CAlg}(\mathsf{Pr}^\mathsf{L}_{\aleph_1})_{\mathsf{W}/}\xleftarrow{\ s\ }\overline{\mathsf{Loc}}^\otimes_{\aleph_1}(\mathsf{W}),\]
where $\mathsf{W}=\mathsf{Cond}_\lambda(\mathsf{Ab})$. By Proposition~\ref{D is a map in Pr_k}, the morphism on the left lies in $\mathsf{Pr}^\mathsf{L}_{\aleph_1}$. On the other hand, Proposition~\ref{CondAb is mu-compact} allows us to apply Corollary~\ref{main theorem on monoidal localizations2}, which shows that the morphism on the right lies in $\mathsf{Acc}_\mu$ and preserves $\aleph_1$-filtered colimits.

The argument of \cite[Tag~\href{https://kerodon.net/tag/06LM}{06LM}]{kerodon} therefore shows that $\mathsf{PAnRing}_\lambda$ is $\mu$-accessible. Finally, \cite[Lem.~4.1.10 and Prop.~4.1.11]{camargo2026notessolidgeometry} proves that $\mathsf{PAnRing}_{\aleph_1}$ admits small colimits. The same arguments apply verbatim for every uncountable cardinal $\lambda$, and hence $\mathsf{PAnRing}_\lambda$ admits small colimits in general.
\end{proof}

A convenient feature of presentable categories is that they are automatically complete \cite[Tag~\href{https://kerodon.net/tag/06PU}{06PU}]{kerodon}. We therefore obtain the following corollary, for which we are not aware of a direct proof.

\begin{corollary}\label{PAnRing is complete}
   The category $\mathsf{PAnRing}_\lambda$ is complete. 
\end{corollary}

In Clausen--Scholze's theory of analytic geometry, one typically works with \emph{analytic rings}, namely pre-analytic rings $A$ satisfying $A^\triangleright\in\mathsf{D}(A)$. In the light setting, where $\lambda=\aleph_1$, pre-analytic rings enjoy an additional important property: for every prime $p$, the Frobenius morphism
\[\phi_p\colon A^\triangleright\to A^\triangleright/p\]
underlies a morphism of pre-analytic rings $A\to A/p$, where $A/p$ carries the pre-analytic ring structure induced by $A$ \cite[Thm.~4.3.2]{camargo2026notessolidgeometry}. As it is unclear whether this property extends to arbitrary $\lambda$, we henceforth restrict to the light setting and suppress $\lambda$ from the notation.

\begin{definition}
Let $A$ be a light pre-analytic ring. We say that $A$ is an \emph{analytic ring} if the object $A^\triangleright$ of $\mathsf{D}(A^\triangleright)$ belongs to the full subcategory $\mathsf{D}(A)$. We denote by $\mathsf{AnRing}\subseteq\mathsf{PAnRing}$ the full subcategory spanned by the analytic rings.
\end{definition}

We are now in a position to state the main application of the machinery developed in Sections~\ref{sec2} and \ref{sec3}. As noted in the introduction, the following theorem was announced in \cite[Lecture~15]{ClausenScholzeAnalyticStacks} without proof and raised as an open point in \cite[\S1.2, Footnote~2]{camargo2026notessolidgeometry}.

\begin{theorem}\label{AnRing is presentable}
    The category $\mathsf{AnRing}$ is $\mu$-presentable, for $\mu=(2^{\aleph_0})^+$.
\end{theorem}

\begin{proof}
According to \cite[Thm.~4.1.13 and Cor.~4.2.7]{camargo2026notessolidgeometry}, the natural inclusion
\[\mathsf{AnRing}\hookrightarrow\mathsf{PAnRing}\]
admits a left adjoint. It follows from Corollary~\ref{PAnRing is complete} and \cite[Tag~\href{https://kerodon.net/tag/04JX}{04JX}]{kerodon} that $\mathsf{AnRing}$ is complete and that the inclusion preserves and reflects small limits. By \cite[Cor.~4.1.12]{camargo2026notessolidgeometry}, we also have that the category $\mathsf{AnRing}$ is stable under sifted, hence filtered, colimits in $\mathsf{PAnRing}$. The desired result then follows from Theorem~\ref{PAnRing is presentable} and Proposition~\ref{enhanced reflection theorem}.
\end{proof}

For convenience, we record the following consequence of the preceding proof.

\begin{corollary}\label{AnRing is complete}
   The category $\mathsf{AnRing}$ is complete. Moreover, the inclusion
   \[\mathsf{AnRing}\hookrightarrow\mathsf{PAnRing}\] 
   preserves and reflects small limits. 
\end{corollary}

\appendix
\section{Auxiliary results on presentable and accessible categories}
In this appendix, we collect some folklore results on presentable and accessible categories that we were unable to find in precisely the form needed here. Rather than merely citing analogous statements, we explain how the relevant arguments adapt to our setting.

We begin with a criterion for a presentable category to be $\kappa$-compactly generated, and therefore $\kappa$-presentable. The case $\kappa=\aleph_0$ is due to Mathew \cite[Lem.~6.22]{mathew2016galois}; the proof below follows his argument for an arbitrary regular cardinal $\kappa$.

\begin{proposition}\label{Mathews lemma}
    Let $\kappa$ be a regular cardinal and let $\mathsf{C}$ be a presentable category. Let $\mathsf{C}_0\subseteq \mathsf{C}^\kappa$ be a full subcategory stable under $\kappa$-small colimits in $\mathsf{C}$. We say that $\mathsf{C}_0$ \emph{detects equivalences} if a map $f\colon X\to Y$ in $\mathsf{C}$ is an equivalence whenever
    \[\operatorname{Map}_\mathsf{C}(Z,X)\xrightarrow{\ f\circ -\ } \operatorname{Map}_\mathsf{C}(Z,Y)\]
    is an equivalence for every $Z\in \mathsf{C}_0$. Then the induced functor
    \[\mathsf{Ind}_\kappa(\mathsf{C}_0)\to \mathsf{C}\]
    is an equivalence if and only if $\mathsf{C}_0$ detects equivalences.
\end{proposition}

\begin{proof}
Let $F\colon \mathsf{Ind}_\kappa(\mathsf{C}_0)\to \mathsf{C}$ be the induced functor, and identify $\mathsf{Ind}_\kappa(\mathsf{C}_0)$ with the full subcategory of $\mathsf{P}(\mathsf{C}_0)$ spanned by the functors $\mathsf{C}_0^\mathrm{op}\to \mathsf{Ani}$ that preserve $\kappa$-small limits \cite[Cor.~5.3.5.4]{lurie2009HTT}. Since $\mathsf{C}_0$ is stable under $\kappa$-small colimits in $\mathsf{C}$, the restricted Yoneda functor
\[\mathsf{C}\to \mathsf{P}(\mathsf{C}_0)\]
factors through $\mathsf{Ind}_\kappa(\mathsf{C}_0)$, and the resulting functor $\yo\colon\mathsf{C}\to\mathsf{Ind}_\kappa(\mathsf{C}_0)$ is right adjoint to $F$ by \cite[Tag~\href{https://kerodon.net/tag/03WN}{03WN}]{kerodon}.

The assumption that $\mathsf{C}_0$ detects equivalences is precisely the statement that $\yo$ is conservative. On the other hand, \cite[Prop.~5.3.5.11]{lurie2009HTT} implies that $F$ is fully faithful. It follows that $F$ is an equivalence, since a fully faithful left adjoint is an equivalence if and only if its right adjoint is conservative; this is the dual of \cite[Tag~\href{https://kerodon.net/tag/03UZ}{03UZ}]{kerodon}.
\end{proof}

We now apply this criterion to coslice categories. The next proposition extends \cite[Tags~\href{https://kerodon.net/tag/06HW}{06HW} and \href{https://kerodon.net/tag/06P6}{06P6}]{kerodon} by removing the compactness hypothesis on $S$. The proof again follows Mathew's argument in \cite[Lem.~6.22]{mathew2016galois}.

\begin{proposition}\label{lemma on coslices}
Let $\kappa$ be a regular cardinal, let $\mathsf{C}$ be a $\kappa$-presentable category, and let $S\in\mathsf{C}$. Then the coslice category $\mathsf{C}_{S/}$ is $\kappa$-presentable, and its subcategory of $\kappa$-compact objects is generated under $\kappa$-small colimits and retracts by the objects
\[(S\to S\sqcup X)\in \left(\mathsf{C}_{S/}\right)^\kappa,\]
for $X\in \mathsf{C}^\kappa$. Moreover, for every morphism $F\colon\mathsf{C}\to\mathsf{D}$ in $\mathsf{Pr}^{\mathsf{L}}_\kappa$, the induced functor
\[F_{S/}\colon \mathsf{C}_{S/}\to\mathsf{D}_{F(S)/}\]
is again a morphism in $\mathsf{Pr}^{\mathsf{L}}_\kappa$.
\end{proposition}

\begin{proof}
As $\mathsf{C}$ is $\kappa$-compactly generated, the full subcategory $\mathsf{C}^\kappa$ generates $\mathsf{C}$ under $\kappa$-filtered colimits. Hence the functors
\[\operatorname{Map}_{\mathsf{C}}(X,-),\qquad X\in\mathsf{C}^\kappa,\]
are jointly conservative. For every $X\in\mathsf{C}^\kappa$ and every object $S\to Y$ of $\mathsf{C}_{S/}$, the universal property of the coproduct gives a natural equivalence
\[\operatorname{Map}_{\mathsf{C}_{S/}}(S\sqcup X,Y)\simeq\operatorname{Map}_{\mathsf{C}}(X,Y).\]

Since the forgetful functor $\mathsf{C}_{S/}\to\mathsf{C}$ preserves $\kappa$-filtered colimits, this equivalence implies that $S\to S\sqcup X$ is $\kappa$-compact in $\mathsf{C}_{S/}$ \cite[Tags~\href{https://kerodon.net/tag/0236}{0236}, \href{https://kerodon.net/tag/02KT}{02KT}, and \href{https://kerodon.net/tag/02PH}{02PH}]{kerodon}. The conservativity of the projection $\mathsf{C}_{S/}\to\mathsf{C}$ then gives that the functors
\[\operatorname{Map}_{\mathsf{C}_{S/}}(S\sqcup X,-),
\qquad X\in\mathsf{C}^\kappa,\]
are jointly conservative.

Denote by $\mathsf{C}_0$ the full subcategory of $\mathsf{C}_{S/}$ generated under $\kappa$-small colimits by the maps $S\to S\sqcup X$ with $X$ in $\mathsf{C}^\kappa$. Note that $\mathsf{C}_{S/}$ is presentable and $\mathsf{C}_0$ is contained in $(\mathsf{C}_{S/})^\kappa$, by \cite[Tags~\href{https://kerodon.net/tag/06P6}{06P6} and \href{https://kerodon.net/tag/064X}{064X}]{kerodon}. Proposition~\ref{Mathews lemma} then implies that the natural functor
\[\mathsf{Ind}_\kappa(\mathsf{C}_0)\to \mathsf{C}_{S/}\]
is an equivalence. It follows that $\mathsf{C}_{S/}$ is $\kappa$-presentable and the description of $(\mathsf{C}_{S/})^\kappa$ follows from the fact that $\mathsf{Ind}_\kappa(\mathsf{C}_0)^\kappa$ is the idempotent completion of $\mathsf{C}_0$ \cite[Lem.~5.4.2.4]{lurie2009HTT}.

Finally, let $F\colon\mathsf{C}\to\mathsf{D}$ be a morphism in $\mathsf{Pr}^{\mathsf{L}}_\kappa$. By the dual of \cite[Tag~\href{https://kerodon.net/tag/02KD}{02KD}]{kerodon}, the induced functor $F_{S/}$ admits a right adjoint. Moreover, $F_{S/}$ preserves $\kappa$-compact objects: on the generators described above, it is given by
\[\mleft[S\to S\sqcup X\mright]\mapsto\mleft[F(S)\to F(S)\sqcup F(X)\mright],\]
which is $\kappa$-compact since $F(X)$ is $\kappa$-compact.
\end{proof}

We next consider coslice categories in the accessible setting. Assuming that $S$ is $\kappa$-compact gives the following analogue, with a simpler description of the $\kappa$-compact objects.

\begin{proposition}\label{lemma on accessible coslices}
Let $\kappa$ be a regular cardinal, let $\mathsf{C}$ be a $\kappa$-accessible category, and let $S\in\mathsf{C}$ be $\kappa$-compact. Then the coslice category $\mathsf{C}_{S/}$ is $\kappa$-accessible, and an object $S\to X$ of $\mathsf{C}_{S/}$ is $\kappa$-compact if and only if $X$ is $\kappa$-compact in $\mathsf{C}$. Moreover, for every morphism $F\colon \mathsf{C}\to \mathsf{D}$ in $\mathsf{Acc}_\kappa$, the induced functor
\[F_{S/}\colon \mathsf{C}_{S/}\to \mathsf{D}_{F(S)/}\]
is again a morphism in $\mathsf{Acc}_\kappa$.
\end{proposition}

\begin{proof}
The first statement follows from \cite[Tags~\href{https://kerodon.net/tag/06HW}{06HW} and \href{https://kerodon.net/tag/06LJ}{06LJ}]{kerodon}. The characterization of the $\kappa$-compact objects of $\mathsf{C}_{S/}$ immediately implies that $F_{S/}$ preserves $\kappa$-compact objects. Finally, \cite[Tags~\href{https://kerodon.net/tag/0236}{0236}, \href{https://kerodon.net/tag/02KT}{02KT}, and \href{https://kerodon.net/tag/02PH}{02PH}]{kerodon} shows that $F_{S/}$ preserves $\kappa$-filtered colimits.
\end{proof}

We next record a simple consequence of the reflection theorem of Ragimov--Schlank \cite[Thm.~1.1]{reflectionprinciple}. Since we will use it repeatedly, it is convenient to state it explicitly.

\begin{proposition}\label{enhanced reflection theorem}
Let $\kappa$ be a regular cardinal, let $\mathsf{C}$ be a $\kappa$-presentable category, and let $\mathsf{C}_0\subseteq\mathsf{C}$ be a full subcategory closed under small limits and $\kappa$-filtered colimits in $\mathsf{C}$. Then $\mathsf{C}_0$ is $\kappa$-presentable, and the inclusion $\mathsf{C}_0\hookrightarrow\mathsf{C}$ admits a left adjoint
\[L\colon\mathsf{C}\to\mathsf{C}_0\]
which is a morphism in $\mathsf{Pr}^{\mathsf{L}}_\kappa$.

Suppose in addition that $\mathsf{C}$ is a commutative algebra object in $\mathsf{Pr}^{\mathsf{L}}_\kappa$ and that
\[\underline{\operatorname{Hom}}(X,A)\in\mathsf{C}_0\qquad\text{for all }X\in\mathsf{C}\text{ and }A\in\mathsf{C}_0,\]
where $\underline{\operatorname{Hom}}$ denotes the internal Hom in $\mathsf{C}$. Then $\mathsf{C}_0$ admits a unique symmetric monoidal structure for which $L$ is symmetric monoidal. With respect to this structure, $\mathsf{C}_0$ is a commutative algebra object in $\mathsf{Pr}^{\mathsf{L}}_\kappa$, and $L\colon\mathsf{C}\to\mathsf{C}_0$ is a morphism in $\mathsf{CAlg}(\mathsf{Pr}^{\mathsf{L}}_\kappa)$.
\end{proposition}

\begin{proof}
    By the reflection theorem \cite[Thm.~1.1]{reflectionprinciple}, the category $\mathsf{C}_0$ is presentable and the inclusion $i\colon \mathsf{C}_0\hookrightarrow\mathsf{C}$ admits a left adjoint $L$. Since $i$ preserves $\kappa$-filtered colimits, the adjunction equivalence
\[\operatorname{Map}_{\mathsf{C}_0}(L(X),Y)\simeq\operatorname{Map}_{\mathsf{C}}(X,i(Y))\]
implies that $L(X)$ is $\kappa$-compact for every $\kappa$-compact object $X\in\mathsf{C}$.

For every $A\in\mathsf{C}_0$, choose a presentation $i(A)\simeq\operatornamewithlimits{colim}_{j\in\mathsf{J}}X_j$ as a small $\kappa$-filtered colimit of objects $X_j\in\mathsf{C}^\kappa$. Applying $L$, we obtain
\[A\simeq L(i(A))\simeq\operatornamewithlimits{colim}_{j\in\mathsf{J}}L(X_j).\]
Since $\mathsf{C}^\kappa$ is essentially small and each $L(X_j)$ is $\kappa$-compact in $\mathsf{C}_0$, the objects $L(X)$, for $X\in\mathsf{C}^\kappa$, form an essentially small collection of $\kappa$-compact objects generating $\mathsf{C}_0$ under $\kappa$-filtered colimits. Hence $\mathsf{C}_0$ is $\kappa$-presentable, and $L$ is a morphism in $\mathsf{Pr}^{\mathsf{L}}_\kappa$. Moreover, if $A$ is $\kappa$-compact, its identity factors through one of the objects $L(X_j)$ in the preceding presentation. Hence every $\kappa$-compact object of $\mathsf{C}_0$ is a retract of $L(X)$ for some $X\in\mathsf{C}^\kappa$; see also \cite[Cor.~5.5.7.3]{lurie2009HTT}.

Now assume the additional hypotheses. By \cite[Lem.~7.5]{reflectionprinciple}, the condition on internal Homs is equivalent to $L$ being \emph{compatible with the symmetric monoidal structure} in the sense of \cite[Def.~2.2.1.6]{higheralgebra}. It therefore follows from \cite[Prop.~2.2.1.9]{higheralgebra} that $\mathsf{C}_0$ inherits a symmetric monoidal structure for which $L$ is symmetric monoidal. Writing $\otimes_0$ for the resulting tensor product, we have
\[A\otimes_0 B\simeq L(i(A)\otimes i(B)), \qquad 1_{\mathsf{C}_0}\simeq L(1_{\mathsf{C}}).\]
The condition on internal Homs then implies that $\underline{\operatorname{Hom}}(i(-),i(-))$ is an internal Hom in $\mathsf{C}_0$. Thus $\otimes_0$ preserves small colimits separately in each variable.

Finally, the unit $L(1_{\mathsf{C}})$ is $\kappa$-compact. If $A,B\in\mathsf{C}_0^\kappa$, choose $X,Y\in\mathsf{C}^\kappa$ such that $A$ and $B$ are retracts of $L(X)$ and $L(Y)$, respectively. Then $A\otimes_0 B$ is a retract of
\[L(X)\otimes_0 L(Y)\simeq L(X\otimes Y),\]
which is $\kappa$-compact since $X\otimes Y$ is $\kappa$-compact. It follows that $\mathsf{C}_0$ is a commutative algebra in $\mathsf{Pr}^{\mathsf{L}}_\kappa$. Since $L$ is symmetric monoidal and preserves $\kappa$-compact objects, it is a morphism in $\mathsf{CAlg}(\mathsf{Pr}^{\mathsf{L}}_\kappa)$, as desired.
\end{proof}

We now turn to commutative algebra objects. The following is a commutative-algebra version of standard monadic compactness results. The ordinary-categorical argument appears in \cite[Thm.~2.18]{low2016heart}. For the higher-categorical free-algebra and bar-resolution arguments, compare \cite[Lems.~6.15--6.16]{chromaticnullstellensatz}; for the stable associative analogue, compare \cite[Prop.~3.2(2)]{efimovKtheory}. We include a proof in the generality needed here.

\begin{proposition}\label{forgetful functor preserves compacts}
Let $\kappa\geq \aleph_1$ be a regular cardinal and let $\mathsf{C}$ be a commutative algebra in $\mathsf{Pr}^\mathsf{L}_\kappa$. Then the forgetful functor
\[U\colon \mathsf{CAlg}(\mathsf{C})\to \mathsf{C}\]
preserves and reflects $\kappa$-compact objects.
\end{proposition}

\begin{proof}
Let $F\dashv U$ denote the free–forgetful adjunction, and put $T\colonequals U\circ F$. We first observe that both $F$ and $T$ preserve $\kappa$-compact objects. For $F$, this follows from the fact that its right adjoint $U$ preserves sifted colimits \cite[Cor.~3.2.3.2]{higheralgebra}. As for $T$, we use the free-algebra formula
\[T(X)\simeq \coprod_{n\geq 0}\operatorname{Sym}^n(X)\simeq \coprod_{n\geq 0}(X^{\otimes n})_{S_n}\]
of \cite[Ex.~3.1.3.14]{higheralgebra}. If $X$ is $\kappa$-compact, then, since $\mathsf{C}$ is an object of $\mathsf{CAlg}(\mathsf{Pr}^\mathsf{L}_\kappa)$, each tensor power $X^{\otimes n}$ is $\kappa$-compact, including $X^{\otimes 0}=1$. Moreover, each classifying anima $\mathsf{B}S_n$ is $\aleph_1$-small; hence \cite[Tag~\href{https://kerodon.net/tag/064X}{064X}]{kerodon} first gives that each orbit $(X^{\otimes n})_{S_n}$ is $\kappa$-compact and then, since the coproduct is countable, that $T(X)$ is $\kappa$-compact.

Now let $A$ be a commutative algebra in $\mathsf{C}$ such that $U(A)$ is $\kappa$-compact. The standard augmented bar resolution has terms $B_n(A)\simeq F(T^n(U(A)))$, and its augmentation induces an equivalence
\[\operatornamewithlimits{colim}_{[n]\in \Delta^\mathrm{op}}B_n(A)\xrightarrow{\ \sim\ }A;\]
see \cite[Prop.~4.7.3.14 and Lem.~3.2.2.6]{higheralgebra}. By the preceding paragraph, every $B_n(A)$ is $\kappa$-compact, so another application of \cite[Tag~\href{https://kerodon.net/tag/064X}{064X}]{kerodon} shows that $A$ is $\kappa$-compact as well.

For the converse, let $\mathsf{K}\subseteq \mathsf{CAlg}(\mathsf{C})$ be the full subcategory spanned by the algebras whose underlying objects are $\kappa$-compact. We have just shown that
\[\mathsf{K}\subseteq \mathsf{CAlg}(\mathsf{C})^\kappa.\]
In order to prove equality, we first claim that $\mathsf{K}$ is stable under $\kappa$-small colimits. Let $D\colon \mathsf{I}\to \mathsf{K}\subseteq \mathsf{CAlg}(\mathsf{C})$ be a $\kappa$-small diagram. Applying the bar resolution objectwise and taking the colimit over $\mathsf{I}$ gives a simplicial commutative algebra
\[S_\ast\colonequals \operatornamewithlimits{colim}_{i\in \mathsf{I}} B_\ast(D_i)\in \mathsf{Fun}(\Delta^\mathrm{op},\mathsf{CAlg}(\mathsf{C})).\]

By the commutation of colimits with colimits, the geometric realization of $S_\ast$ is $\mleft|S_\ast\mright|\simeq \operatornamewithlimits{colim}_{i\in \mathsf{I}}D_i$. Moreover, since $F$ preserves small colimits, we have that
\[S_n\simeq \operatornamewithlimits{colim}_{i\in \mathsf{I}} F(T^n(U(D_i)))\simeq F(X_n), \quad \text{for}\quad X_n\colonequals \operatornamewithlimits{colim}_{i\in \mathsf{I}}T^n(U(D_i)).\]
The objects $U(D_i)$ are $\kappa$-compact, and $T$ preserves $\kappa$-compact objects. Since $\kappa$-compact objects are stable under $\kappa$-small colimits, each $X_n$ is therefore $\kappa$-compact. Hence so is
\[U(S_n)\simeq T(X_n).\]

The forgetful functor $U$ preserves sifted colimits, and in particular geometric realizations. It follows that
\[U\mleft(\operatornamewithlimits{colim}_{i\in \mathsf{I}}D_i\mright)\simeq U\mleft(\operatornamewithlimits{colim}_{[n]\in \Delta^\mathrm{op}}S_n\mright)\simeq \operatornamewithlimits{colim}_{[n]\in \Delta^\mathrm{op}} U(S_n).\]
Once again, this is a $\kappa$-small colimit of $\kappa$-compact objects, and is thus $\kappa$-compact. This establishes the claim.

Next, using the nomenclature of Proposition~\ref{Mathews lemma}, we claim that $\mathsf{K}$ detects equivalences in $\mathsf{CAlg}(\mathsf{C})$. Let $\varphi\colon A\to B$ be a morphism in $\mathsf{CAlg}(\mathsf{C})$ such that
\[\operatorname{Map}_{\mathsf{CAlg}(\mathsf{C})}(C,A)\xrightarrow{\ \varphi\circ -\ } \operatorname{Map}_{\mathsf{CAlg}(\mathsf{C})}(C,B)\]
is an equivalence for every $C\in \mathsf{K}$. For each $X\in \mathsf{C}^\kappa$, the free algebra $F(X)$ belongs to $\mathsf{K}$, as $U(F(X))=T(X)$ is $\kappa$-compact. The adjunction $F\dashv U$ therefore shows that
\[\operatorname{Map}_\mathsf{C}(X,U(A))\xrightarrow{\ U(\varphi)\circ -\ } \operatorname{Map}_\mathsf{C}(X,U(B))\]
is an equivalence for all $X\in \mathsf{C}^\kappa$.

Since $\mathsf{C}^\kappa$ detects equivalences in $\mathsf{C}$, the map $U(\varphi)$ is an equivalence, and hence so is $\varphi$, because $U$ is conservative. Proposition~\ref{Mathews lemma} now gives an equivalence
\[\mathsf{Ind}_\kappa(\mathsf{K})\xrightarrow{\ \sim \ } \mathsf{CAlg}(\mathsf{C}).\]
Finally, $\mathsf{K}$ is stable under retracts: indeed, $U$ preserves retracts and $\kappa$-compact objects of $\mathsf{C}$ are stable under retracts \cite[Tag~\href{https://kerodon.net/tag/064Z}{064Z}]{kerodon}. Consequently, \cite[Lem.~5.4.2.4]{lurie2009HTT} identifies the $\kappa$-compact objects of $\mathsf{Ind}_\kappa(\mathsf{K})$ with the objects of $\mathsf{K}$. Under the preceding equivalence, this gives
\[\mathsf{CAlg}(\mathsf{C})^\kappa=\mathsf{K},\]
which proves that $U$ preserves $\kappa$-compact objects.
\end{proof}

We conclude with a compactness result for localizations of presheaf categories, viewed as objects of $\mathsf{Pr}^{\mathsf{L}}_\kappa$. This result will be used to prove Proposition~\ref{CondAb is mu-compact}.

\begin{proposition}\label{compactness of local objects}
Let $\mu\geq \kappa$ be uncountable regular cardinals, let $\mathsf{C}$ be a $\mu$-small category, and let $W$ be a $\mu$-small set of morphisms between $\kappa$-compact objects of
\[\mathsf{P}(\mathsf{C})=\mathsf{Fun}(\mathsf{C}^{\mathrm{op}},\mathsf{Ani}).\]
Then the full subcategory of $\mathsf{P}(\mathsf{C})$ spanned by the $W$-local objects is a $\mu$-compact object of $\mathsf{Pr}^\mathsf{L}_\kappa$.
\end{proposition}

\begin{proof}
Let $\mathsf{Eq}(\mathsf{Ani})\subseteq \mathsf{Fun}(\Delta^1,\mathsf{Ani})$ be the full subcategory spanned by the equivalences. For each morphism $w\colon X\to Y$ in $W$, consider the functor $T_w\colon \mathsf{P}(\mathsf{C})\to \mathsf{Fun}(\Delta^1,\mathsf{Ani})$ sending $Z\in\mathsf{P}(\mathsf{C})$ to the morphism
\[\operatorname{Map}_{\mathsf{P}(\mathsf{C})}(Y,Z)\to \operatorname{Map}_{\mathsf{P}(\mathsf{C})}(X,Z).\]
The full subcategory $\mathsf{P}_0\subseteq\mathsf{P}(\mathsf{C})$ defined by the pullback
\[\begin{tikzcd}[ampersand replacement=\&]
	{\mathsf{P}_0} \& {\displaystyle\prod_{w\in W}\mathsf{Eq}(\mathsf{Ani})} \\
	{\mathsf{P}(\mathsf{C})} \& {\displaystyle\prod_{w\in W}\mathsf{Fun}(\Delta^1,\mathsf{Ani})}
	\arrow[from=1-1, to=1-2]
	\arrow[hook, from=1-1, to=2-1]
	\arrow[hook, from=1-2, to=2-2]
	\arrow[from=2-1, to=2-2]
\end{tikzcd}\]
is therefore precisely the full subcategory spanned by the $W$-local objects.

Since the maps in $W$ are between $\kappa$-compact objects, the functors $T_w$ lie in $\mathsf{Pr}^\mathsf{R}_\kappa$, and so the preceding square is a pullback in $\mathsf{Pr}^\mathsf{R}_\kappa$. Indeed, $\mathsf{Pr}^\mathsf{R}_\kappa$ admits small limits and the inclusion
\[\mathsf{Pr}^\mathsf{R}_\kappa\hookrightarrow\widehat{\mathsf{Cat}}\]
preserves them \cite[Prop.~5.5.7.6]{lurie2009HTT}. Passing to left adjoints therefore yields a pushout square
\[\begin{tikzcd}[ampersand replacement=\&]
	{\displaystyle\coprod_{w\in W}\mathsf{Fun}(\Delta^1,\mathsf{Ani})} \& {\mathsf{P}(\mathsf{C})} \\
	{\displaystyle\coprod_{w\in W}\mathsf{Eq}(\mathsf{Ani})} \& {\mathsf{P}_0}
	\arrow[from=1-1, to=1-2]
	\arrow[from=1-1, to=2-1]
	\arrow[from=1-2, to=2-2]
	\arrow[from=2-1, to=2-2]
\end{tikzcd}\]
in $\mathsf{Pr}^\mathsf{L}_\kappa$.

The same argument as in the proof of Lemma~\ref{Q is nice} shows that $\mathsf{P}(\mathsf{C})$ is a $\mu$-compact object of $\mathsf{Pr}^\mathsf{L}_\kappa$. It also shows that $\mathsf{Fun}(\Delta^1,\mathsf{Ani})$ is $\mu$-compact, by taking $\mathsf{C}=(\Delta^1)^{\mathrm{op}}$, while $\mathsf{Eq}(\mathsf{Ani})\simeq\mathsf{Ani}$ is $\mu$-compact by taking $\mathsf{C}=\ast$. Since $W$ is $\mu$-small, \cite[Tag~\href{https://kerodon.net/tag/064X}{064X}]{kerodon} implies that all vertices of the pushout square are $\mu$-compact, as desired.
\end{proof}

\printbibliography

@unpublished{lurie2018spectral,
  shorthand={SAG},
  title={Spectral algebraic geometry},
  author={Lurie, Jacob},
  note={Available on \href{https://www.math.ias.edu/~lurie/}{the author's website}},
  year={2018}
}

@unpublished{higheralgebra,
  shorthand={HA},
  title={Higher Algebra},
  author={Lurie, Jacob},
  note={Available on \href{https://www.math.ias.edu/~lurie/}{the author's website}},
  year={2017}
}

@misc{haine2025fullyfaithfulfunctorspushouts,
      title={Fully faithful functors and pushouts of $\infty$-categories}, 
      author={Peter J. Haine and Maxime Ramzi and Jan Steinebrunner},
      year={2025},
      eprint={2503.03916},
      archivePrefix={arXiv},
      primaryClass={math.CT},
}

@misc{aoki2025higherpresentablecategorieslimits,
      title={Higher presentable categories and limits}, 
      author={Ko Aoki},
      year={2025},
      eprint={2510.13503},
      archivePrefix={arXiv},
      primaryClass={math.CT}
}

@misc{kerodon,
  shorthand = {Ker},
  author       = {Jacob Lurie},
  title        = {Kerodon},
  howpublished = {Available at \url{https://kerodon.net}},
  year         = {2025},
}

@unpublished{heyer20246functorformalismssmoothrepresentations,
      title={6-Functor Formalisms and Smooth Representations}, 
      author={Claudius Heyer and Lucas Mann},
      year={2024},
      eprint={2410.13038},
      archivePrefix={arXiv},
      primaryClass={math.CT},
}

@unpublished{camargo2026notessolidgeometry,
      title={Notes on Solid Geometry}, 
      author={Rodríguez Camargo, Juan Esteban},
      year={2026},
      eprint={2603.03012},
      archivePrefix={arXiv},
      primaryClass={math.AG},
}

@unpublished{standardtstructures,
      title={Standard $t$-structures}, 
      author={Peter J. Haine and Mauro Porta and Jean-Baptiste Teyssier},
      year={2025},
      eprint={2504.07473},
      archivePrefix={arXiv},
      primaryClass={math.CT},
}

@online{ClausenScholzeAnalyticStacks,
  author       = {Clausen, Dustin and Scholze, Peter},
  title        = {Analytic Stacks},
  note={Available at \href{https://www.youtube.com/playlist?list=PLx5f8IelFRgGmu6gmL-Kf_Rl_6Mm7juZO}{the IHES YouTube channel}},
  date = {2023}
}

@book{lurie2009HTT,
  shorthand={HTT},
  title={\href{https://www.math.ias.edu/~lurie/papers/HTT.pdf}{Higher Topos Theory}},
  author={Lurie, Jacob},
  series={Annals of Mathematics Studies},
  year={2009},
  volume={170},
  publisher={Princeton University Press},
  %ISBN={978-0-691-14048-3}
}

@book{borceux1994HCA2,
  title={\href{https://doi.org/10.1017/CBO9780511525865}{Handbook of Categorical Algebra 2: Categories and Structures}},
  author={Borceux, Francis},
  series={Encyclopedia of Mathematics and its Applications},
  year={1994},
  volume={51},
  publisher={Cambridge University Press},
  %ISBN={978-0-521-44179-7}
}

@article{chromaticnullstellensatz,
    author = {Burklund, Robert and Schlank, Tomer M. and Yuan, Allen},
     TITLE = {\href{https://annals.math.princeton.edu/articles/22772}{The chromatic Nullstellensatz}},
   JOURNAL = {Ann. of Math.},
  FJOURNAL = {Annals of Mathematics},
  year = {2022},
      NOTE = {To appear. \href{https://arxiv.org/abs/2207.09929}{arXiv:2207.09929}},
}

@article{efimovKtheory,
    author = {Efimov, Alexander I.},
     TITLE = {\href{https://arxiv.org/abs/2405.12169}{$K$-theory and localizing invariants of large categories}},
   JOURNAL = {Publ. Math. Inst. Hautes \'Etudes Sci.},
  FJOURNAL = {Publications math\'ematiques de l'IH\'ES},
  year = {2024},
      NOTE = {To appear. \href{https://arxiv.org/abs/2405.12169}{arXiv:2405.12169}},
}

@article{reflectionprinciple,
    AUTHOR = {Ragimov, Shaul and Schlank, Tomer},
     TITLE = {\href{https://doi.org/10.1112/topo.70060}{The $\infty$-categorical reflection theorem and applications}},
   JOURNAL = {J.~Topol},
      YEAR = {2026},
      volume = {19},
    NUMBER = {1},
}

@article{basechange,
 author = {Haine, Peter J.},
 title = {\href{https://doi.org/10.1016/j.jpaa.2025.107993}{From nonabelian basechange to basechange with coefficients}},
 journal = {J. Pure Appl. Algebra},
 volume = {229},
 number = {7},
 pages = {30},
 year = {2025}
}

@article{aaron,
    author = {Mazel-Gee, Aaron},
    title = {\href{https://doi.org/10.2140/agt.2019.19.3217}{The universality of the Rezk nerve}},
journal = {Algebr. Geom. Topol.},
volume = {19},
number = {7},
pages = {3217--3260},
year = {2019}
}

@article{mathew2016galois,
  title={\href{https://doi.org/10.1016/j.aim.2015.12.017}{The Galois group of a stable homotopy theory}},
  author={Mathew, Akhil},
  journal={Adv. Math.},
  volume={291},
  pages={403--541},
  year={2016}
}

@article{AokiTTG,
   title={\href{http://dx.doi.org/10.1007/s00209-023-03210-z}{Tensor triangular geometry of filtered objects and sheaves}},
   volume={303},
   number={3},
   journal={Math. Z.},
   author={Aoki, Ko},
   year={2023}
}

@article{low2016heart,
  author  = {Low, Zhen Lin},
  title   = {\href{https://doi.org/10.70930/tac/6vu6kq36}{The heart of a combinatorial model category}},
  journal = {Theory Appl. Categ.},
  volume  = {31},
  number  = {2},
  pages   = {31--62},
  year    = {2016},
}

@article{CartierCyclotomic,
 author = {Antieau, Benjamin and Nikolaus, Thomas},
 title = {\href{https://doi.org/10.1090/jams/951}{Cartier modules and cyclotomic spectra}},
 fjournal = {Journal of the American Mathematical Society},
 journal = {J. Am. Math. Soc.},
 volume = {34},
 number = {1},
 pages = {1--78},
 year = {2021}
}

@article{monads,
    author = {Simon Henry and Nicholas J. Meadows},
    title = {\href{https://dx.doi.org/10.21136/HS.2025.05}{Higher theories and monads}},
    journal = {High. Struct.},
    year = {2025},
    pages = {227--268},
    volume={9},
    number={1},
}

@incollection{Monk1989,
  author    = {Monk, James Donald},
  title     = {\href{https://euclid.colorado.edu/~monkd/monk48.pdf}{The Number of {B}oolean Algebras}},
  booktitle = {Handbook of Boolean Algebras},
  editor    = {Monk, James Donald and Bonnet, Robert},
  volume    = {2},
  pages     = {469--489},
  publisher = {North-Holland},
  address   = {Amsterdam},
  year      = {1989}
}

@misc{montagnani2025axiomaticapproachanalytic1affineness,
      title={An axiomatic approach to analytic $1$-affineness}, 
      author={Matteo Montagnani and Emanuele Pavia},
      year={2025},
      eprint={2509.04341},
      archivePrefix={arXiv},
      primaryClass={math.AG},
      %url={https://arxiv.org/abs/2509.04341}, 
}

@article{CLL,
    author = {Bastiaan Cnossen and Tobias Lenz and Sil Linskens},
    title = {\href{https://doi.org/10.1112/topo.70044}{Parametrized stability and the universal property of global spectra}},
    journal = {J. Topol.},
    year = {2025},
    volume = {18},
    number = {4},
}

@phdthesis{Kipp,
    author = {Kipp, Niklas},
    title = {\href{https://doi.org/10.5283/epub.78438}{A six-functor formalism for syntomic cohomology of $p$-adic formal schemes}},
    school = {Universit{\"a}t Regensburg},
    year = {2026}
}

\textsc{Wern Juin Gabriel Ong, University of Bonn, Bonn 53115, Germany}

\textit{Email address:} \href{mailto:wgabrielong@uni-bonn.de}{wgabrielong@uni-bonn.de}

\medskip

\textsc{Gabriel Ribeiro, Department of Mathematics, ETH Zurich, 8092 Zurich, Switzerland}

\textit{Email address:} \href{mailto:gabriel.ribeiro@math.ethz.ch}{gabriel.ribeiro@math.ethz.ch}

\end{document}